\documentclass[11pt]{article}
\usepackage{amsmath, amssymb, theorem, latexsym, epsfig}
\numberwithin{equation}{section}

\theoremstyle{plain}
\theorembodyfont{\itshape}
\newtheorem{theorem}{Theorem}[section]
\newtheorem{proposition}[theorem]{Proposition}
\newtheorem{lemma}[theorem]{Lemma}
\newtheorem{corollary}[theorem]{Corollary}
\newtheorem{convention}[theorem]{Convention}
                                                                                
\theorembodyfont{\rmfamily}
\newtheorem{definition}[theorem]{Definition}

\newtheorem{example}[theorem]{Example}
\newtheorem{remark}[theorem]{Remark}
\newenvironment{proof}{{\noindent \textbf{Proof}\,\,}}{\hspace*{\fill}$\Box$\medskip}

\title{On 4-reflective complex analytic planar billiards}
\author{Alexey Glutsyuk
\thanks{ CNRS, France (UMR 5669 (UMPA, ENS de Lyon) and UMI 2615 (Lab. J.-V.Poncelet)). 
Permanent address:  Unit\'e de Math\'ematiques 
Pures et Appliqu\'ees, M.R., \'Ecole Normale Sup\'erieure de Lyon,
46 all\'ee d'Italie, 69364 Lyon 07, France.  \newline Email:
aglutsyu@ens-lyon.fr}
\thanks{National Research University Higher School of Economics (HSE), Moscow, Russia}
 \thanks{Supported by part by RFBR grants 
10-01-00739-a, 13-01-00969-a 
and NTsNIL\_a (RFBR-CNRS)  10-01-93115, by ANR grant ANR-13-JS01-0010.}}
\begin{document}
\maketitle
\begin{abstract} The famous conjecture of V.Ya.Ivrii \cite{Ivrii} says that 
{\it in every billiard with infinitely-smooth boundary in a Euclidean space 
the set of periodic orbits has measure zero}. In the present paper we study its  
complex analytic version for quadrilateral 
orbits in two dimensions, with reflections from  holomorphic curves. 
 We present the complete classification of 4-reflective  complex analytic 
counterexamples: billiards formed by four holomorphic curves in the projective plane that have open set of 
quadrilateral orbits. This extends the  previous 
author's result \cite{alg} classifying 4-reflective complex planar algebraic counterexamples. We provide applications to real 
planar billiards: classification of 4-reflective germs of real planar $C^4$-smooth pseudo-billiards; solutions of 
Tabachnikov's Commuting Billiard Conjecture and the 4-reflective 
case of Plakhov's Invisibility Conjecture (both in two dimensions;  the boundary is required to be piecewise $C^4$-smooth). 
%In particular, we retrieve the solution of the particular case of 
%Ivrii's Conjecture for quadrilateral orbits in planar billiards \cite{gk1, gk2}. 
We provide a survey and a small technical result concerning higher number of complex reflections. 
 \end{abstract}
\tableofcontents
\def\cc{\mathbb C}
\def\oc{\overline{\cc}}
\def\cp{\mathbb{CP}}
\def\wt#1{\widetilde#1}
\def\rr{\mathbb R}
\def\var{\varepsilon}
\def\tt{\mathcal T}
\def\var{\varepsilon}
\def\mce{\mathcal E}
\def\mcf{\mathcal F}
\def\mcfnab{\mcf_{ab}|_{N_{ab}}}
\def\ha{\hat a}
\def\hb{\hat b}
\def\hc{\hat c}
\def\hd{\hat d}
\def\nn{\mathbb N}
\def\mcd{\mathcal D}
\def\rp{\mathbb{RP}}
\def\mct{\mathcal T}
\def\zz{\mathbb Z}

\section{Introduction}

The famous V.Ya.Ivrii's conjecture \cite{Ivrii} 
 says that {\it in every billiard with infinitely-smooth 
boundary in a Euclidean space of any dimension the set of periodic orbits 
has measure zero.} As it was shown by V.Ya.Ivrii \cite{Ivrii}, his conjecture implies the famous 
H.Weyl's conjecture  on the two-term 
asymptotics of the spectrum of Laplacian \cite{HWeyl11}. A brief historical survey of 
both conjectures with references is presented in \cite{gk1,gk2}. 

For the proof of Ivrii's conjecture it suffices to show that for every $k\in\mathbb N$ 
the set ot $k$-periodic orbits has measure zero. For $k=3$ this 
was proved in  \cite{bzh, rychlik, stojanov, vorobets, W}. For $k=4$ in dimension two this was proved in 
\cite{gk1, gk2}. 

\begin{remark}
Ivrii's conjecture is open 
already for billiards with piecewise-analytic boundaries, and we believe that this is its principal case. 
In the latter case Ivrii's conjecture is equivalent to the statement saying that for every 
$k\in\mathbb N$ the set of $k$-periodic orbits has empty interior.
\end{remark}

Besides the traditional real billiards, where each ray hitting the boundary is reflected back to the same side, it is interesting to study 
the so-called pseudo-billiards (introduced in Section 5), where some reflections change the side. Pseudo-billiards naturally arise, e.g., in 
the invisibility theory. One can ask the following question analogous to Ivrii's Conjecture: 
{\it classify those pseudo-billiards that have an open (positive measure) set of periodic orbits.} 
This question is closely related, e.g.,  to Plakhov's Invisibility Conjecture \cite[conjecture 8.2]{pl} and Tabachnikov's Commuting Billiard Conjecture 
\cite[p. 58]{tabcom}, which is related to the famous Birkhoff Conjecture on integrable billiards \cite[p. 95]{tab}. 

It appears that planar Ivrii's conjecture and all its analogues for pseudo-billiards have the same complexification stated and partially studied 
in  \cite{alg, odd} and 
recalled below. This is the problem  to classify all the 
so-called {\it $k$-reflective complex planar analytic billiards:} those collections 
of $k$ complex analytic curves  in $\cp^2$ for which the corresponding billiard 
has an open set of $k$-periodic orbits.  Its studying presents 
 a unifying approach to the original Ivrii's conjecture and all its above-mentioned  analogues altogether. 

In the present paper we solve the complex classification problem completely for $k=4$ 
(Theorem \ref{an-class}, the main result).  As an application, we provide  
the complete classification of germs of $C^4$-smooth real planar 
pseudo-billiards having open set of quadrilateral orbits (Subsections 5.1, 5.2).  
As applications of the latter result, we give solutions of Tabachnikov's Commuting Billiard Conjecture and the 4-reflective Plakhov's Invisibility Conjecture, both in two-dimensional piecewise $C^4$-smooth case (Subsections 5.3, 5.4). 

The classification of analytic 4-reflective germs of pseudo-billiards follows almost immediately from the main complex result. The proof of the 
classification of smooth pseudo-billiards is also done by complex methods and is based on 
the theory of Cartan's prolongations of Pfaffian systems. It is analogous to Yu.G.Kudryashov's arguments from \cite[section 2]{gk2} 
reducing the piecewise $C^4$-smooth case of 4-reflective Ivrii's Conjecture to the piecewise-analytic case. 

The state of art of the classification problem of $k$-reflective complex billiards in the general case is presented in Section 6 together with 
small new technical results (Theorem \ref{kmer} and Corollary \ref{ckhmer}). 

Basic definitions and statement of main result are given below.

\subsection{Main result: classification of 4-reflective complex analytic planar billiards}
 To recall the complexified Ivrii's conjecture and state the main result, let us recall some basic 
definitions contained in \cite[section 1]{alg}. 
The complex plane 
$\cc^2$ with affine coordinates $(z_1,z_2)$ is equipped with the complexified Euclidean metric. It is the standard complex-bilinear 
quadratic form $dz_1^2+dz_2^2$. This defines the notion of symmetry with respect to a complex line, 
reflections with respect to complex lines and more 
generally, reflections of complex lines with respect to complex analytic (algebraic) curves. 
The symmetry is defined by the same formula, as in the real case. More details concerning the complex reflection law 
are given  in  Subsection 2.3. 

\begin{remark} The geometry of the complexified Euclidean metric is 
somewhat similar to that of its another real form: the pseudo-Euclidean metric. 
Billiards in pseudo-Euclidean spaces were studied, e.g.,  in \cite{drag, khes}.
\end{remark}

\begin{definition}  A complex projective line $l\subset\cp^2\supset\cc^2$ is {\it isotropic}, 
if either it  coincides with the infinity line, or the complexified Euclidean quadratic 
form vanishes on $l$. Or equivalently, a  line is isotropic, if it passes through some 
of two points with homogeneous coordinates $(1:\pm i:0)$: the so-called {\it isotropic 
points at infinity} (also known as {\it cyclic} (or {\it circular}) points). 
\end{definition}

\begin{convention} \label{conv-curves} Everywhere below by 
an {\bf irreducible analytic curve} in $\cp^n$ we 
mean a non-constant $\cp^n$- valued holomorphic function on a connected Riemann surface. 
\end{convention}

\begin{definition} \label{deforb} \cite[definition 1.3]{alg}  A {\it complex analytic (algebraic) planar billiard} is a finite collection 
of  irreducible complex analytic (algebraic) 
curves $a_1,\dots,a_k\subset\cp^2$ that are not isotropic lines;  set $a_{k+1}=a_1$, $a_0=a_k$. 
A {\it $k$-periodic billiard orbit} is a collection of points $A_j\in a_j$,  $A_{k+1}=A_1$, $A_0=A_k$, such that  for every $j=1,\dots,k$ one has 
$A_{j+1}\neq A_j$, the tangent line $T_{A_j}a_j$ is not isotropic 
and the complex lines $A_{j-1}A_j$ and $A_jA_{j+1}$ are 
symmetric with respect to the line $T_{A_j}a_j$ and are distinct from it. (Properly saying, we have to take points $A_j$ together with 
prescribed branches of curves $a_j$ at $A_j$: this specifies the  line $T_{A_j}a_j$ in unique way, if $A_j$ is a self-intersection point 
of the curve $a_j$.) 
 \end{definition}

\begin{definition} \label{defref} \cite[definition 1.4]{alg}  A complex analytic (algebraic)  billiard $a_1,\dots,a_k$ is {\it $k$-reflective,} if 
it has an open set of $k$-periodic orbits. In more detail, this means that there exists 
an open set of pairs $(A_1,A_2)\in a_1\times a_2$ extendable to $k$-periodic 
orbits $A_1\dots A_k$. (Then the latter property automatically holds for every other 
pair of neighbor mirrors $a_j$, $a_{j+1}$.) 
\end{definition}

{\bf Problem (Complexified  version of Ivrii's conjecture).}  
{\it Classify all the  $k$-reflective complex analytic (algebraic)  billiards.}

% Let us first recall the definition of complex billiard and its trajectories. 
% The complexified planar Euclidean metric is  the standard complex-bilinear quadratic form on $\cc^2$. It 
%defines reflections with respect to complex lines and more 
%generally, reflections of complex lines with respect to complex algebraic curves. 
%We deal with {\it complex analytic billiards:} finite collections of complex analytic  
%curves $a_1,\dots,a_k$. A {\it billiard trajectory} is a collection of points $A_j\in a_j$ 
%such that each two neighbor  complex lines $A_{j-1}A_j$ and $A_jA_{j+1}$ are 
%symmetric with respect to the tangent line to the curve $a_j$ at $A_j$.  
%See the next subsection for detailed definition of complex reflection law. 
%
%\begin{definition} A complex analytic billiard $a_1,\dots,a_k$ 
%is {\it $k$- reflective,} if the set of initial pairs $(A_1,A_2)\in a_1\times a_2$ 
%extendable to $k$- periodic trajectories contains an open subset in $a_1\times a_2$. 
%Or in other words, if the set of $k$-periodic orbits has a non-empty interior. 
%\end{definition}
%
%The {\it complexified Ivrii's conjecture} asks: whether is it true that for every 
%complex analytic billiard the set of its periodic orbits has empty interior, in other words, 
%for every $k$ there are no $k$-reflective billiards? 
%
%As it was shown by the author in \cite{alg}, the complexified Ivrii's conjecture is 
%wrong for $k=4$. The algebraic counterexamples for $k=4$ 
%were classified in the same paper. 
%
%In the present paper we classify analytic counterexamples for $k=4$. 

\begin{theorem} \label{an-class} A complex planar analytic
billiard $a$, $b$, $c$, $d$ is 4-reflective, if and only if 
it has one of the three following types:

1) one of the mirrors, say $a$ is a line,   $c=a$, the 
curves $b$ and $d$ are symmetric with respect to the line $a$ and distinct from it, see Section 5,  Fig.\ref{fig-sym}; 
%\cite[fig.1]{alg}; 

2) the mirrors are distinct lines through the same point $O\in\cp^2$, the pair of lines $(a,b)$ is transformed to $(d,c)$ 
by complex rotation around $O$, i.e.,  a complex isometry $\cc^2\to\cc^2$ fixing $O$ with unit Jacobian, see Section 5,  Fig.\ref{fig-lines};
%\cite[fig.2]{alg}; 

3)  $a=c$, $b=d$,  and they are distinct confocal conics, see Section 5, Fig.\ref{fig-ellipses}--\ref{fig-par}. 
% \cite[theorem 1.10 and fig.4]{alg}. 
\end{theorem}

\begin{remark} \label{remalg} Theorem \ref{an-class} in the algebraic case is given by \cite[theorem 1.11]{alg}, which implies 
the 4-reflectivity of billiards of types 2) and 3). The proof of 4-reflectivity of billiards of type 1) repeats the proof in the algebraic case, see  \cite[example 1.7]{alg}.\end{remark}

%In Section 5 we give applications of Theorem \ref{an-class} to real billiards 
%and invisibility. 

\subsection{The plan of the proof of Theorem \ref{an-class}}

 Theorem \ref{an-class} is obviously 
implied by the two following theorems.

\begin{theorem} \label{onealg} Every 4-reflective complex planar  analytic billiard with at least one
algebraic  mirror has one of the above types  1)--3). 
\end{theorem}
%
%\begin{remark} Theorem \ref{tline} in the algebraic case was proved in \cite{alg}. 
%\end{remark}

\begin{theorem} \label{tallalg} 
Let in a complex planar analytic 4-reflective billiard no mirror be a line. Then 
all the mirrors are algebraic curves. 
\end{theorem}
%
%\begin{theorem} \label{talg} \cite{alg} 
%Every 4-reflective billiard with algebraic mirrors is one 
%of the above types 1) and 2). 
%\end{theorem}

Theorems \ref{onealg} and \ref{tallalg} are proved in Subsection 3.4 and Section 4 respectively. 

\begin{remark} \label{remirgerm} 
Theorem \ref{an-class} is local and can be stated  for a {\it germ of 
4-reflective analytic billiard:} a collection of irreducible germs of analytic curves 
$(a,A)$, $(b,B)$, $(c,C)$, $(d,D)$ in $\cp^2$ such that the 
quadrilateral $ABCD$ lies in an open set of quadrilateral  orbits of the corresponding billiard. 
\end{remark}

For the  proof of Theorem \ref{an-class} we study the maximal analytic extensions of the mirrors. These are analytic 
curves parametrized by abstract connected Riemann surfaces, which we will denote by 
$\ha$, $\hb$, $\hc$, $\hd$. The latter   
are called the {\it maximal normalizations,} see the corresponding background material in Subsection 2.2. We represent the 
 open set of quadrilateral orbits as a subset in $\ha\times\hb\times\hc\times\hd$ and will denote it  by $U_0$. Its closure 
 $$U=\overline{U_0}\subset\ha\times\hb\times\hc\times\hd$$
 in the usual topology is an analytic subset with only two-dimensional irreducible components. It will be called the {\it 4-reflective set}, see  \cite[definition 2.13 and proposition 2.14]{alg}. The complement $U\setminus U_0$ consists of the so-called degenerate quadrilateral orbits: 
 quadrilaterals $ABCD$ satisfying the reflection law that have either a pair of coinciding neighbor 
 vertices, or a pair of coinciding adjacent edges, e.g.,  an edge tangent to a mirror through an adjacent vertex,  or an isotropic tangency vertex. 
 
 \def\mcrr{\mathcal R}
 \def\mcp{\mathcal P}
\def\pp{\mathbb P}
\def\mch{\mathcal H}
\def\mcp{\mathcal P}
 
 One of the main ideas of the proof of Theorem \ref{an-class} is similar to that from \cite{gk1, gk2, alg}: to study the degenerate orbit set 
 $U\setminus U_0$. This idea itself together with basic algebraic geometry allowed to treat the algebraic case in \cite{alg}. 
 One of the key facts used in the proof was properness (and hence, epimorphicity)  
 of the projection $U\to\ha\times\hb$ to the position of two neighbor vertices. 
 In the algebraic case the properness is automatic (follows from Remmert's Proper Mapping Theorem \cite[p.34]{griff}), 
 but in the general analytic case under consideration it isn't. We prove that the above projection is indeed proper in the analytic case. 
The most part of the proof of Theorem \ref{an-class}, and in particular, the proof of properness are based on studying restricted versions 
 of Birkhoff distribution, which was introduced in \cite{bzh}. All the Birkhoff distributions are briefly described below; 
 more details are given in Subsection 2.7. 
  
  \begin{definition} Let $M$ be an n-dimensional (real or complex) analytic manifold.
Let $\mcd$ be a $d$-dimensional analytic distribution on $M$, i.e., $\mcd(x)\subset T_xM$ is a $d$-dimensional
subspace for every $x\in M$ and the map $x\mapsto\mcd(x)$ is analytic. Let $l\leq d$. An $l$-dimensional 
surface $S\subset M$ is said to be an {\it integral surface} for the distribution $\mcd$, if $T_xS\subset\mcd(x)$ 
for every $x\in S$. 
\end{definition}

Consider the projectivization of the tangent bundle $T\cp^2$: 
 $$\mcp=\mathbb P(T\cp^2).$$
 It is the space of pairs $(A,L)$: $A\in\cp^2$, $L\subset T_A\cp^2$ is a one-dimensional subspace. 
 The space $\mcp$ is three-dimensional and it carries the standard two-dimensional contact distribution $\mch$: the plane  
$\mch(A,L)\subset T_{(A,L)}\mcp$ is the preimage of the line $L\subset T_A\cp^2$ under the derivative of the bundle projection 
$\mcp\to\cp^2$. The product  $\mcp^k$ carries the product distribution $\mch^k$. Let $\mcrr_{0,k}\subset\mcp^k$ denote the 
subset of  points $((A_1,L_1),\dots,(A_k,L_k))$ such that for every $j$ one has $A_j\in\cc^2\subset\cp^2$, 
$A_{j\pm1}\neq A_j$, the lines $A_jA_{j-1}$, 
$A_jA_{j+1}$ are  symmetric with respect to the  line $L_j$, and  the three latter lines are distinct and non-isotropic. The 
above product distribution induces the so-called {\it Birkhoff distribution} $\mcd^k$ on $\mcrr_{0,k}$, see \cite{bzh}. 
 It is well-known \cite{bzh} that for every 
analytic billiard $a_1,\dots,a_k$ {\it the natural lifting to $\mcp^k$ 
of any analytic family of its $k$-periodic orbits $A_1\dots A_k$ with $L_j=T_{A_j}a_j$ 
lies in $\mcrr_{0,k}$ and is tangent to Birkhoff distribution.}  In particular, 
if the billiard is $k$-reflective, then {\it the lifting to $\mcrr_{0,k}$ of an open set of its $k$-periodic orbits is an integral surface of 
Birkhoff distribution.} 

We will study the following restricted versions $\mcd_a$ and $\mcd_{ab}$ of Birkhoff distribution that correspond respectively to 
4-reflective billiards $a$, $b$, $c$, $d$ with  one given mirror $a$ (or two given mirrors $a$ and $b$). The products $\ha\times\mcp^3$ and 
$\ha\times\hb\times\mcp^2$ admit natural inclusions to $\mcp^4$ induced by parametrizations $\ha\to a$, $\hb\to b$. 
Let 
$M_a\subset\ha\times\mcp^3$, $M_{ab}\subset\ha\times\hb\times\mcp^2$ denote the  closures of the corresponding pullbacks of the set  
$\mcrr_{0,4}$. The distributions $\mcd_a$, $\mcd_{ab}$ are the 
pullbacks of the Birkhoff distribution $\mcd^4$ on $\mcrr_{0,4}$. They are 3- and 2-dimensional 
singular analytic distributions on $M_a$ and $M_{ab}$ in the sense of Subsection 2.6. 
For every billiard as above the 
natural lifting to $\ha\times\mcp^3$ ($\ha\times\hb\times\mcp^2$) of any open set of its quadrilateral orbits lies in $M_a$ ($M_{ab}$) and 
is an integral surface of the corresponding distribution $\mcd_a$ (respectively, $\mcd_{ab}$). 

%More detail on Birkhoff distributions are given in Subsection 2.6. 

 The proof of Theorem \ref{an-class} is split into the following steps.
 
 Step 1. Case of two neighbor algebraic mirrors. 
 In this case it is easy to show that all the mirrors are algebraic (Proposition \ref{twoalg} in Subsection 2.1). This together with 
 \cite[theorem 1.11]{alg} implies that the billiard under question is of one of the types 1)--3), see Remark \ref{remalg}. 
 
 From now on we consider that no two neighbor mirrors are algebraic. 
 
 Step 2. Preparatory description of the complement $U\setminus U_0$. In Subsection 2.4 we 
 study degenerate quadrilaterals $ABCD\in U\setminus U_0$ 
 with a pair of coinciding neighbor vertices, say $A=D$, analogously to the arguments from \cite[p.320]{gk2}. 
 Under mild additional assumptions, in particular, $B,C\neq A=D$, 
 we show that the other mirrors $b$ and $c$ are special curves called {\it triangular spirals centered at $A$.} 
 Namely, they are phase curves of algebraic line fields on 
 $\cp^2$: the so-called triangular line fields centered at $A$ 
 introduced in the same subsection (Proposition \ref{intspir}).   
 %In the case, when $A=B=C=D$, we show (under other mild assumptions) that $a=c$ (Proposition \ref{a=c}). 
 One of the key arguments used in the 
 proof of Theorem \ref{an-class} is Proposition \ref{pnonalg}, which says that every triangular spiral with at least two distinct centers is algebraic. 
 %In the proof of Proposition \ref{a=c} we use the results of \cite{alg}  recalled in Subsection 2.3. 
In Subsections 2.3 and 2.5 we recall the results of \cite[subsections 2.1, 2.2]{alg} on 
 partial description of degenerate quadrilaterals in $U\setminus U_0$ with either an isotropic tangency vertex, 
 or an edge tangent to a mirror through an adjacent vertex. 
 
 Step 3. Properness of the projection $U\to\ha\times\hb$ (Section 3, Corollary \ref{cepi}). 
 To prove it, we study the Birkhoff distribution $\mcd_{ab}$ and prove its 
 non-integrability in Subsections 3.1--3.3. Moreover, we show that  {\it the closure in $M_{ab}$ of the union of its  integral surfaces (if any) is a two-dimensional analytic subset in $M_{ab}$} (Lemma \ref{lnonint} and Corollary \ref{cnonint}.)  In the proof of the latter statement and in what follows 
 we use Proposition \ref{anint} from Subsection 2.6. It deals with  an $m$-dimensional singular analytic distribution, a given union of 
 $m$-dimensional integral  surfaces and the minimal analytic set $M$ containing the latter union. 
 Proposition \ref{anint} states that {\it the restriction to $M$ of the distribution is $m$-dimensional and integrable.} Proposition \ref{anint} is 
 a key tool for the whole paper. 
 
%In the proof of the latter statement we use the fact that  the distribution $\mcd_{ab}$, which is defined on $M_{ab}^0$, 
% extends up to a singular analytic distribution on $M_{ab}$ 
% (see Subsection 2.5, which contains preparatory material on singular 
% analytic distributions). 
% This implies that the minimal analytic subset containing a given union of integrable surfaces is invariant and the 
% distribution is integrable there (Proposition \ref{anint}). 
 The set $U$ is  identified with either the above two-dimensional analytic subset in $M_{ab}$, or a smaller analytic subset. 
 This together with Proper Mapping Theorem implies properness of the projection $U\to\ha\times\hb$ (Corollary \ref{cepi}). The proof of Lemma \ref{lnonint}  
 is done by contradiction. The contrary would imply the existence of at least three-dimensional invariant 
 irreducible analytic subset $M\subset M_{ab}$ 
 where the distribution $\mcd_{ab}$ is integrable. Then a complement $M^0\subset M$ to a smaller analytic subset 
 is saturated by open sets of quadrilateral orbits of  4-reflective billiards $a$, $b$, $c$, $d$
 with variable mirrors $c=c(x)$ and $d=d(x)$, $x\in M^0$. We treat separately two cases: 
 
 - some of the projections of the set $M$ to the space of triples $(A(x),B(x),G(x))$, $G=C,D$, is not bimeromorphic.
 
 - both latter projections are bimeromorphic. 
 
 The first case will be treated in Subsection 3.2. We show that there exist $x,y\in M^0$ projected to the same vertices 
 $A$, $B$, $D=D_0$ but with distinct tangent lines 
  $T_{D_0}d(x)\neq T_{D_0}d(y)$, $D_0$ being not a cusp\footnote{Everywhere in the paper by {\it cusp} we mean 
the singularity of an arbitrary  irreducible singular germ of analytic curve, not necessarily the one given by equation $x^2=y^3+\dots$ 
in appropriate coordinates.}
% The {\it degree} of a cusp is its intersection index with a generic line though the singularity; the 
%degree of a regular germ is one.} 
of the curves $d(x)$ and $d(y)$, the projection to $(A,B,D)$ being a local submersion at $x$, $y$. 
We then deduce that the billiard 
  $d(y)$, $d(x)$, $c(x)$, $c(y)$ is  4-reflective (as in \cite[proof of lemma 3.1]{alg}), 
 and the mirror $c(x)$ is a triangular spiral with center $D_0$ (Proposition \ref{intspir}, Step 2). 
 Then we slightly deform  $y$  with fixed vertices $A$ and $B$ to a point $y'$ 
 so that the corresponding mirror $d(y')$ intersects $d(x)$ at a point $D_1\neq D_0$. We get analogously that the curve $c(x)$ is a triangular 
 spiral with two distinct centers $D_0$ and $D_1$. 
 This implies that $c(x)$ is algebraic (Proposition \ref{pnonalg}, Step 2). 
 Similarly, we show that $c(y)$ is algebraic, fixing $y$ and deforming $x$. 
 Hence, the mirror $d(x)$ of the 4-reflective billiard $d(y)$, $d(x)$, $c(x)$, $c(y)$ is algebraic, as are $c(x)$ and $c(y)$ 
 (Proposition \ref{twoalg}, Step 1). Similarly, $a$ and $b$ are algebraic, as are $c(x)$ and $d(x)$. The contradiction thus 
 obtained  implies that the first case is impossible.
 
 In the second case we show (in Subsection 3.3) that for an open set of points $x\in M^0$ the mirrors $c(x)$ and $d(x)$ are lines. Hence, the 
 curves $a$ and $b$ are algebraic, by Step 1, -- a contradiction. Finally, none of the above cases is possible. The contradiction thus obtained 
 will prove Lemma \ref{lnonint}. 
% 
% - either there exists an open set of triples $(A,B,D)$, each of them  
% realized by two distinct points $x,y\in M^0$ with transversely intersected mirrors $
% We show that either $c$ and $d$ are lines for an open set of points $x\in M^0$, or 
% there exists an 
% open set of quadrilaterals $ABCD$ 
% We fix a billiard $a$, $b$, $c$, $d$ and its quadrilateral orbit $ABCD$. For every other close billiard 
% $a$, $b$, $c'$, $d'$ from the above family the billiard $c$, $c'$, $d'$, $d$ is 4-reflective. Our goal is to 
% show that one can choose the billiard $a$, $b$, $c'$, $d'$ so that in addition, $d'$ intersects $d$ 
%  at a point close to $D$ that 
% depends nontrivially on $c'$ and $d'$. We show that this is the case, provided that $c$ and $d$ are not lines.  
% This together with results of Step 2 applied to the billiard $c$, $c'$, $d'$, $d$ 
% implies that $c$ and $c'$ are triangular spirals with variable 
% centers. Hence, $c$ is  algebraic by Proposition \ref{pnonalg} (Step 2). Analogously, $c'$ is algebraic. 
% Hence, $d$ is also algebraic, by Step 1 applied to the same 
% billiard. Thus, $c$ and $d$ are algebraic, and hence, so are $a$ and $b$, by Step 1 applied to the billiard $a$, $b$, $c$, $d$. 
 
 Step 4. Case of one algebraic mirror, say $a$: proof of  Theorem \ref{onealg} (Subsection 3.4). Properness of the projection $U\to\ha\times\hb$ (Step 3) implies properness of the projection  $U\to\hb$ (algebraicity). Therefore, the preimage in $U$ of every point $B\in \hb$ is a compact 
 holomorphic curve. This immediately implies that the mirror $c$ is algebraic and there are two possibilities:
 
 -  either all the mirrors are algebraic, and we are done; 
 
 - or the projection of the above preimage to the position of the point $D$ is constant for every $B$. 
 
 In the latter case we show that $a=c$ is a line and the mirrors $b,d\neq a$ are symmetric with respect to it: the billiard has type 1). This will 
 prove Theorem \ref{onealg}. 
 
 From now on we consider that no mirror is algebraic. We show that this case is impossible. This will prove Theorem \ref{tallalg} and hence, 
 Theorem \ref{an-class}. 
 
 Step 5.  Case of intersected mirrors, say $a$ and $b$ intersect at a point $A$. Under the additional assumption that $A$ is regular and 
  not an isotropic tangency point for both $a$ and $b$ we show that $a=c$ (Corollary \ref{inters} proved in 
  Subsection 3.5). 
 The set $U$ contains a non-empty at most one-dimensional 
  compact analytic subset of quadrilaterals $AACD$ (properness of projection, Step 3). For the proof of Corollary \ref{inters} we 
  show in Subsection 3.5 that this is a discrete subset in 
  $U$ consisting  of quadrilaterals with all the vertices coinciding with $A$. 
  Indeed, otherwise, if the above subset were one-dimensional, this would immediately imply that 
  some of the mirrors $c$ or $d$ is algebraic,  -- a contradiction.   
%  In what follows we study the Birkhoff distribution $\mcd_a$ using the above results together with the involutivity theory of Pfaffian systems 
%  and a version of Cauchy--Kovalevskaya theorem. 
%Recall that it is a three-dimensional singular analytic distribution on a 6-dimensional analytic set $M_a$. We fix a 
%connected component of  the open set of 
%quadrilateral orbits of the billiard $a$, $b$, $c$, $d$. It is an integral surface, which we will denote $S$. 
  
  Step 6.  Proof of Theorem \ref{tallalg}. To do this, we study the three-dimensional Birkhoff distribution $\mcd_a$ on the 6-dimensional 
  analytic set $M_a$. We fix a connected component of  the open set of quadrilateral orbits of the billiard $a$, $b$, $c$, $d$. 
  It is an integral surface of the distribution $\mcd_a$, which we will denote $S$. 
   We consider the minimal analytic subset $M\subset M_a$ containing $S$, which is   
irreducible and three-dimensional (easy to show), and study the restriction $\mcd_M$ to $M$ of the distribution $\mcd_a$. 
%at least three-dimensional: otherwise it is two-dimensional, its fibers over points $A\in\ha$ are compact 
%  holomorphic curves and the mirror $b$ is then obviously algebraic, -- a contradiction. 
%  We consider  of the 
%  distribution $\mcd_a$ and study it using Cartan--Kuranishi--Rashevsky involutivity theory of Pfaffian systems. The corresponding background 
%  material is recalled in Subsection 4.1. 
We   treat two separate cases: 1) $dim \mcd_M=2$ (Subsection 4.1); 2) $dim \mcd_M=3$ (Subsection 4.3). 
%the distribution $\mcd_M$ is  either 
%  two-dimensional, or three-dimensional non-involutive (Subsection 4.2); 2) the distribution $\mcd_M$ is three-dimensional involutive 
%  (Subsection 4.3).  The first case will be basically reduced to the two-dimensional case: we show that $S$ is always tangent to 
%  a (single-valued or double-valued) singular integrable two-dimensional analytic distribution contained in $\mcd_M$, the integral plane distribution. 
In the two-dimensional Case 1) we show that there exists an open subset $V\subset M$ saturated by integral surfaces of the distribution $\mcd_M$ 
that correspond to 4-reflective billiards $a$, $b(x)$, $c(x)$, $d(x)$ with $b(x)$ intersecting $a$ (easily follows from transcendence of the 
curve $a$). We then deduce 
%% We consider its integral surfaces through points $x\in M$ that correspond to 4-reflective billiards $a$, $b(x)$, $c(x)$, $d(x)$ with $b(x)$ intersecting $a$: their existence 
% easily follows from definition and the transcendence of the curve $a$. Moreover, the latter 
% integral surfaces saturate an open subset $V\subset M$. We show 
that either the mirror $b(x)$ is a line for all $x\in V$ (and hence, for 
 all $x$ regular for both $M$ and $\mcd_M$), or the mirror $c(x)$ coincides 
 with $a$ for all $x$ as above. This basically follows from  Corollary \ref{inters}, Step 5. The first subcase is impossible, since then 
 the mirror $b$ of the initial transcendental billiard would be a line, -- a contradiction.   In the second subcase the projection $\nu_C(M)$ 
 of the whole variety $M$ to the position of the vertex $C$ 
 lies in  $a$. For a generic $A\in\ha$ we consider its preimage 
 $W_A\subset M$ under the projection $\nu_a:M\to\ha$, which is a projective 
 algebraic variety. It follows that the projection $\nu_C(W_A)$ lies in a transcendental curve  $a$, while it 
 should be an algebraic subset in $\cp^2$ (Remmert's Proper Mapping and Chow's  Theorems). Hence, $\nu_C(W_A)$ is discrete. On the other 
 hand, it cannot be discrete, whenever $b$ is neither a line, nor a conic, by \cite[proposition 2.32]{alg}. 
 The contradiction thus obtained  shows that  Case 1) is impossible. The three-dimensional Case 2) is treated analogously, but it is more 
 technical. The existence of two-dimensional integral surfaces as above of a three-dimensional distribution $\mcd_M$ is not automatic. 
 Its proof is based on Cartan--Kuranishi--Rashevskii involutivity theory of Pfaffian systems. The corresponding background material is recalled in 
 Subsection 4.2.

\section{Preliminaries}

\subsection{ Case of two neighbor algebraic mirrors}

\begin{proposition} \label{twoalg}  Let in a 4-reflective billiard $a$, $b$, $c$, $d$ the mirrors $a$ and $b$ be algebraic curves. 
Then all the mirrors are algebraic.
\end{proposition}

\begin{proof} By symmetry, it suffices to prove algebraicity of the mirror $c$. Fix a quadrilateral orbit $A_0B_0C_0D_0$. 
Consider the family of quadrilateral orbits $ABCD$ with fixed $D=D_0$. 
They are locally parametrized by the line $l=AD$, which lies in the space $\cp^1$ 
of lines through $D$. The point $A$ depends algebraically on $l$, since 
$a$ is algebraic. Similarly, the line $AB$, and hence, the point $B$ depend algebraically on $l$, since $a$, $b$ are algebraic and $AB$ is 
symmetric to $l$ with respect to the line $T_Aa$. 
Analogously, the line $BC$, which  is symmetric to $AB$ with respect to the line $T_Bb$, depends algebraically on $l$. 
The line $DC$ also depends algebraically on $l$, being  the reflected image of the line $l$ with respect to the fixed line $T_Dd$. Finally, the 
variable intersection point $C=BC\cap DC$ should also depend algebraically on $l$. Hence, $c$ is algebraic. The proposition is proved. 
\end{proof}

\subsection{Maximal analytic extension}

Recall that a germ $(a,A)\subset\cp^n$ of analytic curve is {\it irreducible}, if it is the image of a germ of analytic mapping 
$(\cc,0)\to\cp^n$. 

\begin{definition} \label{order} \cite[definition 5]{odd} Consider two holomorphic mappings of connected Riemann surfaces 
$S_1$, $S_2$ with base points $s_1\in S_1$ and $s_2\in S_2$ 
to $\cp^n$, $f_j:S_j\to\cp^n$, $j=1,2$, $f_1(s_1)=f_2(s_2)$. We say that 
$f_1\leq f_2$, if there exists a holomorphic 
mapping $h:S_1\to S_2$, $h(s_1)=s_2$, such that $f_1=f_2\circ h$. This defines a partial order on the set of classes of Riemann surface 
mappings to $\cp^n$ up to conformal reparametrization respecting base points. 
\end{definition}

The following proposition is classical, see the proof, e.g., in \cite{odd}. 

\begin{proposition} \label{ext} \cite[proposition 2]{odd}. Every irreducible germ of analytic curve in $\cp^n$ has maximal analytic extension. In more detail, let  $(a,A)\subset\cp^n$ be an irreducible germ of analytic curve. There exists an abstract connected 
Riemann surface $\hat a$ with base 
 point $\hat A\in\hat a$ (which we will call the {\bf maximal normalization} of the germ $a$) 
 and a holomorphic mapping $\pi_a:\ha\to \cp^n$, $\pi_a(\hat A)=A$ with the following properties:
 
 -  the image of germ at $\hat A$ of the mapping $\pi_a$ 
 is contained in $a$; 
 
 -  $\pi_a$ is the maximal mapping with the above property in the sense of Definition \ref{order}.
 
 Moreover, the mapping $\pi_a$ is unique up to composition with conformal isomorphism of Riemann surfaces respecting base points. 
 \end{proposition}

 \begin{corollary} \label{clift} Let $M$ be a complex manifold, and let $f:M\to\cp^n$ be a non-constant holomorphic mapping. 
 Let $U\subset M$ be an irreducible analytic subset, and let  the restriction $f|_U$  have rank one on an open  subset. 
 Let $x\in U$ be a regular point, $A=f(x)$; then the image of the germ $f:(U,x)\to\cp^n$ is an irreducible germ $(a,A)$ of analytic curve. 
 Let $\pi_a:\hat a\to a$ be its maximal normalization. 
 Let $\hat U$ be the normalization of the analytic set $U$ (see \cite[p.78]{chirka}), $\pi_U:\hat U\to U$ be the natural projection (which 
is invertible on the regular part of $U$). 
 Then there exists a unique holomorphic lifting $F:\hat U\to\ha$  such that $f\circ\pi_U=\pi_a\circ F$.
 \end{corollary}
 
 \begin{proof} 
 For every point $y\in\hat U$ and any point $z\in \hat U$ close enough to $y$ there exists an analytic curve in $\hat U$ 
 through $y$ and $z$. This  together with Proposition \ref{ext} (applied to the latter curves) and Hartogs' and Osgood's Theorems 
 imply the corollary.
 \end{proof}
 
% The corollary follows from definition and the fact that for every
%
%\begin{definition} We say that two analytic curves 
%$\phi_1:S_1\to\cp^2$ and $\phi_2:S_2\to \cp^2$ 
%are {\it equivalent}, if there exist open subsets $U_1\subset S_1$, $U_2\subset S_2$ 
%and a conformal mapping $\psi:U_1\to U_2$ such that $\phi_1=\phi_2\circ\psi$. 
%We say that an analytic curve $\phi_1:S_1\to\cp^2$ {\it contains} another 
%analytic curve $\phi_2:S_2\to\cp^2$, if 
%\begin{proposition} Every analytic curve admits a maximal analytic extension. 
%%%%%END-CONVENTION

\subsection{Complex reflection law}

\def\mcl{\mathcal L}
The material presented in this subsection is contained in \cite[subsection 2.1]{alg}.
%, except for Corollary \ref{ccoinc}.

We fix an Euclidean metric on $\rr^2$ and consider its complexification: the 
complex-bilinear quadratic form $dz_1^2+dz_2^2$ on the complex affine plane $\cc^2\subset\cp^2$. 
We denote the infinity line in $\cp^2$ by $\oc_{\infty}=\cp^2\setminus\cc^2$.   

\begin{definition}  The {\it symmetry} $\cc^2\to\cc^2$ with respect to a non-isotropic 
complex line $L\subset\cp^2$  is the unique non-trivial complex-isometric involution 
fixing the points of the line $L$. It extends to a projective transformation of the ambient plane $\cp^2$. 
For every $x\in L$ it acts on the space $\cp^1$ of lines through $x$, and this action is called {\it symmetry at $x$}. 
If $L$ is an isotropic line through a finite point $x$, then a pair of  lines through $x$ is called symmetric with respect to $L$, if 
it is a limit of symmetric pairs of lines with respect to  non-isotropic lines converging to $L$. 
\end{definition}

\begin{lemma} \label{lim-refl} \cite[lemma 2.3]{alg} 
Let $L$ be an isotropic line through a finite point $x$. A pair of lines $(L_1,L_2)$ through $x$ is 
symmetric with respect to $L$, if and only if some of them coincides with $L$.  
\end{lemma} 

\begin{convention} \label{conv2} For every irreducible analytic curve $a\subset\cp^2$ and a point $A\in\ha$ the {\bf local branch} $a_A$ of the curve $a$ at $A$ is 
the germ of curve $\pi_a:(\ha,A)\to\cp^2$, which is contained in $a$. By $T_Aa$ we denote the 
tangent line to the local branch $a_A$ at $\pi_a(A)$. Sometimes 
we idendify a point (subset) in $a$ with its preimage in the normalization $\hat a$ and denote 
both subsets by the same symbol. In particular, given a subset in $\cp^2$, 
say a line $l$, we set $\hat a\cap l=\pi_a^{-1}(a\cap l)\subset\hat a$. 
 If $a,b\subset\cp^2$ are two irreducible analytic curves, and $A\in\hat a$, $B\in\hat b$, $\pi_a(A)\neq\pi_b(B)$, 
 then for simplicity we write $A\neq B$ and the line $\pi_a(A)\pi_b(B)$ will be referred to, as $AB$. 
\end{convention}
 
 \begin{definition} \label{def-law} 
 A triple of points $BAD\in(\cp^2)^3$ satisfies the {\it complex reflection law} with respect to a given line $L$ through $A$, 
 if one of the following statements holds:
 
 - either $B,D\neq A$, the line $L$ is non-isotropic and the lines $AB$, $AD$ are symmetric with respect to $L$; 
 
 - or $B,D\neq A$, the line $L$ is isotropic and some of the lines $AB$, $AD$ coincides with $L$;
 
 - or $A$ coincides with some of the points $B$ or $D$. 
 \end{definition}
  
  \def\Int{\operatorname{Int}}
  
\begin{definition} 
 Let $a_1,\dots,a_k\subset\cp^2$ be an analytic (algebraic) billiard, and let $\ha_1,\dots,\ha_k$ be the maximal normalizations of its 
 mirrors. 
Let $P_k\subset\ha_1\times\dots\times \ha_k$ denote the subset  corresponding to $k$-periodic billiard orbits. 
The set $P_k$ is contained in the subset $Q_k\subset\ha_1\times\dots\times\ha_k$ of (not necessarily periodic) $k$-orbits: the $k$-gons 
$A_1\dots A_k$ such that for every $2\leq j\leq k-1$ one has $A_j\neq A_{j\pm1}$, the line $T_{A_j}a_j$ is not isotropic  and 
the lines $A_jA_{j-1}$, $A_jA_{j+1}$ are symmetric with respect to it and distinct from it. 
Let $U_0=\Int(P_k)$ denote the interior of the subset $P_k\subset Q_k$. Set   
 $$U=\overline{U_0}\subset\ha_1\times\dots\ha_k: \text{ the closure is taken in the usual product topology.} $$
 The set $U$ will be called the  {\it $k$-reflective set}. 
 \end{definition}

 \begin{proposition} \label{comp-set} \cite[proposition 2.14]{alg}. 
 The  $k$-reflective  set $U$ is an analytic 
 (algebraic) subset in $\ha_1\times\dots\times\ha_k$. The billiard is $k$-reflective, if and only if the $k$-reflective set $U$ is non-empty; then 
 each its irreducible component is two-dimensional.  If the billiard is $k$-reflective, then 
 for every point 
 $A_1\dots A_k\in U$ each triple $A_{j-1}A_jA_{j+1}$ satisfies the complex reflection law from Definition \ref{def-law} 
 with respect to the line $T_{A_j}a_j$, and each projection $U\to\ha_j\times\ha_{j+1}$ is a submersion on an open dense subset in $U$. 
%(It is epimorphic, if the billiard is algebraic.)
\end{proposition}

{\bf Addendum.} {\it For every $k$-reflective billiard the latter projections $U\to\ha_j\times\ha_{j+1}$ are  local biholomorphisms on the set of those 
$k$-periodic orbits whose vertices are not cusps of the corresponding mirrors.} 

The addendum follows from definition. 

%\begin{corollary} \label{ccoinc} \cite[corollary 2.10]{odd} 
%Let $a_1,\dots, a_k$ be a $k$-reflective analytic billiard in $\cp^2$. Let $A_1\dots A_k\in\ha_1\times\dots\times\ha_k$ 
%be a point of the $k$-reflective set, and let $A_j=A_{j+1}$ for some $j$. Then we have one of the following possibilities:
%
%(i) $A_j$ is either a marked, or a double point;
%
%(ii) $a_1=\dots=a_k$, $A_1=\dots=A_k$; 
%
%(iii) up to cyclic mirror renaming, 
%there exists an $s<j$ such that $a_{s+1}=\dots=a_j$, $A_s\neq A_{s+1}=\dots= A_j$, and the line $A_sA_j$ coincides with $T_{A_j}a_j$. 
%\end{corollary}

\subsection{Triangular algebraic line fields and spirals}

Here we deal with a 4-reflective complex analytic billiard $a$, $b$, $c$, $d$ whose $4$-reflective set $U$ contains a quadrilateral $ABCD$ 
with coinciding vertices $A=D$. We show (Proposition \ref{intspir}) that under mild genericity assumptions (implying, e.g., that 
 $ABCD$ is not a single-point quadrilateral) 
either the mirrors $b$ and $c$ are conics, or they are  so-called triangular spirals centered at $A$: phase curves of algebraic line fields 
invariant under the rotations around $A$. 
We show (Proposition \ref{pnonalg}) that every triangular spiral with two distinct centers is algebraic. 

%In the case, when $ABCD$ is single-point, we show that under mild genericity assumptions one has $a=c$.  
%
%namely, this is the case, when there exists an 
%with intersected neighbor mirrors $a$ and $d$. We study  one-parameter analytic families $\Gamma$ of degenerate quadrilateral 
%orbits in a 4-reflective component with   neighbor vertices $A$ and $D$ coinciding with a given point of intersection $a\cap d$. 
%Complexifying arguments from \cite[p.320]{gk2}, we show in Proposition \ref{prosper} that if some of the vertices $B$ or $C$, say $B$ is non-constant along $\Gamma$, then the corresponding mirror $b$ 
%is a so-called triangular spiral centered at $A$: 
%a phase curve of an algebraic line field introduced below and called triangular line field centered at $A$.  In the case, when 
%all the mirrors intersect at the same point $A$, we show in Proposition \ref{degspir} that $a=c$ under very weak genericity assumptions. 
%In the proof of algebraicity of mirrors we use elementary 
%Proposition \ref{pnoalg} saying that a triangular spiral with two distinct centers is always algebraic.

To define triangular spirals and state the above-mentioned results, we introduce yet another restricted 
Birkhoff distribution on the space of ``framed 
triangles with fixed vertex''. Let us fix a point $A\in\cc^2$ (take it as the origin) and a non-trivial 
complex isometry $H\in SO(2,\cc)\setminus Id$ fixing $A$. Recall that $\mcp=\mathbb P(T\cp^2)$, and $\mch$ is the standard contact plane field on 
$\mcp$, see Subsection 1.2. Namely, for every $x=(B,L)\in\mcp$, where $B\in\cp^2$, $L\subset T_B\cp^2$ is a one-dimensional subspace, 
the plane $\mch(x)\subset T_x\mcp$ is the preimage of the line $L$ under the differential of the bundle projection $\mcp\to\cp^2$. 
Consider the product $\mcp^2$ equipped with the  four-dimensional algebraic distribution $\mch^2=\mch\oplus\mch$. Let 
$\mct_{A,H}\subset\cp^2\times\cp^2$ denote the subset of pairs $(B,C)$  such that $B,C\neq A$, $B\neq C$, the lines $AB$, $AC$ 
are distinct, non-isotropic and  $AC=H(AB)$. 
Let $M_{A,H}^0\subset\mcp^2$ denote the subset of those pairs $((B,L_B),(C,L_C))$, for which $(B,C)\in\mct_{A,H}$, the lines 
$L_B$, $L_C$ are non-isotropic, the lines $AB$, $BC$ are symmetric with respect to the line $L_B$; $AC$, $BC$ 
are symmetric with respect to the line $L_C$; $AB\neq L_B$, $AC\neq L_C$. 
%, and the lines $AB$, $AC$, $L_B$, $L_C$ are non-isotropic. 
Set 
$$M_{A,H}=\overline{M_{A,H}^0}\subset\mcp^2: \text{ the closure in the usual topology.}$$
This is a three-dimensional projective algebraic variety, and $M_{A,H}^0\subset M_{A,H}$ is its Zariski open and dense subset. 

\begin{proposition} The variety $M^0_{A,H}$ is smooth and transversal to the distribution $\mch^2$.
\end{proposition}

\begin{proof} The smoothness is obvious. The restriction $\nu: M^0_{A,H}\to \mct_{A,H}$ of the bundle projection $\mcp^2\to(\cp^2)^2$ 
  is a local diffeomorphism, by construction. 
For every $x=((B,L_B),(C,L_C))\in M^0_{A,H}$ the subspace $\mch^2(x)\subset T_x(\mcp^2)$ is the preimage of the direct sum 
$L_B\oplus L_C\subset T_{(B,C)}(\cp^2)^2$ under the differential of the bundle projection. Thus, it suffices to show that for every $(B,C)\in\mct_{A,H}$ 
the space $T_{(B,C)}\mct_{A,H}$ is transversal to $L_B\oplus L_C$. Here $L_B$, $L_C$ are arbitrary  lines such that 
$AB$ and $BC$ are symmetric with respect to the line $L_B$ and $AC$, $BC$ are symmetric with respect to the line $L_C$.

For every $(B,C)\in\mct_{A,H}$ finitely punctured lines $AB\times C$ and $B\times AC$ are contained in $\mct_{A,H}$, by definition. We 
identify $AB$ and $AC$ with the corresponding one-dimensional subspaces in $T_B\cp^2$ and $T_C\cp^2$ respectively. Thus, 
$AB\oplus AC\subset T_{(B,C)}\mct_{A,H}$ and $AB\oplus AC$ is transversal to $L_B\oplus L_C$ in $T_{(B,C)}(\cp^2\times\cp^2)$, 
since $AB\neq L_B$ and $AC\neq L_C$ by 
definition. This proves the proposition.  
\end{proof}

\begin{corollary} For all $x\in M^0_{A,H}$ the intersections 
$$\mcd_{A,H}(x)=\mch^2(x)\cap T_xM^0_{A,H}$$
are one-dimensional and form an algebraic line field on $M^0_{A,H}$.
\end{corollary}

\begin{proof} The transversal variety $M^0_{A,H}$ and distribution $\mch^2$ in the ambient six-dimensional space $\mcp^2$ have 
dimensions 3 and 4 respectively. Hence, the intersections of their tangent spaces are one-dimensional. The algebraicity of the line field $\mcd_{A,H}$ is obvious.
 \end{proof}

%The restriction to $M_{A,H}^0$ of the distribution $\mch^2$ is an algebraic distribution, which will be denoted $\mcd_{A,H}$.

The next proposition shows that the line field $\mcd_{A,H}$ has an algebraic first integral: an appropriate holomorphic branch of squared 
perimeter of triangle. To define the latter branch, let us introduce the following definition. 

\begin{definition} \label{concord} Let $A,B,C\in\cc^2$, $B\neq A,C$, and let the lines $AB$, $BC$ be non-isotropic. Let $L$ be a symmetry line of the pair of 
lines $AB$, $BC$. The symmetry $\sigma: AB\to BC$ with respect to the line $L$ is an isometry with respect to the complexified Euclidean metric. 
For every line $l=AB,BC$ the complex distance function $l\to\cc$, $x\mapsto|Bx|=dist(B,x)$ has two affine branches that differ by sign. 
Let us choose those affine distance functions on the lines $AB$, $BC$ for which $|B\sigma(x)|\equiv-|Bx|$. We then 
say that the distances $|BA|=|AB|$, $|BC|=|CB|$ calculated with respect to the above affine distance functions 
are {\it $L$-concordant.} 
\end{definition}

\begin{remark} The $L$-concordant distances are well-defined up to simultaneous change of sign. Their ratio $\frac{|AB|}{|BC|}$ is 
uniquely defined. 
\end{remark}
 
\begin{example} \label{escort} Let in the above conditions $A,B,C\in\rr^2$, and let $|AB|$, $|BC|$ be the Euclidean distances. 
Then $|AB|$, $|BC|$ are $L$-concordant, if $L$ is the exterior 
bisector of the angle $\angle ABC$. Otherwise $L$ is the interior bisector and $|AB|$, $-|BC|$ are $L$-concordant, see Figure 1. 
\end{example}

Take an arbitrary point $((B,L_B),(C,L_C))\in M_{A,H}^0$. We identify the lines $L_B$ and $L_C$ with the corresponding projective lines 
in $\cp^2$.  Let us normalize the distances $|AB|$, $|BC|$, $|CA|$ to make $|AB|$ and $|BC|$  $L_B$-concordant and $|BC|$, $|CA|$ 
$L_C$-concordant. The perimeter $P=|AB|+|BC|+|CA|$ thus constructed is well-defined up to sign, and its square $P^2$ is a well-defined 
holomorphic function on $M_{A,H}^0$, see Figure 1.  

\begin{figure}[ht]
  \begin{center}
 % \vspace{-0.3cm}
   \epsfig{file=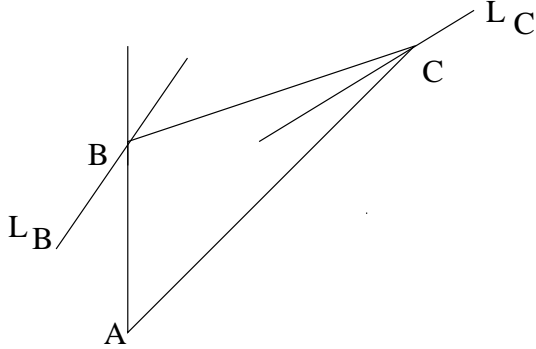}
 %  \vspace{-0.3cm}
    \caption{Concordant lengths in the real case: here $P=|AB|+|BC|-|AC|$.}
  \label{fig-length}
  \end{center}
\end{figure}
%
% Let us fix a holomorphic length function on $AB\times AB$  in an arbitrary way. Now let us  choose the holomorphic 
% length function on $BC\times BC$ so that $|BB'_1|=|BB'_2|$. This is possible, since $|BB'_1|^2=|BB'_2|^2$, by symmetry.  
% This definition is independent on the choice of the point $B'$. Given the length function on  $BC\times BC$ thus constructed, 
%we choose the length function on $AC\times AC$ analogously to the above construction, with $B'$ replaced by  $C'\in L_C$.  
%The perimeter  $|AB|+|BC|+|CA|$ thus constructed is well-defined up to sign, as is the initial length function on $AB\times AB$.

\begin{proposition} The above squared perimeter $P^2$ is a first integral of the line field $\mcd_{A,H}$. The space $M_{A,H}$, the line field 
$\mcd_{A,H}$ and $P^2$ are invariant with respect to the complex rotation group $SO(2,\cc)$ fixing $A$. 
\end{proposition} 
%zdes

\def\mct{\mathcal T}
\def\mctah{\mct_{A,H}}

\begin{proof} The  invariance follows from construction. The statement saying that $P^2$ is a  first integral is 
a complexification of the  classical 
statement on the  real perimeter and the real Birkhoff distribution, see, e.g., \cite[section 2]{bzh}. Its 
 proof is analogous to that in the real case. The proposition is proved. 
\end{proof}

\def\mcs{\mathcal S}

\begin{proposition} \label{prosper}  Let $P^2: M_{A,H}^0\to\cc$ be the squared perimeter function from the above proposition. Let 
$p\in\cc$, $\mcs_p$ be an irreducible component of the level set $\{ P^2=p\}$ in $M^0_{A,H}$. 
The projections $\nu_G:\mcs_p\to\cp^2$ to the position of  the vertex $G=B,C$ have discrete preimages, and thus, 
 are submersions on Zariski open 
dense subsets. The restriction to $\mcs_p$ of the line field $\mcd_{A,H}$ is 
sent by each projection to an $SO(2,\cc)$-invariant (multivalued) algebraic line field on $\cp^2$ depending on the choice of $G$ and 
called {\bf triangular line field centered at $A$ with 
parameters $H$, $p$.}  
%Thus, the projection of each complex orbit of the line field $\mcd_{A,H}$ to the position of any of the vertices $B$ or $C$ is a 
%phase curve of a corresponding triangular line field.
\end{proposition}

\begin{proof} The contrary to the discreteness of preimages  of the projection, say $\nu_B$ would imply constance of the perimeter on 
a one-parameter family of triangles $ABC$ with fixed 
vertices $A$ and $B$, fixed line $AC=H(AB)=L$  and variable $C\in L$. This is obviously impossible. The algebraicity and invariance 
of the projected line field obviously follow from  the algebraicity and invariance of the surface $\mcs_p$ and submersivity. 
\end{proof}

\begin{definition}
A {\it triangular spiral centered at} $A$ 
%with parameters  $(H,p)$ 
is a complex orbit of a triangular line field centered at $A$, see Fig.\ref{fig-spir}a).
\end{definition}

%\begin{remark} If $ABC\in\mctah^0$, then $AB\neq AC$. The quasi projective varieties $\mctah$ and $\wt\mctah$ are non-singular. 
%The line field $\mcd_{A,H}$ is nonsingular on $\wt\mctah^0$ and   invariant under the action of group 
%$SO(2,\cc)$ of complex rotations: the isometries fixing $A$. The corresponding triangular line field is also $SO(2,\cc)$-invariant. 
%\end{remark}
%
%The key arguments in the proof of the algebraicity of mirrors are the following  propositions.

\begin{proposition} \label{intspir}\footnote{Real triangular spirals were introduced in \cite[p.320]{gk2}, where a real 
version of Proposition \ref{intspir} was proved. Our proof of Proposition \ref{intspir}  is analogous to arguments from loc. cit.}  Let $(a,A)$, $(b,B)$, $(c,C)$, $(d,D)$ be germs of analytic curves in $\cp^2$ forming a 
 4-reflective analytic planar billiard:  
the quadrilateral $ABCD$ is contained in the 4-reflective set $U$, cf. Remark \ref{remirgerm}. Let 
the mirror germs $a$ and $d$ intersect: $A=D$. Let $B,C\neq A$,  $AB\neq T_Aa, T_Bb$, $AC\neq T_Dd, T_Cc$, 
and let the  lines $AB$, $T_Aa$, $T_Dd$, $T_Bb$, $T_Cc$  be not isotropic.  
If $AB\neq AC$, then the mirrors $b$ and $c$ are triangular spirals centered at $A$. Otherwise, if $AB=AC$, 
then the mirrors $b$ and $c$ are conics: complex circles centered at $A$, see Fig.\ref{fig-spir}. 
\end{proposition}

 \begin{figure}[ht]
  \begin{center}
 % \vspace{-0.3cm}
   \epsfig{file=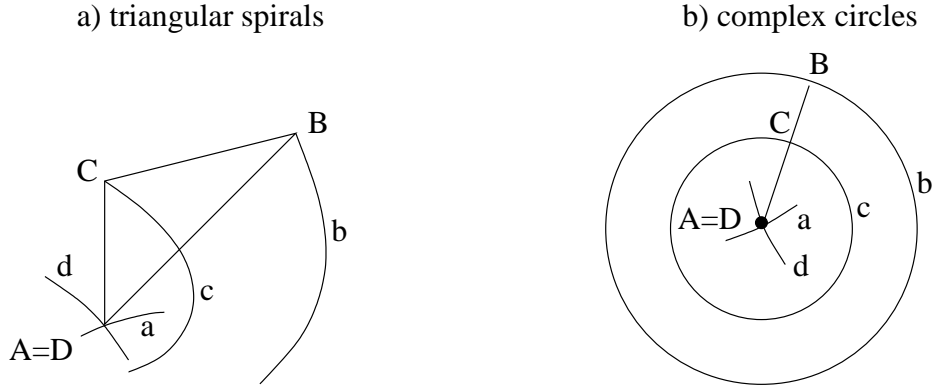}
 %  \vspace{-0.3cm}
    \caption{Family of degenerate orbits with $A=D$: the mirrors $b$ and $c$ are either spirals, or conics}
  \label{fig-spir}
  \end{center}
\end{figure}

\begin{proof} There exists an irreducible germ $\Gamma\subset U$ of analytic curve at $ABCD$  parametrized by local small 
complex parameter  $t$ and consisting of quadrilaterals $AB_tC_tD_t$ with fixed vertex $A$: $AB_0C_0D_0=ABCD$. 
 Let us fix it. One has $D_t\equiv D=A$. 
This follows from the fact that $D_t$ is found as a point of intersection of the curve $d$ with the line $L_t$ symmetric to $AB_t$ with 
respect to the tangent line $T_Aa$. Indeed,   the line 
$L_0$  is transverse to $T_Dd$, as is $AC$, 
by assumption and since  $L_0$ and $AC$ are symmetric  with respect to the line $T_Dd$. 
Therefore, the intersection point $D_t\in L_t\cap d$ identically coincides with $D=D_0$. 
 Let $H$ denote the composition of 
symmetries with respect to the tangent lines, first $T_Aa$, then $T_Dd$. Thus, $H\in SO(2,\cc)$ fixes $A$  and  $H(AB_t)=AC_t$
 for every $AB_tC_tD\in\Gamma$. We take $A$ as the origin. 

Case 1): $AB\neq AC$. Then $B\neq C$ and the  germ $\Gamma$ is embedded into 
$M^0_{A,H}$ via the mapping $t\mapsto ((B_t,T_{B_t}b), (C_t,T_{C_t}c))$, by construction and non-isotropicity condition. 
Its image is 
 a phase curve of the line field $\mcd_{A,H}$, analogously to discussions in \cite{bzh} and \cite[p.320]{gk2}. 
 This together with Proposition \ref{prosper} 
 implies that the projection $\Gamma\to\cp^2$ to the position of each one of the vertices $B$ and $C$ sends $\Gamma$ to 
 a triangular spiral centered at $A$, see Fig.\ref{fig-spir}a). Hence, $b$ and $c$ are triangular spirals. 
 
 Case 2): $AB=AC$. Then $H=\pm Id$ and $AB_t\equiv AC_t$. 
 Note that at least one of vertices, either $B_t$, or $C_t$ varies, since $\Gamma$ is a  curve. 
 To treat the case under consideration, we use the following remark.
 
 \begin{remark} \label{remlines} There exist no $k$-reflective analytic planar billiard such that some its two neighbor mirrors coincide with the same line: such a billiard would have no $k$-periodic orbits in the sense of Definition \ref{deforb}  (cf. \cite[proof of corollary 2.19]{alg}).
\end{remark}

 Subcase 2a): $B_t\equiv C_t\not\equiv const$. This implies that $b=c$ and the line $AB_t$ is tangent to $b$ at variable point $B_t$, as in  
 loc. cit. Therefore, $b=c=AB$, which is impossible by the above remark. Hence, this subcase 
 is impossible. 
 
 Subcase 2b): $B_t\not\equiv C_t$. Without loss of generality we consider that $B\neq C$.  Thus, for every $t$ small enough the 
 points  $A$, $B_t$ and $C_t$ are distinct and  lie on the same line. Note that $T_Bb, T_Cc\neq AB=AC$, by the condition of the proposition. 
  Hence, $T_{B_t}b, T_{C_t}c\perp AB_t$ for all $t$. This implies that $B_t,C_t\not\equiv const$ and $b$, $c$ are complex circles 
  centered at $A$, see Fig.\ref{fig-spir}b).   This  proves Proposition \ref{intspir}. 
 \end{proof}
%
%The triangle $ABC$ represents a points in $\mctah^0$, by definition. 
%The two latter statements together imply that both $B'$ and $C'$ are non-constant on $\Gamma$ 
%and vary along triangular spirals centered at $A$: the corresponding 
%parameters are $H$ and appropriate squared perimeter of the triangle $ABC$. This follows from definition and reflection law and proves the 
%proposition.

\begin{proposition} \label{pnonalg} Let a planar analytic curve be a triangular spiral with respect to two distinct centers. Then it is algebraic.
\end{proposition}

\begin{proof} A triangular spiral is a phase curve of a triangular algebraic line field. The latter field  is invariant under complex rotations: 
the isometries fixing the center of the spiral with unit Jacobian. 
%Moreover it follows by construction that each triangular line field has irreducible graph as a section of the projectivized 
%tangent bundle. 
Suppose the contrary: the spiral under consideration is not algebraic. Then the corresponding line field is uniquely defined:  
two algebraic line fields  
coinciding  on a non-algebraic curve (which is Zariski dense)  coincide everywhere. Thus, 
the latter line field is invariant under complex rotations around two distinct centers. The latter rotations generate the whole group 
of complex isometries of $\cc^2$ with unit Jacobian. 
Thus, the line field is invariant under all of them, which is impossible. The contradiction thus obtained  
proves the proposition.
\end{proof} 

%FIXME: say more about irreducibility of graph of algebraic triangular line field and coincidence. 

%We will also use the following version of Proposition \ref{inspire} treating the case, when a  quadrilateral $ABCD\in U$ degenerates to 
%a single point.
%
%\begin{proposition} \label{degspir} 
%Let $(a,A)$, $(b,B)$, $(c,C)$, $(d,D)$ be a germ of 4-reflective analytic planar billiard. Let $A=B=C=D$ as points in $\cp^2$. Let the curve 
%$b$ be not a line, the line $T_Bb$ be non-isotropic 
%and  $D$ be not marked point for the curve $d$ (up to exchange of $b$ and $d$). 
%Then $(a,A)=(c,C)$. 
%\end{proposition}
%
%\begin{proof} Let $U$ be  the 4-reflective set. 
%First note that no two neighbor mirror germs coincide: 
%$(a,A)\neq(b,B)$, $(b,B)\neq (c,C)$ etc. Indeed, otherwise there would exist an irreducible 
%germ at $ABCD$ of analytic curve $\Gamma\subset U$ such that $\Gamma\setminus ABCD$ consists of 
%quadrilaterals with coinciding pair of neighbor vertices that are forbidden by either Proposition \ref{neighbmir}, or Remark \ref{remlines}. 
%This together with nonlinearity implies that  there exists a germ at $ABCD$ of analytic curve $\Gamma\subset U$ 
%such that for every $A'B'C'D'\in\Gamma\setminus ABCD$ one has $A'\neq B'$, $A',B'\neq A$, $C'\neq B'$,   
% the line $A'B'=B'C'$ is tangent to $b$ at $B'$ and is not isotropic. Therefore, $a_A\equiv c_C$, by Corollary \ref{connect} and the 
% regularity of the germ $(d,D)$. Proposition \ref{degspir} is proved.
% \end{proof}
% 
% FIXME: RED proof
%zdesa
%COMMENT: Oboidemsia bez nego. 

\subsection{Tangencies in $k$-reflective billiards}
Here we recall the results of \cite[subsection 2.4]{alg}. 

  We  deal with  $k$-reflective analytic planar billiards $a_1,\dots,a_k$ in $\cp^2$. 
  %: the mirrors are (germs of) analytic curves parametrized by (germs of) analytic mappings $\pi_{a_j}:\hat a_j\to a_j$, 
%$\hat a_j$ are neighborhoods of zero in $\cc$, $j=1,\dots,k$. 
Let $U\subset\hat a_1\times\dots\times\hat a_k$ be the $k$-reflective set. The results of loc.cit.  
presented below concern  degenerate quadrilaterals in $U\setminus U_0$: limits $A_1\dots A_k$ of $k$-periodic orbits 
 such that for a certain $j$  with $a_j$ being not a line the tangent line $T_{A_j}a_j$ and the adjacent edges 
$A_{j\pm1}A_j$ collide to the same non-isotropic limit. Then the limit vertex $A_j$ will be called a  {\it tangency vertex}. 
Proposition \ref{tang} shows that the latter cannot happen to be the 
only degeneracy of the limit $k$-gon. Its Corollary \ref{connect} presented at the end of the subsection concerns the case, 
when $k=4$. It says that if the tangency vertex is distinct from its neighbor limit vertices, then its opposite vertex 
should be either also  a  tangency vertex, or  a cusp with a non-isotropic tangent line. 
Proposition \ref{neighbmir} extends Proposition \ref{tang} to the case, when  some subsequent mirrors coincide 
and the corresponding subsequent vertices of a limiting orbit collide. 
%Its Corollary  \ref{lnoncoinc} shows that there are no 4-reflective algebraic billiards with a pair of coinciding neighbor mirrors. 

\begin{definition} \label{defmark} A point of  a planar irreducible analytic curve is {\it marked}, if it is either a cusp, or an isotropic tangency point.  
Given a parametrized curve $\pi_a:\hat a\to a$, a point $A\in\hat a$ is marked, if it corresponds to a marked point 
of the local branch $a_{A}$, see Convention \ref{conv2}. 
\end{definition}

\begin{proposition} \label{tang} \cite[proposition 2.16]{alg} Let $a_1,\dots,a_k$ and $U$ be as above. 
%be a (germ of)  $k$-reflective analytic 
%billiard.  
Then $U$ contains no $k$-gon $A_1\dots A_k$ with the following properties:

- each pair of neighbor vertices correspond to distinct points, and no vertex is a marked point;

- there exists a unique $s\in\{1,\dots,k\}$ such that the line $A_sA_{s+1}$ is tangent to the curve $a_s$ at $A_s$, 
and the latter curve is not a line, see Fig.\ref{fig-tang1}. 
\end{proposition}

\begin{remark} A  real version of Proposition \ref{tang} is contained in \cite{gk2} 
(lemma 56, p.315 for $k=4$, and its  generalization (lemma 67, p.322) for higher $k$). 
\end{remark}

\vspace{-0.5cm}
\begin{figure}[ht]
  \begin{center}
   \epsfig{file=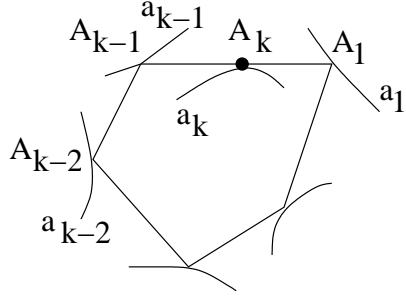}
    \caption{Impossible degeneracy of simple tangency: $s=k$.}
    \label{fig-tang1}
  \end{center}
\end{figure}

 \begin{proposition} \label{neighbmir} \cite[proposition 2.18]{alg} Let $a_1,\dots,a_k$ and $U$ be as at the beginning of the subsection. Then 
 $U$ contains no $k$-gon $A_1\dots A_k$ with the following properties:
 
1) each its vertex is not a marked point of the corresponding mirror;

2) there exist $s,r\in\{1,\dots,k\}$, $s<r$ such that 
$a=a_s=a_{s+1}=\dots=a_r$, $A_s=A_{s+1}=\dots=A_r$, and $a$ is not a line;
%, and $A_i\neq A_{i\pm1}$;

3)   For every $j\notin \mathcal R=\{ s,\dots,r\}$ one has 
$A_j\neq A_{j\pm1}$ and the line $A_{j-1}A_j$ is not tangent to $a_j$ at $A_j$, see Fig.\ref{fig-tang2an}.
\end{proposition}

\begin{figure}[ht]
  \begin{center}
   \epsfig{file=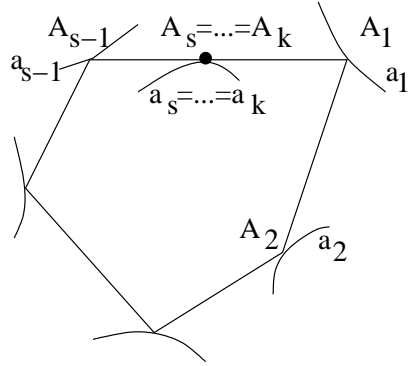}
    \caption{Coincidence of subsequent  vertices and mirrors: $r=k$.}
    \label{fig-tang2an}
  \end{center}
\end{figure}

 \begin{corollary} \label{connect} \cite[corollary 2.20]{alg} 
 Let $a$, $b$, $c$, $d$ be a 4-reflective analytic 
 billiard, and let $b$ be not a line. 
 Let $U\subset\hat a\times\hat b\times\hat c\times\hat d$ be the 4-reflective set. Let $ABCD\in U$ be such that 
 $A\neq B$, $B\neq C$, the line $AB=BC$ is tangent to the curve $b$ at $B$ and is not isotropic. Then 
 
 - either $AD=DC$ is tangent to the curve $d$ at $D$,  $\pi_a(A)=\pi_c(C)$, $a=c$ and the  corresponding local branches  
coincide,  i.e., $a_A=c_C$ (see Convention \ref{conv2}): ``opposite tangency connection'', see Fig.\ref{fig-tang-opp6}a); 

 - or $D$ is a cusp of the local branch $d_D$ and the tangent line $T_Dd$ is not isotropic: ``tangency--cusp connection'', see Fig.\ref{fig-tang-opp6}b). 
  \end{corollary}
 
\begin{figure}[ht]
  \begin{center}
   \epsfig{file=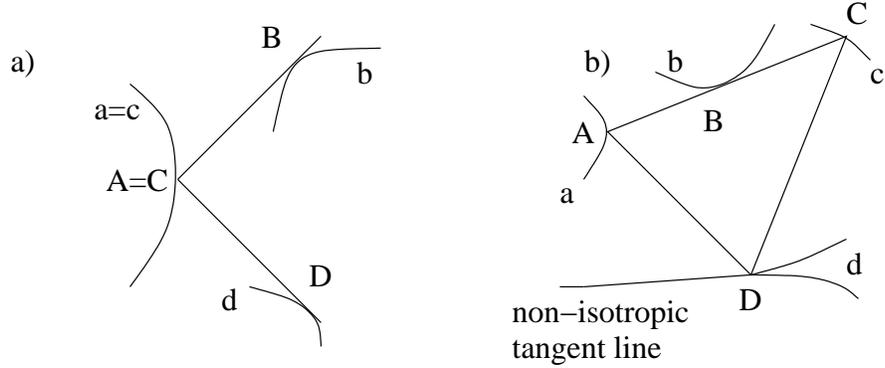}
    \caption{Opposite degeneracy to tangency vertex: tangency or cusp.}
    \label{fig-tang-opp6}
  \end{center}
\end{figure}

\subsection{Singular analytic distributions}
Here we recall definitions and properties of singular analytic distributions.  The author believes that 
the material of the present subsection is known to specialists, but he did not find it in literature.

\begin{definition} Let $W$ be a complex manifold, $n=dim W$, $\Sigma\subset W$ be a nowhere dense closed subset, $m\leq n$. Let 
$\mcd$ be an analytic field of codimension $m$ vector subspaces $\mcd(y)\subset T_yW$, $y\in W\setminus\Sigma$. We say that 
$\mcd$ is a {\it singular analytic distribution} of codimension $m$ (dimension $n-m$) with the singular set $\Sigma=Sing(\mcd)$, 
if it extends analytically to no point in $\Sigma$ and 
each $x\in W$ has  a neighborhood $U$ where there exists a finite collection 
$\Omega$ of  holomorphic 1-forms  such that $\mcd(y)=\{ v\in T_yW \ | \ \Omega(v)=0\}$ for every 
$y\in U\setminus\Sigma$.  (This generalizes the definition of a codimension one singular holomorphic foliation \cite[p.11]{cerveau}. 
A similar definition in smooth case can be found in \cite[p.8]{isid}.) 
 \end{definition}
%  \newpage
  \begin{remark} \label{remmer}  Every $k$-dimensional singular analytic distribution on a complex manifold $W$ is  
   defined by a meromorphic\footnote{Recall that a mapping $V\to W$ of complex manifolds (or analytic sets in complex manifolds) is {\it meromorphic}, 
 if it is well-defined and holomorphic on an open and dense subset in $V$, 
 and the closure of its graph is an analytic subset in $V\times W$, see Convention \ref{convan}. It is well-known that if  $W$ is compact and 
 $V$ is irreducible,  then the set of  indeterminacies of every 
 meromorphic mapping $V\to W$ is contained in the union of the singular set of $V$ and an analytic subset in $V$ 
 of codimension at least two. A mapping is {\it bimeromorphic}, if it 
 is meromorphic together with its inverse.} section of the Grassmanian $k$-subspace bundle $Gr_k(TW)$ of the tangent bundle $TW$, and vice versa: 
   each meromorphic section defines a $k$-dimensional singular analytic 
   distribution. Its singular set is an analytic subset in $W$ of codimension 
   at least two, being the indeterminacy locus of a meromorphic section of a bundle with compact fibers. 
   \end{remark}
   
   %Oka Coherence Theorem; Hilbert's sygysies theorem.

 \begin{example} \label{exrest} 
 Let $M$ be a complex analytic manifold,  $N\subset M$ be a connected complex submanifold, $\mcd$ be a (regular) analytic 
 distribution on $M$. The intersection $\mcd|_N(x)=T_xN\cap\mcd(x)$ with  $x\in N$ has constant and minimal dimension 
 on an open and dense subset $N^0\subset N$. The subspaces $\mcd|_N(x)\subset T_xN$ form 
 a singular analytic distribution $\mcd|_N$ on $N$ that is called 
 the {\it restriction}  to $N$ of the distribution $\mcd$. Its singular set is contained in the complement $N\setminus N^0$: the set of those points 
 $x$, where the above dimension is not minimal. The restriction to $N$ of a singular analytic distribution $\mcd$ on $M$ with 
 $N\not\subset Sing(\mcd)$ is defined analogously; 
 it is also a singular analytic distribution on $N$ whose singular set is contained in the union of the intersection $Sing(\mcd)\cap N$ and 
 the set of those points $x\in N$, where the above dimension $dim(\mcd|_N(x))$ is not minimal. 
 \end{example}
 
  \begin{example} Let $\mcd$ be a singular analytic distribution on a complex manifold $W$. Let $M$ be another connected 
  complex manifold, and let 
$\phi:M\to W$ be  a non-constant holomorphic mapping. For every $x\in M$ set 
$$\phi^*\mcd(x)=(d\phi(x))^{-1}(\mcd(\phi(x))\cap d\phi(x)(T_xM))\subset T_xM.$$
The subspaces $\phi^*\mcd(x)$ form a singular analytic distribution on $M$ called the {\it pullback distribution}. In the case, when $\phi$ is an 
immersion on an open and dense subset, the dimension of the distribution $\phi^*\mcd$ equals the minimal dimension of the above 
intersection. 
%
%Then  $\mcd$ induces a pullback singular analytic distribution on $M$ denoted $\phi^*\mcd$, 
%whose dimension is equal to the minimal dimension of the 
%intersection $\mcd(\phi(y)) \cap (d\phi(y))(T_yM)$ for $y\in M'$. 
%The singular set of a singular analytic distribution is always an analytic subset. 
%Moreover, the  singular set of a singular distribution on a complex analytic manifold always 
%has complex codimension at least two. 
%The author is pretty sure that the latter statement on codimension is well-known to specialists, but had not found it in the literature, except for 
%the case of one-dimensional distribution. 
%
%SSYLKU??? 
%
%This statement will not be used in the paper, and  its proof is omitted to save the space. The singular set of a singular analytic distribution 
%on an analytic space is an analytic subset consisting of the singularities of the space and those of the distribution. 
\end{example}

\begin{convention} \label{convan} Let $W$ be a complex manifold, $M\subset W$ be an analytic subset. Everywhere below for simplicity we 
say that a {\bf subset $N\subset M$ is analytic,} if it is an analytic subset of the ambient manifold $W$. 
\end{convention}

 \begin{definition} \label{moreg} Let $W$ be a complex manifold, 
 $M\subset W$ be an irreducible analytic subset, and let $\mcd$ be a singular analytic distribution on $W$, $M\not\subset Sing(\mcd)$. 
 There exists an open and dense subset of those points\footnote{Everywhere below for an analytic set $M$ by $M_{reg}$ 
 ($M_{sing}$)  we denote the set of its smooth (respectively, singular) points} $x\in M_{reg}$ regular for $\mcd$,   for which the intersection 
 $\mcd|_M(x)=\mcd(x)\cap T_xM$ has minimal dimension. 
 Then we say that the subspaces $\mcd|_M(x)$ form a {\it singular analytic distribution  $\mcd|_M$ on $M$}. It is regular on an open dense 
 subset $M^0_{reg}\subset M_{reg}$. Its {\it singular set} $M\setminus M^0_{reg}$  is  the 
 union of the set $M_{sing}$ and the set of those points $x\in M_{reg}$ where the distribution $\mcd|_M$ does not extend analytically. 
The distribution $\mcd|_M$  is also called the {\it restriction} to $M$ of the distribution $\mcd$. The 
 restriction of a singular analytic distribution $\mcd|_M$ to an irreducible analytic subset $V\subset M$, $V\not\subset Sing(\mcd|_M)$ 
 is a singular analytic distribution on $V$ 
 defined analogously: it coincides with $\mcd|_V$. In the case, when the analytic set $M$ is a union of several irreducible components,  the 
 restriction of the distribution $\mcd$ to each  component will be referred to, as a singular distribution on $M$ 
 (which may have different dimensions on different components). 
 \end{definition}
 
 \begin{example} \label{ex-sing} The Birkhoff distribution $\mcd^k$ 
 introduced at the end of Section 1 extends to a singular analytic distribution on the closure $\overline{\mcrr_{0,k}}\subset\mcp^k$. 
% Similarly, the distributions $\mcd_a$ and $\mcd_{ab}$ introduced at the same place, which are pullbacks of the distribution 
% $\mcd^4$ on $\overline{\mcrr_{0,4}}$, are singular analytic distributions 
  \end{example}

 \begin{definition} An {\it integral $l$-surface} of a singular analytic  distribution $\mcd$ on an analytic variety\footnote{Everywhere below by {\it analytic variety} we mean an analytic subset in a complex manifold} $M$
is a holomorphic connected $l$-dimensional surface $S\subset M$ lying outside the singular set of $\mcd$ 
  such that $T_xS\subset\mcd(x)$ for every $x\in S$. An $m$-dimensional singular analytic distribution is {\it integrable,} if there exists an integral $m$-surface through each its regular point. 
 \end{definition}
 
 \begin{remark} The singular set of a singular analytic distribution is always an analytic subset in the ambient variety 
 (see Convention \ref{convan}), as in Remark \ref{remmer}. 
 In general an integral surface is not an analytic set. Indeed, a generic linear vector field on $\cp^2$ has transcendental orbits. Hence, they  
are not analytic subsets in $\cp^2$, by Chow's Theorem \cite[p.167]{griff}.  
 \end{remark}

\begin{definition} \label{defint} Let $M$ be an irreducible analytic subset in a complex manifold $V$. A $p$-dimensional {\it intrinsic} singular 
analytic distribution on $M$ is a meromorphic section $\mcd:M\to Gr_p(TV)|_M$ of the Grassmanian bundle (see Footnote 3) such that 
$\mcd(x)\subset T_xM$ for regular points $x\in M$ where $\mcd$ is holomorphic. 
\end{definition}

\begin{remark} \label{intrinsic} Each singular analytic distribution is an intrinsic one. Conversely, each intrinsic singular analytic distribution 
is transformed to a usual singular analytic distribution by a bimeromorphic mapping. Namely, consider the bundle projection 
$\pi:Gr_p(TV)\to V$ and 
the graph $\Gamma$ of the section $\mcd$ (which is an analytic subset in $Gr_p(TV)$). Let $\mcf$ be the tautological 
distribution on $Gr_p(TV)$: for every $\lambda\in Gr_p(TV)$  the subspace $\mcf(\lambda)\subset T_{\lambda}Gr_p(TV)$ is the 
preimage of the subspace $\lambda\subset T_{\pi(\lambda)}V$ under the differential $d\pi(\lambda)$. The restriction 
$\mcf|_{\Gamma}$ is a singular holomorphic distribution transformed to $\mcd$ by the bimeromorphic projection $\pi:\Gamma\to M$. 
\end{remark}

%We will  also deal with the following more general  situation. Let $M$ be an irreducible 
%analytic subset in a complex manifold $V$,  $\Sigma\subset M$ be a proper analytic subset that contains the singular locus: 
%$M\setminus\Sigma\subset M_{reg}$.  Let $\mcd$ be a $p$-dimensional 
%analytic distribution on $M\setminus\Sigma$ such that the closure of its 
%graph as a section of the Grassmanian bundle $Gr_p(TV)|_M$ is an analytic subset  $\Delta_M\subset Gr_p(TV)$. 
%Then we say that $\mcd$ is an {\it intrinsic singular analytic distribution} on $M$. Its lifting $\wt\mcd$ to $\Delta_M$ under the bundle 
%projection  is a singular analytic distribution. 
%Indeed, consider the tautological distribution $\mcf$ on $Gr_p(TV)$: 
%for every $\lambda\in Gr_p(TV)$ the  subspace $\mcf(\lambda)\subset T_{\lambda}Gr_p(TV)$ is the preimage of the subspace 
%$\lambda\subset T_{\pi(\lambda)}V$ under the differential $d\pi(\lambda)$. The restriction  $\mcf|_{\Delta_M}$ coincides with $\wt\mcd$, 
%by construction. The projection $\pi:\Delta_M\to M$ has a diffeomorphic inverse on $M\setminus\Sigma$ that transforms $\mcd$ to 
%$\wt\mcd$.   This shows that {\it every intrinsic singular analytic distribution restricted to its analyticity domain 
%is transformed by a diffeomorphism to a true singular analytic distribution}, and the closure of the graph of the latter 
%diffeomorphism is an analytic set. 
%\end{remark}

\begin{proposition} \label{anint} Let an $m$-dimensional  (intrinsic) singular analytic distribution $\mcd$ 
on an analytic subset $N$ in a complex manifold 
have at least one $m$-dimensional integral surface. Given an arbitrary union $S$ of $m$-dimensional integral surfaces, let 
$M\subset N$ denote the minimal analytic subset in $N$ containing $S$. 
Then the restriction $\mcd|_M$ is an integrable $m$-dimensional singular analytic distribution. 
\end{proposition}

\begin{proof} It suffices to prove the proposition for a true singular analytic distribution,  locally defined as   kernel field  
of a system of holomorphic 1-forms (Remark \ref{intrinsic}). 
%The analytic set $M$ is obviously irreducible. 
%Here by $M^0_{reg}\subset M$ we denote the subset of points regular for both $M$ and the distribution $\mcd$. 
The set $\{ x\in M^0_{reg} \ | \ \mcd(x)\subset T_xM\}$ coincides with all of $M^0_{reg}$, since it contains 
$S\cap M^0_{reg}$ and its closure 
is an analytic subset in $M$. Similarly, the set of those points in $M^0_{reg}$ where the distribution $\mcd|_M$ satisfies the Frobenius integrability condition coincides with all of  $M^0_{reg}$, since it  contains  
$S$ and its closure is an analytic subset in $M$. 
Thus,    $\mcd|_M$ is an $m$-dimensional integrable distribution. The proposition is proved. 
\end{proof}

\subsection{Birkhoff distributions and periodic orbits}

Here we recall the definition and basic properties of Birkhoff distribution and its restricted versions. Consider  the space 
 $\mcp=\mathbb P(T\cp^2)$, which consists of pairs $(A,L)$, $A\in\cp^2$, $L$ being a one-dimensional subspace in $T_A\cp^2$. 
Its natural projection to $\cp^2$ will be denoted by $\Pi$. The {\it standard contact structure} is the two-dimensional 
analytic distribution $\mch$ on $\mcp$ 
given by the  $d\Pi$-pullbacks of the lines $L$: 
$$\mch(A,L)=(d\Pi(A,L))^{-1}(L)\subset T_{(A,L)}\mcp.$$
The distribution $\mch^k=\oplus_{j=1}^k\mch$ is the $2k$-dimensional 
product distribution on $\mcp^k$. Recall that $\mcrr_{0,k}\subset\mcp^k$ is the subset of  $k$-tuples 
$((A_1,L_1),\dots,(A_k,L_k))$ such that for every $j=1,\dots,k$ one has $A_j\in\cc^2=\cp^2\setminus\oc_{\infty}$, $A_j\neq A_{j\pm 1}$, the lines $A_jA_{j-1}$, 
$A_jA_{j+1}$ are  symmetric with respect to the line $L_j$, and the three latter lines are distinct and non-isotropic. 
This is a $2k$-dimensional smooth quasiprojective variety. The {\it Birkhoff distribution} 
$\mcd^k$ is the restriction to $\mcrr_{0,k}$ of the product distribution $\mch^k$: 
\begin{equation}\mcd^k(x)=T_x\mcrr_{0,k}\cap\mch^k(x) \text{ for every } x\in\mcrr_{0,k}.\label{birkdist}\end{equation}
% For every analytic planar billiard $a_1,\dots,a_k$ we lift its $k$-periodic orbits $A_1\dots A_k\in\ha_1\dots\ha_k$ 
% to the space $\mcrr_{0,k}$ via the mapping $A_1\dots A_k\mapsto ((A_1,T_{A_1}a_1),\dots,(A_k,T_{A_k}a_k))$ and identify the orbits 
% with their images.  
%% 
%% 
%% 
%% It is known, see \cite{bzh}, that {\it every lifted analytic family of $k$-periodic orbits $A_1,\dots,A_k$ of an analytic billiard $a_1,\dots,a_k$ 
%% is tangent to Birkhoff distribution.}  In particular, {\it a connected open set of $k$-periodic orbits in a $k$-reflective billiard is an integral 
%% surface of Birkhoff distribution.} 
% 
% In what follows we deal with the restricted Birkhoff distributions that govern 4-reflective billiards with either one or two given mirrors. 
%  Namely, 
This is a $k$-dimensional analytic distribution. It is the complexification of the real Birkhoff distribution introduced in \cite{bzh}. 
For every two irreducible 
analytic curves $a,b\subset\cp^2$ with maximal normalizations $\pi_a:\ha\to \cp^2$, $\pi_b:\hb\to \cp^2$ we will denote 
$$\mcp_a=\ha\times\mcp^3; \ \mcp_{ab}=\ha\times\hb\times\mcp^2.$$
  We consider the natural embeddings $\eta_a:\mcp_a\to\mcp^4$, $\eta_{ab}:\mcp_{ab}\to\mcp^4$: 
  $$\eta_a(A,(B,L_B),(C,L_C),(D,L_D))=((\pi_a(A),T_Aa), (B,L_B),(C,L_C),(D,L_D)),$$
   $$\eta_{ab}(A,B,(C,L_C),(D,L_D))=((\pi_a(A),T_Aa),  (\pi_b(B),T_Bb), (C,L_C),(D,L_D)).$$
  \begin{remark} \label{remcusps} 
  The critical points of the mappings $\eta_a$ ($\eta_{ab}$) are contained in the sets $Cusp_a\subset\mcp_a$, $Cusp_{ab}\subset\mcp_{ab}$ of 
  those points 
   for which $A$ ($A$ or $B$) is a cusp of the corresponding curve (see Footnote 1 in Section 1). 
   The mappings $\eta_a$ and $\eta_{ab}$ are immersions outside 
   the sets $Cusp_a$ and $Cusp_{ab}$. 
   \end{remark}
   Consider the subsets
  \begin{equation}
  M_a^0=\eta_a^{-1}(\mcrr_{0,4})\setminus Cusp_a\subset\mcp_a, \ M_{ab}^0=\eta_{ab}^{-1}(\mcrr_{0,4})\setminus Cusp_{ab}
  \subset\mcp_{ab},
  \label{mab}\end{equation}
   $$M_a=\overline{M_a^0}\subset\mcp_a; \ M_{ab}=\overline{M_{ab}^0}\subset\mcp_{ab}:$$  
the closures are taken in the usual topology. The subsets $M_a\subset\mcp_a$ and $M_{ab}\subset\mcp_{ab}$ are obviously analytic. The {\it restricted (pullback) Birkhoff distributions} 
  $\mcd_a$ on $M_a^0$ and $\mcd_{ab}$ on $M_{ab}^0$ respectively are  the pullbacks of the Birkhoff distribution $\mcd^4$: 
 \begin{equation}
   \mcd_a(x)=(d\eta_a(x))^{-1}(\mcd^4(\eta_a(x))\cap d\eta_a(x)(T_xM_a^0))\subset T_xM_a^0, \   x\in M_a^0; 
   \label{mcda}\end{equation}
  \begin{equation} \mcd_{ab}(x)=(d\eta_{ab}(x))^{-1}(\mcd^4(\eta_{ab}(x))\cap d\eta_{ab}(x)(T_xM_{ab}^0))\subset T_xM_{ab}^0, \  
 x\in M_{ab}^0.\label{birkrest}\end{equation}
  They extend to singular analytic distributions on $M_a$ and $M_{ab}$ respectively in the sense of Subsection 2.6. For example,  
  $\mcd_a$ is the restriction to $M_a$ of the distribution $T\hat a\oplus\mathcal H^3$ on $\mcp_a=\hat a\times\mcp^3$.  One has 
 $$dim M_a=6, \ dim\mcd_a=3; \ \ dim M_{ab}=4, \ dim\mcd_{ab}=2.$$
 
\begin{definition} \label{defnondeg} (complexification of \cite[definition 14]{gk2}) 
Let $k\in\nn$. A $k$-gon $A_1\dots A_k\in(\cp^2)^k$ is said to be {\it non-degenerate,} if 
for every $j=1,\dots,k$ (we set $A_{k+1}=A_1$, $A_0=A_k$) one has $A_j\in\cc^2=\cp^2\setminus\oc_{\infty}$,  $A_{j+1}\neq A_j$, 
$A_{j-1}A_j\neq A_jA_{j+1}$ and the line $A_jA_{j+1}$ is not isotropic.
We will call the complex lines $A_jA_{j\pm1}$ the {\it edges adjacent to } $A_j$. 
\end{definition}

\begin{remark} \label{thebove} 
The above sets $\mcrr_{0,k}$, $M_a^0$, $M_{ab}^0$  are projected to the sets of non-degenerate $k$-gons (quadrilaterals). 
A periodic billiard orbit  in the sense of Definition \ref{deforb} is non-degenerate, provided that its vertices lie in $\cc^2$, 
its edges are non-isotropic and 
every two adjacent edges are distinct. The $k$-reflective set $U$  of a $k$-reflective billiard contains  an open and dense subset 
$U_1\subset U$ of those non-degenerate orbits whose vertices are not marked points (see Definition \ref{defmark}) 
of the corresponding mirrors. 
\end{remark}

\begin{definition}  (cf. \cite[definition 16]{gk2}) Consider some of the above Birkhoff distributions, let us denote it $\mcd$, and let $M$ denote 
the underlying manifold (e.g., $\mcrr_{0,k}$, $M_{ab}^0$,...) carrying $\mcd$. Consider the projections of the manifold $M$ to the positions of 
vertices of the $k$-gon (quadrilateral). A subspace $E\subset\mcd(x)$, $x\in M$, is said to be 
{\it non-trivial}, if the restriction to $E$ of the differential of each above projection has positive rank. (Then the  rank equals one.) An 
 {\it integral surface} of the same distribution is {\it non-trivial}, if its tangent planes are non-trivial. 
\end{definition}

For every analytic billiard $a$, $b$, $c$, $d$ there exist natural analytic mappings $\Psi_a:\ha\times\hb\times\hc\times\hd\to 
\mcp_a$, $\Psi_{ab}: \ha\times\hb\times\hc\times\hd\to \mcp_{ab}$: 

$$\Psi_a(ABCD) = (A,(B,T_Bb),(C,T_Cc),(D,T_Dd));$$
 \begin{equation}\Psi_{ab}(ABCD)=(A,B,(C,T_Cc), (D,T_Dd)).\label{embed}\end{equation}

\begin{proposition} \label{birktang} Let $a$, $b$, $c$, $d$ be a 4-reflective billiard. 
The mappings $\Psi_a$, $\Psi_{ab}$ send the subset $U_1\subset U$  (see Remark \ref{thebove}) 
to $M_a^0$, $M_{ab}^0$, and the images of its connected components are non-trivial integral surfaces of 
the restricted Birkhoff distributions $\mcd_a$ and $\mcd_{ab}$ respectively. 
Vice versa, each non-trivial integral surface of any of the latter distributions is the image of an open  set of 
quadrilateral orbits of a 4-reflective billiard 
$a$, $b$,  $c$, $d$ with given mirror $a$ (respectively, given mirrors $a$ and $b$). 
\end{proposition}

The proposition is the direct complexification of an  analogous result from \cite{bzh} and Yu.G.Kudryashov's lemmas 
\cite[section 2, lemmas 17, 18]{gk2}.

Everywhere below for every $x\in M_{ab}^0\subset\mcp_{ab}=\ha\times\hb\times\mcp^2$ we denote 
 $$l_a=l_a(x)=A(x)D(x), \ l_b=l_b(x)=B(x)C(x).$$
 \begin{remark} \label{lab} The lines $l_a$ and $l_b$ depend only on $(A,B)=(A(x),B(x))$: these are the lines symmetric to $AB$ with respect to the 
 lines $T_Aa$ and $T_Bb$ respectively. Sometimes we will write $l_a=l_a(A,B)$, $l_b=l_b(A,B)$. 
 \end{remark}
 For every $x\in M^0_{ab}$ ($x\in M^0_a$) 
 the corresponding lines $L_G$, $G=(B,) C,D$, will be denoted by $L_G(x)$. The projections to the positions of vertices will be denoted by 
 $$\nu_a:\mcp_a\to\ha, \ \nu_{ab}:\mcp_{ab}\to\ha\times\hb,$$ 
 $$\nu_G:\mcp_a, \mcp_{ab}\to\cp^2, \ x\mapsto G(x) \text{ for } G=B,C,D \text{ (respectively, } G=C,D).$$

\begin{remark} \label{rempro} 
The above projections $\nu_a$ and $\nu_{ab}$ are proper and epimorphic. The corresponding preimages of points 
are compact and naturally identified with projective algebraic varieties. 
\end{remark}

 \begin{proposition} \label{pnontriv} The planes of the Birkhoff distribution $\mcd_{ab}$ on $M_{ab}^0$ 
  are non-trivial, and hence, so is each its integral surface. 
\end{proposition}

\def\wtmab{M_{ab}}
\begin{proof} Suppose the contrary: there exists an $x\in M_{ab}^0$ such that the differential of some of the above projections, say $\nu_a$ 
vanishes identically on $\mcd_{ab}(x)$. We consider only the case of projection $\nu_a$: the cases of other projections are treated analogously. 
Then the kernel 
$$K(x)=\operatorname{Ker}(d\nu_{b}|_{\mcd_{ab}(x)})\subset\mcd_{ab}(x)$$
is at least one-dimensional subspace. The vertices $A(x)=\nu_a(x)\in \hat a$, $B(x)=\nu_b(x)\in \hat b$ are not cusps, 
since $x\in M_{ab}^0$ by assumption. 
The line functions $l_a$, $l_b$ have zero derivatives along $K(x)$, as do the vertices $A$, $B$. 
The derivative along  $K(x)$ of at least one of the vertices $C$ or $D$, say $D$ is not identically zero. Then $d\nu_D(K(x))=l_a(x)\neq
L_D(x)$, since $D\in l_a$ and $l_a$ has zero derivative. Thus, $d\nu_D(\mcd_{ab}(x))\not\subset L_D(x)$, -- a contradiction to the definition of 
the distribution $\mcd_{ab}$. The proposition is proved. 
\end{proof} 
%
%
%there exists  a trivial integral surface $S\subset M_{ab}^0$. 
%Then its projection to the position of some vertex is  constant. 
%
%Case 1): the projection $\nu_{ab}: S\to\ha\times\hb$ has rank 1 at a generic point. Then $S$ is fibered by 
%holomorphic curves on which $A,B\equiv const$, hence $l_a,l_b\equiv const$, see Remark \ref{lab}. Some of the vertices $D$ or $C$, 
% say $D$ should be non-constant along some one-dimensional fiber $\Gamma$ of the latter fibration. 
% Thus, while $y$ moves along the curve $\Gamma$,  
%the line $l_a(y)=A(y)D(y)$ remains constant and distinct from $C(y)D(y)$, while the point $D(y)$ moves along the constant line $l_a$. Hence, 
%$L_D(y)=l_a=A(y)D(y)$, by definition and since $\Gamma$ is tangent to the distribution $\mcd_a$. On the other hand, 
%$L_D(y)\neq A(y)D(y)$, by definition and since $y\in M_{ab}^0\subset\eta_{ab}^{-1}(\mathcal R_{0,4})$. The contradiction thus obtained shows that Case 1) is impossible.
%
%Case 2): the projection $\nu_{ab}$ is constant on $S$. This case is treated analogously to Case 1) and is also impossible. 
%
%Case 3): the projection $\nu_{ab}|_S$ has rank 2 at a generic point, while some of the vertices $D$ or $C$, say $D$ is constant along 
%the surface $S$. This means that the lines $A(y)B(y)$ with $y\in S$ form a two-dimensional family, 
%while their reflection images $l_a(y)$ from $T_Aa$ pass through the same point $D$ and hence, form a one-dimensional 
%family. This is obviously impossible. Proposition \ref{pnontriv} is proved. 
%\end{proof} 

\section{Non-integrability of the Birkhoff distribution $\mcd_{ab}$ and corollaries} 

\subsection{Main lemma, corollaries and plan of the proof}

In the present section we prove the following lemma on the non-integrability of the two-dimensional 
Birkhoff distribution $\mcd_{ab}$ and corollaries. 

\begin{lemma} \label{lnonint} For every pair of  analytic curves $a,b\subset\cp^2$ distinct from isotropic lines that are not both lines 
the corresponding Birkhoff distribution $\mcd_{ab}$ is non-integrable. Moreover, there is no 
 three-dimensional irreducible analytic subset $M\subset M_{ab}$ (see Convention \ref{convan}) 
 %tangent to $\mcd_{ab}$ 
 such that  $M\cap M^0_{ab}\neq\emptyset$ and the restriction $\mcd_{ab}|_M$ is 
 two-dimensional and integrable. 
\end{lemma}

\begin{corollary} \label{cnonint} In the conditions of Lemma \ref{lnonint} the union of all the integral surfaces of the 
Birkhoff distribution $\mcd_{ab}$ in $M^0_{ab}$ is contained in a two-dimensional analytic subset in $M_{ab}$. 
\end{corollary}

\begin{proof} The minimal analytic subset $M\subset M_{ab}$ containing all the integral surfaces is tangent to 
$\mcd_{ab}$, and the distribution $\mcd_{ab}$ is integrable there (Proposition \ref{anint}). Hence, $dim M=2$, 
 by Lemma \ref{lnonint}. This proves the corollary. 
\end{proof} 

\begin{remark} In the case, when $a$ and $b$ are lines, the statements of Lemma \ref{lnonint} and the corollary are false. In this case there 
exists a one-parametric family of 4-reflective billiards $a$, $b$, $c$, $d$ of type 2) from Theorem \ref{an-class}. The corresponding open sets of 
quadrilateral orbits form a one-parametric family of integral surfaces of the distribution $\mcd_{ab}$. They saturate an open and dense subset in 
a three-dimensional analytic subset in $M_{ab}$. 
\end{remark}

%\begin{proof} Suppose the contrary:  the minimal analytic subset $M\subset\wt M_{ab}$ containing the union $\Sigma$ 
%of all the integral surfaces 
% is at least three-dimensional. The tangency locus of the variety $M$ with the Birkhoff distribution is an analytic 
%subset (eventually, with singularities deleted) that contains $\Sigma$. Hence, the tangency locus is all of $M$, and thus, 
%the latter is an invariant analytic subset of dimension at least three. The contradiction thus obtained to Lemma \ref{lnonint} proves the 
%corollary.
%\end{proof}
%Everywhere below by 
%$$\nu_{ab}:M_{ab}\to \ha\times\hb, \ \nu_G: M_{ab}\to\cp^2; \ G=C,D$$  we denote the natural projections 
%to the positions of the corresponding vertices: $(A,B)\in\ha\times\hb$ and $G\in\cp^2$ respectively. 
%Everywhere below for a 
%given 4-reflective component $U\subset\ha\times\hb\times\hc\times\hd$, by $U_0\subset U$ we denote 
%the open and dense subset consisting of 4-periodic orbits. 

Let $\Psi=\Psi_{ab}:\ha\times\hb\times\hc\times\hd\to \mcp_{ab}$ be the mapping  from (\ref{embed}). 

\begin{corollary} \label{cepi} Let $a$, $b$, $c$, $d$ be a 4-reflective complex planar analytic billiard, and let  
$U$ be the 4-reflective set. Then the image $\Psi(U)\subset\mcp_{ab}$ is a two-dimensional analytic subset lying in 
$M_{ab}$. The natural projection $U\to\ha\times\hb$ is a proper epimorphic mapping. 
\end{corollary}

\begin{proof} In the case, when both $a$, $b$ are algebraic, the curves $c$, $d$ are also algebraic (Proposition \ref{twoalg}), and the 
statements of the corollary follow immediately. Thus, without loss of generality we consider that some of the curves $a$, $b$ is not algebraic. 
It suffices to prove  the first statement of the corollary.  Then its second statement, which is equivalent to the 
properness and the epimorphicity of the analytic set projection $\nu_{ab}:\Psi(U)\to\ha\times\hb$, follows from the 
properness of the projection $\mcp_{ab}\to\ha\times\hb$ and 
Remmert's Proper Mapping Theorem \cite[p.34]{griff}. 
The image $\Psi(U)$ lies in $M_{ab}$, which follows from definition. Recall that $U_1\subset U$ denote the open and dense subset of 
non-degenerate orbits whose vertices are not marked points. 
Let $\mcs\subset\mcp_{ab}$ denote the minimal analytic subset containing $\Psi(U_1)$, which obviously contains $\Psi(U)$. 
Each its irreducible component is two-dimensional, as is $U_1$, by   Corollary  \ref{cnonint} 
and since $\Psi(U_1)$ is a union of integral surfaces of the distribution $\mcd_{ab}$, see  Proposition  \ref{birktang}.  
The projections  $\nu_C,\nu_D:\mcs\to\cp^2$ to the positions of 
the vertices $C$ and $D$ have rank one on an open dense subset, and  $\nu_C(\mcs)\subset c$, $\nu_D(\mcs)\subset d$: this holds
 on  $\Psi(U_1)$, and hence, on each irreducible component of the set $\mcs$.  
 Let $\hat\mcs$ denote the normalization of the analytic set $\mcs$, and $\pi_{\mcs}:\hat\mcs\to\mcs$ denote the natural projection. 
 The above projections lift to holomorphic mappings $\nu_{\hat g}:\hat\mcs\to\hat g$, $g=c,d$: 
 $\nu_{G}\circ\pi_{\mcs}=\pi_{g}\circ\nu_{\hat g}$ on $\hat\mcs$ (Corollary \ref{clift}). This yields an ``inverse'' mapping 
 $\Psi^{-1}=\nu_{ab}\times\nu_{\hat c}\times\nu_{\hat d}:\hat\mcs\to\ha\times\hb\times\hc\times\hd$. 
 Its image is contained in $U$, by its analyticity (Proposition \ref{comp-set}) and since  the image $U_1$ of the set $\Psi(U_1)$ (lifted to $\hat\mcs$) lies in $U$. 
 This together with the inclusion $\Psi(U)\subset\mcs$ implies that $\Psi(U)=\mcs$ and proves the corollary.
\end{proof}
%
%\begin{corollary} \label{onealg} Let in a 4-reflective planar analytic billiard one of the mirrors be algebraic. 
%Then the billiard is of one of the types given by Theorem \ref{an-class}. 
%\end{corollary}

\begin{corollary} \label{inters} Let $a$, $b$, $c$, $d$ be a 4-reflective planar analytic  billiard, and none of the mirrors $a$, $b$ be a line. 
Let  $a$ and $b$ intersect at a point $A$ 
represented by some non-marked points in  $\ha$ and $\hb$.  Then $a=c$ and $a\neq b$.
\end{corollary}

At the end of the section we prove Theorem \ref{onealg} and 
Corollary \ref{inters}. Both of them 
will be used further on in the proof of Theorem \ref{tallalg}. 

%\begin{remark} In Corollary \ref{inters} the billiard is as in classification Theorem \ref{an-class}, by Corollary \ref{onealg}. 
%\end{remark}

{\bf Plan of the proof of Lemma \ref{lnonint}.} Recall that $a$ and $b$ are not both lines. 
In the case, when they are both algebraic curves, there exist at most unique analytic curves $c$ and $d$ such that the billiard 
$a$, $b$, $c$, $d$ is 4-reflective, and if they exist, they are algebraic (Proposition \ref{twoalg} and Theorem \ref{an-class} in the algebraic case, 
see Remark \ref{remalg}). Thus, the only integral 
surfaces of the distribution $\mcd_{ab}$ are  given by the open set of its quadrilateral orbits, by Propositions \ref{birktang} and 
\ref{pnontriv}. Moreover, the latter orbit set is a Zariski open dense subset in a projective algebraic surface. 
This immediately implies the statement of Lemma \ref{lnonint}. 
Everywhere below we consider that some of the curves 
$a$ or $b$ is transcendental and prove the lemma  by contradiction. 
 Suppose the contrary to Lemma \ref{lnonint}: there exists a three- or four-dimensional irreducible analytic subset $M\subset\mcp_{ab}$ contained in 
$M_{ab}$ such that  $M\cap M^0_{ab}\neq\emptyset$ and 
the restriction $\mcd_M$ to $M^0=M\cap M_{ab}^0$ of the distribution $\mcd_{ab}$ is two-dimensional and integrable. (In the second case 
$M=M_{ab}$.) The complement 
\begin{equation} \Sigma^0=M\setminus M^0=M\setminus M^0_{ab}\subset M\label{sigmao}\end{equation}
is an analytic subset of positive codimension in $M$,   and $M^0$ is dense in $M$. 
Thus, every $x\in M^0$ is contained in an integral surface, and the latter is formed by quadrilateral orbits of a 4-reflective billiard 
$a$, $b$, $c(x)$, $d(x)$ (Propositions \ref{birktang} and \ref{pnontriv}). We show that there exists an $x\in M^0$ such that the corresponding mirrors $c(x)$, $d(x)$ are algebraic. This together with Proposition \ref{twoalg} implies that $a$ and $b$ are algebraic. 
The contradiction thus obtained will prove Lemma \ref{lnonint}. 

 For the proof of Lemma \ref{lnonint} we study the projections of the set $M$ to the positions 
of three vertices $(A,B,D)$: set
%and to the positions of the same vertices and the line $L_D$ 
$$ \nu_{ab,D}:x\mapsto(A(x), B(x), D(x)); \ M_D=\nu_{ab,D}(M)\subset\ha\times\hb\times\cp^2.$$
%\ \nu_{ab, D}^L: x\mapsto(A(x), B(x), D(x), L_D(x)),$$
%$$\nu_{ab,C}:x\mapsto(A(x), B(x), C(x)); \ \nu_{ab,C}^L: x\mapsto(A(x),B(x), C(x), L_C(x)),$$
%$$, \  M_{D,L_D}=\nu_{ab,D}^L(M)\subset\ha\times\hb\times\cp^2\times\cp^{2*},$$
%$$p_{L,D}: M_{D,L_D}\to M_D: \ (A,B,D,L_D)\mapsto(A,B,D).$$
Analogous projections and spaces are defined with $D$ replaced by $C$. 

  \begin{remark} \label{rproper} 
%  The projection $\nu_{ab}:\mcp_{ab}\to\ha\times\hb$ is proper, epimorphic and its preimages are naturally identified 
%  with projective varieties. This implies the same statements for its restriction $\nu_{ab}:M\to\ha\times\hb$
%and the above projections $\nu_{ab,D}$, $\nu^L_{ab,D}$, $p_{D,L_D}$. In more detail, the epimorphicity of  the restriction $\nu_{ab}|_M$ 
%follows from Remmert's Proper Mapping Theorem and the fact that the projection of each integral surface in $M$ is a local 
% diffeomorphism at its  generic point (the integral surfaces being two-dimensional and non-trivial, see Proposition \ref{pnontriv}). 
  The   images $M_C=\nu_{ab,C}(M)$, $M_D=\nu_{ab,D}(M)$ 
  % and $M_{D,L_D}=\nu^L_{ab,D}(M)$ 
  are irreducible analytic subsets in $\ha\times\hb\times\cp^2$, by  
  % and $\ha\times\hb\times\cp^2\times\cp^{2*}$ respectively, by 
  Remark \ref{rempro}, Remmert's Proper Mapping Theorem and irreducibility of the variety $M$. 
%Analogous statements hold for similarly defined sets $M_C$ and $M_{C,L_C}$. All the latter analytic subsets are 
%irreducible, as is $M$. 
\end{remark}

%For every $G=D,C$ the correspondence $\lambda:M_G\to\cp^{2*}$: $(A,B,G)\mapsto L_G$ induced by the inverse $p_{L,G}^{-1}$ will be 
%called the {\it line correspondence}. 
In Subsections 3.2 and 3.3 respectively we treat  the following cases:

- some of the projections $\nu_{ab,C}$, $\nu_{ab,D}$ is not bimeromorphic (see Footnote 3 in Subsection 2.6); 

- both latter projections are bimeromorphic.

\def\mcr{\mathcal R}

In what follows, we denote $\Sigma^1\subset M^0$ the subset of  points $x$ such that 

- either $x$ is a singular point of the variety $M$, 

- or it is a singular point of the distribution $\mcd_{M}$, 

- or the restriction to $\mcd_M(x)$ of the differential $d\nu_{ab}(x)$ has rank less than two, 

- or $x$ is a critical point of  the projection $\nu_{ab,D}$: a point where the rank of differential is not maximal, 

- or its image under the latter is a singularity of the image, 

- or the differential of the projection $\nu_D:M_{ab}^0\to\cp^2$: $y\mapsto D(y)$ vanishes on the distribution plane $\mcd_{M}(x)$,

- or one of the three latter statements holds with $D$ replaced by $C$. 

Let $\Sigma^0$ be the same, as in (\ref{sigmao}).  Set 
\begin{equation}\Sigma=\Sigma^0\cup\Sigma^1\subset M.\label{defsig}\end{equation}
This is an analytic subset in $\mcp_{ab}$ that has positive codimension in $M$. Its complement in $M$ 
is contained in $M^0$ and dense in $M$.  

\begin{remark} \label{lreg} For every point $x\in M\setminus\Sigma$ the corresponding germs $(g,G(x))$ of mirrors $g=a,b,c(x),d(x)$, $G=A,B,C,D$,  
are regular, and the points $G(x)$ are not isotropic tangency points. This  follows from the definition of the set 
$\Sigma^0\subset\Sigma$ (for the mirrors $a$ and $b$) and from the two last conditions in the definition of the set $\Sigma^1\subset\Sigma$ 
(for the mirrors $c(x)$ and $d(x)$). The projection $\nu_{ab}:S\to\ha\times\hb$ of each integral surface $S$ 
of the distribution $\mcd_M$ in $M\setminus\Sigma$ is 
a local diffeomorphism, by the Addendum to Proposition \ref{comp-set}. 
\end{remark}

\subsection{Case of a non-bimeromorphic projection}

Here we prove Lemma \ref{lnonint} in the case, when some of the projections $\nu_{ab,C}$, $\nu_{ab,D}$, say $\nu_{ab,D}$  is not bimeromorphic. 
Its proof is based on the following proposition.

\begin{proposition}  \label{bil4refl} Let $\Sigma$ be the same, as in (\ref{defsig}). For every two points $x,y\in M\setminus\Sigma$ 
projected to the same $(A,B)\in\ha\times\hb$ such that either 
$(C(x),L_C(x))\neq (C(y),L_C(y))$, or  $(D(x),L_D(x))\neq (D(y),L_D(y))$, the billiard $c(x)$, $d(x)$, $d(y)$, $c(y)$ is 4-reflective. 
\end{proposition}

\begin{proof} The proof of the proposition repeats the final argument from  \cite[proof of lemma 3.1]{alg}. 
%The quadrilaterals $ABC(z)D(z)$, $z=x,y$, are interior points in the set of quadrilateral orbits of the 
%corresponding billiards $a$, $b$, $c(z)$, $d(z)$. The points $A$, $B$ are not cusps of the 
%curves $a$, $b$, since $x,y\notin\Sigma$. 

Case (i):  $C(x)\neq C(y)$ and $D(x)\neq D(y)$. Each billiard $a$, $b$, $c(z)$, $d(z)$, $z=x,y$ has two-dimensional family of 
quadrilateral orbits $A'B'C_zD_z$ close to $ABC(z)D(z)$ where $C_z=C_z(A',B')$, $D_z=D_z(A',B')$ depend analytically 
on parameters $(A',B')\in\ha\times\hb$ (the Addendum to Proposition \ref{comp-set}). The 
corresponding quadrilaterals $C_xD_xD_yC_y$ are periodic orbits of the billiard 
$c(x)$, $d(x)$, $d(y)$, $c(y)$, by definition and reflection law, see Fig.\ref{fig-cddc}, and depend analytically on parameters $(A',B')$. 
Therefore,  the  billiard $c(x)$, $d(x)$, $d(y)$, $c(y)$ is 4-reflective. 

 \begin{figure}[ht]
  \begin{center}
 % \vspace{-0.3cm}
   \epsfig{file=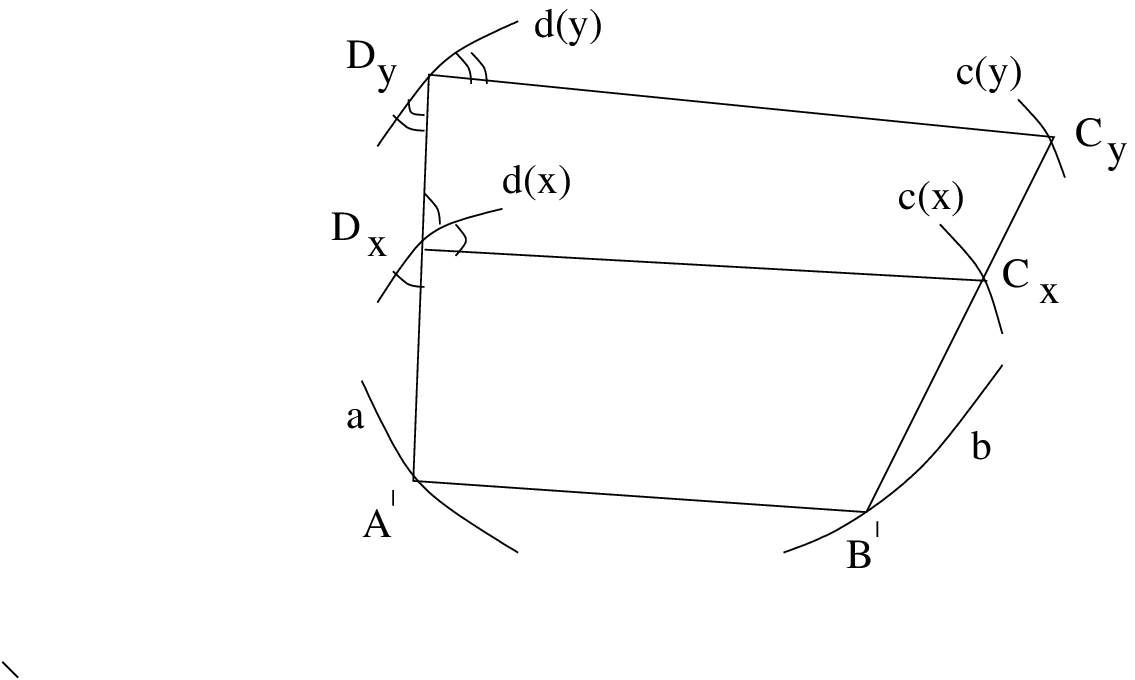}
 %  \vspace{-0.3cm}
    \caption{The 4-reflective billiard $c(x)$, $d(x)$, $d(y)$, $c(y)$: open set of quadrilateral orbits $C_xD_xD_yC_y$.}
  \label{fig-cddc}
  \end{center}
\end{figure}

Case (ii): $D(x)=D(y)=D_0$ but $L_D(x)=T_{D(x)}d(x)\neq L_D(y)=T_{D(y)}d(y)$ (the same case with $D$ replaced by $C$ is symmetric).
Let us show that one can achieve the inequalities of 
Case (i) by deforming $x$ and $y$ as orbits of fixed billiards. The mirrors $d(x)$ and $d(y)$ intersect transversely at $D_0$. 
Therefore, deforming $(A,B)$, one can achieve that the line $l_a=l_a(A,B)$ intersects 
$d(x)$ and $d(y)$ at two distinct points close to $D_0$. This lifts to deformation of $x$ and $y$ along their integral surfaces (the 
last statement of Remark \ref{lreg}), and we get new $x$, $y$ with $\nu_{ab}(x)=\nu_{ab}(y)$ and $D(x)\neq D(y)$. 
One has $(C(x),L_C(x))\neq(C(y),L_C(y))$: otherwise, one would have $l=C(x)D(x)=C(y)D(y)$, hence $D(x)=D(y)=l_a\cap l$, -- a contradiction. 
Hence, one can achieve that $C(x)\neq C(y)$ via small deformation, by the above argument. 
 %Thus, we get new $x$, $y$ with $\nu_{ab}(x)=\nu_{ab}(y)$, $D(x)\neq D(y)$ and $C(x)\neq C(y)$. 
This reduces us to Case (i) and proves the proposition.  
\end{proof}

\begin{proposition} \label{dense-mult}  Let the projection $\pi_{ab,D}:M\to M_D$ be not bimeromorphic. Then up to interchanging $C$ and $D$, 
there exists an open and dense subset of points $z\in M_D$ for which there exist  
$$x,y\in M\setminus\Sigma \text{ with }  \nu_{ab,D}(x)=\nu_{ab,D}(y)=z, \ L_D(x)\neq L_D(y)$$
% Their existence  follows from definition. Indeed, in Case 1) of Corollary \ref{line-cases} take arbitrary $(A,B,D)\in M_D\setminus\Sigma_D$, 
% where  $\Sigma_D$ is the same, as in (\ref{defsigl}). The preimage 
%$\nu_{ab,D}^{-1}(A,B,D)\subset M$ is a curve that is not entirely contained in $\Sigma$, and the mapping $x\mapsto L_D(x)$ is non-constant 
%on its complement to $\Sigma$. Then we can take arbitrary $x,y$ in the latter complement that correspond to distinct lines 
%$L_D$. In Case 2a) of the corollary we take arbitrary point $(A,B,D)\in M_D^0=M_D\setminus\sigma_D$, where $\sigma_D$ is the same, as in 
%(\ref{defsigsig}). It has at least two 
%preimages $x,y\in M$ corresponding to distinct lines $L_D$, and both of them lie outside $\Sigma$ by definition. 
such that the projection $\nu_{ab,D}$ is a local submersion at both $x$ and $y$. 
\end{proposition} 
%Moreover, the set of those $x\in M\setminus\Sigma$ 
%for which the corresponding above $y$ exists is open and dense. 

\begin{proof} The non-bimeromorphicity implies that there exists an open and dense subset of points $z=(A,B,D)\in M_D$ for which there exist 
$x,y\in \pi_{ab,D}^{-1}(z)\setminus\Sigma$, $x\neq y$ such that the projection $\pi_{ab,D}$ is a local submersion at both $x$ and $y$. 
In the case, when $L_D(x)\neq L_D(y)$ for at least some $x$, $y$, $z$ as above, the statement of the proposition follows immediately. 
Now suppose the contrary: there exists an open and dense subset of points $z\in M_D$ such that for every 
$x,y\in \pi_{ab,D}^{-1}(z)\setminus\Sigma$, $x\neq y$ one has $L_D(x)=L_D(y)=L_D$. Then $C(x)=C(y)$, since the point $C(x)=C(y)$ is found 
as the intersection point of two lines depending only on $z$ and not on $x$ or $ y$: the line $BC$ symmetric to $AB$ with respect to the line 
$T_Bb$; the line $DC$ symmetric to $AD$ with respect to the line $L_D$. Therefore, $\pi_{ab,C}(x)=\pi_{ab,C}(y)$, but 
$L_C(x)\neq L_C(y)$, since $x\neq y$. Hence, the statement of the proposition holds with $D$ replaced by $C$. The proposition is proved.
\end{proof} 

Let us fix arbitrary $x$, $y$ as in Proposition \ref{dense-mult}. Set  

$$D_0=D(x)=D(y), \ (A,B)=\nu_{ab}(x)=\nu_{ab}(y).$$

{\bf Claim 1.} {\it   The curves $c(x)$, $c(y)$ are either both triangular spirals centered at $D_0$, or both conics: complex circles 
centered at $D_0$.} 

\begin{proof} The germ 
at $C(x)D_0D_0C(y)$ of  billiard $c(x)$, $d(x)$, $d(y)$, $c(y)$ is 4-reflective (Proposition \ref{bil4refl}). 
%The degenerate quadrilateral $C(x)D_0D_0C(y)$ 
It satisfies the non-isotropicity and line non-coincidence 
conditions of Proposition \ref{intspir}, since $x$, $y$ correspond to non-degenerate quadrilaterals. For example, 
the tangent line to a mirror through each  vertex is non-isotropic and distinct from the adjacent edges, by non-degeneracy. 
This together with Proposition \ref{intspir}  implies the claim.
\end{proof}

\begin{corollary} \label{coralg} The mirrors $c(z)$ and $d(z)$ are algebraic for $z=x,y$.
 \end{corollary}
 
 \begin{proof} It suffices to prove that 
 the mirrors $c(z)$, $z=x,y$, are both algebraic: then so are $d(z)$, by 4-reflectivity of the billiard 
 $c(x)$, $d(x)$, $d(y)$, $c(y)$ and Proposition  \ref{twoalg}.
 Suppose the contrary: say $c(x)$ is not algebraic. Take an arbitrary $y'\in M\setminus\Sigma$ close to $y$ with 
 %$\nu_{ab}(y')=(A,B)=\nu_{ab}(y)$ and $D(y')\neq D_0$. 
$D_1=D(y')\in d(x)$, $D_1\neq D_0$, $L_D(y')=T_{D_1}d(y')\neq T_{D_1}d(x)$. 
 It exists, since the projection $\nu_{ab,D}$ is a local submersion at $y$ and $L_D(y)\neq L_D(x)$. Deforming $x$ along its integral 
 surface of the distribution $\mcd_M$ one can achieve that $(A,B)=\nu_{ab}(x)=\nu_{ab}(y')$ (Remark \ref{lreg}), and 
 then $D(x)=D_1$.
% (Proposition \ref{dnonc}). 
% The mirror $d(y')$ intersects $d(x)$ at a point $D_1\neq D_0$ close to $D_0$. This follows from the inclusion $D(y')\in AD_0\setminus D_0$,  
% pairwise transversality of the lines $AD_0$, $L_D(x)=T_{D_0}d(x)$, $L_D(y)=T_{D_0}d(y)$ and Remark \ref{lreg}. 
 The billiard $c(x)$, $d(x)$, $d(y')$, $c(y')$ is 
 4-reflective, by Proposition \ref{bil4refl}, and its 4-reflective set contains a degenerate quadrilateral  
 $C_1D_1D_1C_2$ close to $C(x)D_0D_0C(y)$. This together with 
 Proposition \ref{intspir} implies that $c(x)$ is a triangular spiral centered at $D_1$, as in Claim 1. Thus, $c(x)$ is a triangular 
 spiral with two distinct centers $D_0$ and $D_1$. Hence, it is algebraic, by  
 Proposition \ref{pnonalg}. The contradiction thus obtained proves the corollary.
 \end{proof}

 Thus, the 4-reflective billiard $a$, $b$, $c(x)$, $d(x)$ has two neighbor algebraic mirrors  $c(x)$ and $d(x)$. Hence, 
 $a$ and $b$ are also algebraic, by Proposition \ref{twoalg}.  This contradicts the non-algebraicity assumption and 
 proves Lemma \ref{lnonint}.

 \subsection{Case of bimeromorphic projections}
 
 Here we prove Lemma \ref{lnonint} in the case, when both projections $\pi_{ab,C}$, $\pi_{ab,D}$ are bimeromorphic. 
 To do this, we show that for every $x\in M\setminus\Sigma$ both mirrors $d(x)$ and $c(x)$ are lines. We then get a contradiction as above. 
 %This will prove Lemma \ref{lnonint}.  
% consider that Case 2b) of Corollary \ref{line-cases} holds for both $G=D,C$: 
% $dim M=dim M_{G,L_G}=dim M_G=3$, and the inverse $p_{L,G}^{-1}: M_G\to  M_{G,L_G}$ is holomorphic  on $M_G^0=M_G\setminus
% \sigma_G$, see  (\ref{defsigsig}).  Thus, the line correspondence 
% $$\lambda:M_G\to\cp^{2*}:(A,B,G)\mapsto L_G,$$ 
% which is the composition 
% of the latter inverse with the projection to the position of the line $L_G$, is also holomorphic on $M_G^0$. 

 It suffices to prove the above statement for the mirrors $d(x)$ only.   
 We first show (the next proposition) that the lines $L_D(x)\subset T_{D(x)}\cp^2$ locally depend 
 only on $D(x)$ and form a holomorphic line field. Afterwards we show (Proposition \ref{ppencil}) that the restriction of the latter line 
 field to each projective line is tangent to a pencil of lines through the same point.
  This easily implies that its integral curves are lines (Proposition \ref{proplines}). 
 
% We then deduce (the next corollary) that for every $(A,B)$ from an open and dense subset $\mcr\subset\ha\times\hb$ the lines 
% $L_D=\lambda(A',B',D)\subset T_D\cp^2$ 
% with $(A',B')$ close to $(A,B)$ induce a holomorphic line field on a neighborhood of the projective line 
% $l_a=l_a(A,B)$. Afterwards we show that the latter line field is transverse to $l_a$ (Proposition \ref{ptransv}). This implies 
% (Lemma \ref{pencil}) that its phase curves, which coincide with the mirrors $d(x)$, form a pencil of lines through the same point.  

There exists an open and dense subset of  points $x\in M\setminus\Sigma$ where the germ of projection $\nu_{D}:y\mapsto D(y)$ 
is a submersion. Indeed, the contrary would imply that $d=\nu_D(M\setminus\Sigma)$ is a curve. Therefore, $D(y)$ is locally  
determined by $(A,B)=(A(y),B(y))$ as a point of intersection $l_a(A,B)\cap d$, hence $dim M_D=2$. But 
$dim M_D=dim M\geq3$, by assumption and bimeromorphicity, -- a contradiction. Note that $dim M_D=3=dim M$, 
by the latter inequality and since $D\in l_a(A,B)$ for every $(A,B,D)\in M_D$. 
Increasing the "exceptional set" $\Sigma$, we will assume that $\nu_{D}$ is a submersion on 
all of $M\setminus\Sigma$. Its level sets in $M\setminus\Sigma$ are thus holomorphic curves.  
 \begin{proposition} \label{propla} Every $x\in M\setminus\Sigma$ has  a neighborhood $V=V(x)\subset M\setminus\Sigma$ such that the lines 
 $L_D(y)$, $y\in V$ depend only on $D(y)$ and thus, form a holomorphic line field $\Lambda_D$ on a neighborhood $W\subset\cp^2$ 
 of the point $D(x)$. 
 \end{proposition}
 
 \begin{proof} Let $V$ be a neighborhood of the point $x$ regularly fibered by local level curves of the submersion $\nu_D$. 
 For every $y\in V$ there exists a holomorphic curve  $\gamma(y)\subset V$ that corresponds exactly to the family of 
 quadrilateral orbits of the billiard 
 $a$, $b$, $c(y)$, $d(y)$ with fixed vertex $D=D(y)$. The line $L_D=L_D(y)= T_{D(y)}d(y)$ obviously remains constant along the curve $\gamma(y)$. 
 On the other hand, the curve $\gamma(y)$ should obviously coincide with the level curve $\{\nu_D=D(y)\}$. 
 This proves the proposition.
 \end{proof}
  
  \begin{proposition} \label{ppencil}  The restriction to each projective line of the field $\Lambda_D$ from Proposition \ref{propla} 
  is tangent to a pencil of lines through the same point.
  \end{proposition}
  
  \begin{proof} For every $(A,B)\in\ha\times\hb$ denote by $F^D_{A,B}=A\times B\times l_a(A,B)\simeq\oc$ the fiber over $(A,B)$ of the set $M_D$. 
  The fiber $F^C_{A,B}$ of the set $M_C$ is defined analogously. The projections $\nu_{ab,G}$, $G=C,D$ and 
  the mappings $\nu_{ab,D}^{-1}:M_D\to M$, $R_{DC}=\nu_{ab,C}\circ\nu_{ab,D}^{-1}:M_D\to M_C$ 
  are bimeromorphic. Hence, the  indeterminacies of the mapping $\nu_{ab,D}^{-1}$ (which include those of $R_{DC}$) 
  form an analytic subset of codimension at least two,  
  i.e., dimension at most one (Footnote 3 in Subsection 2.6). The image of its projection to $\ha\times\hb$ is an analytic subset $Ind$ 
  of dimension at most one (Proper Mapping Theorem). Let $Reg\subset\ha\times\hb\setminus Ind$
  denote the open and dense subset of those $(A,B)$ for which 
  $A\neq B$ and the lines $l_a(A,B)$, $l_b(A,B)$, $T_Aa$, $T_Bb$ are not isotropic and distinct. Then for every 
  $(A,B)\in Reg$ the following statements hold: 
  
  - there exists an invertible holomorphic (hence M\"obius) mapping $T_{A,B}:l_a(A,B)\to l_b(A,B)$ for which 
  \begin{equation} R_{DC}(A,B,D)\equiv(A,B,T_{A,B}(D));\label{rcd}\end{equation}

- the restriction to $l_a(A,B)$ of the line field $\Lambda_D$ is holomorphic   and hence, the projective lines 
$\Lambda_D(y)$, $y\in l_a(A,B)$, form a rational curve $\Gamma_{A,B}\subset\cp^{2*}$. 
  
  The former statement follows from bimeromorphicity by definition. The global holomorphic family of lines $\Lambda_D(y)$ 
  depending on $y\in l_a(A,B)\simeq F^D_{A,B}\subset M_D$ from the 
  latter statement is given by the holomorphic inverse $\nu_{ab,D}^{-1}:F^D_{A,B}\to M$: 
  $y\mapsto(A,B,(C(y),\Lambda_C(y)), (D(y),\Lambda_D(y)))$.
  
%   include the tuple 
%  Let us prove the latter one. Fix  some  $(A,B)\in Reg$ and consider the 
%  family of lines  
%    
%    The line field  $\Lambda_D$ extends analytically as a family of symmetry lines of line pairs $(l_a,\lambda(D))$ along paths in $l_a$ 
%    that avoid a finite number of those 
%    points $D\in l_a$ for which either $D=T_{A,B}(D)$, or the line $\lambda(D)$ is isotropic. Hence, $\Lambda_D$ is 
%   at most double-valued and algebraic on $l_a$. It is single-valued, since the contrary would imply that 
%  the inverse $\nu_{ab,D}^{-1}$ is multivalued on the fiber $F^D_{A,B}\subset M_D$ with $(A,B)\in Reg$, -- a contradiction. 
%  Erasing isolated algebraic type singularities yields a holomorphic mapping $\Lambda_D:l_a\to\cp^{2*}$. 
  
       It suffices to prove the statement of Proposition \ref{ppencil} 
       for an open set of projective lines, e.g., for all the lines $l_a(A,B)$, $(A,B)\in Reg$. 
  Suppose the contrary: for some $(A,B)\in Reg$ the restriction to $l_a=l_a(A,B)$ of the line field $\Lambda_D$ 
  is not tangent to a pencil of lines, i.e., the curve $\Gamma=\Gamma_{A,B}$ has degree at least two. Note that  the family of lines 
  $$\lambda(D)=DT_{A,B}(D), \ D\in l_a=l_a(A,B)$$
is a rational curve  $\lambda\subset\cp^{2*}$ parametrized by $D\in l_a$. 
  
  {\bf Claim 1.} {\it The curve $\lambda$ has degree at least three.}
  
  \begin{proof} Let $Q$ denote the intersection point of the infinity line $\oc_{\infty}\subset\cp^2$ with non-isotropic line $l_a$; thus, 
   $Q\neq I_{1,2}$.  Let us take  a  point 
  $P\in\oc_{\infty}\setminus\{ Q, I_1,I_2\}$ and a finite line $l$ through $P$.  There are two symmetry lines of the pair of lines 
  $l_a$, $l$. They intersect the infinity line at distinct points  $E_1, E_2\in\oc_{\infty}$ depending only on $P$: $E_j=E_j(P)$. For 
  a  generic $P\in\oc_{\infty}$ for every $j=1,2$ there are at least two distinct points $D_j,D_j'\in l_a$ such that the projective lines  
  $\Lambda_D(D_j)$, $\Lambda_D(D_j')$ pass through $E_j=E_j(P)$, since $deg \Gamma\geq2$. Either the four points $D_1$, $D_1'$, 
  $D_2$, $D_2'$ are distinct, or at most two of them coincide: the latter happens exactly in the case, when,
  $\Lambda_D(Q)=\oc_{\infty}\supset\{ E_1,E_2\}$; then $D_1=D_2=Q$. 
  This implies that there are always at least three distinct points $D^1,D^2,D^3\in l_a$ 
  such that each line $\Lambda_D(D^i)$, $i=1,2,3$, passes through some of the points $E_j$; thus, $P\in\lambda(D^i)$. 
  The corresponding points 
  $\lambda(D^i)\in\cp^{2*}$ of the curve $\lambda$  lie in the projective line $P^*$, by construction. This proves the claim. 
  \end{proof}
  
  {\bf Claim 2.} {\it The curve $\lambda$ has degree at most two.}
  
  \begin{proof} Let $X$ denote the point of intersection of the lines $l_a$ and $l_b$, $Y=T_{A,B}(X)\in l_b$. For every point $C\in l_b\setminus \{ Y\}$ 
  the curve $\lambda$ intersects the dual line $C^*$ in at most two points with multiplicity 1: at the points $\lambda(T_{A,B}^{-1}(C))$ and 
  may be  $\lambda(X)$ (if $\lambda(X)=l_b$). The fact that in the latter case the multiplicity of the intersection point $\lambda(X)$ 
  equals one follows from the assumption that $C\neq Y$: the lines $\lambda(D)$ with $D$ close to $X$ asymptotically focus at 
  $Y\neq C$ and are transverse to $l_a$.   This proves the claim.
  \end{proof}
  
  Claims 1 and 2 contradict each other. This proves Proposition \ref{ppencil}.
  \end{proof}
  
  \begin{proposition} \label{proplines} The phase curves of the line field $\Lambda_D$ form a pencil of projective lines through the same point.
  \end{proposition}
  
  \begin{proof} Suppose the contrary: some phase curve $S$ is not a line. Consider a projective line $L=\Lambda_D(P)$ 
  tangent to $S$ at some point $P$. Then for every $Q,R\in L$, $Q\neq R$, close to $P$ and distinct from it  the projective 
  lines $\Lambda_D(Q)$ and $\Lambda_D(R)$ are distinct from the line $L=\Lambda_D(P)$ and intersect each other at a point outside the line $L$. 
  Therefore, the lines of the restriction to $L$ of the field $\Lambda_D$  are not tangent to a pencil of lines through the same point.
  The contradiction thus obtained to Proposition \ref{ppencil} proves Proposition \ref{proplines}. 
\end{proof}

Proposition \ref{proplines} applied to both $\Lambda_C$ and $\Lambda_D$  implies that for every $x\in M\setminus\Sigma$ the mirrors 
$c(x)$ and $d(x)$ are lines. Hence, $a$ and $b$ are algebraic (Proposition \ref{twoalg}), -- a contradiction.  Lemma \ref{lnonint} is proved.

\subsection{Case of one algebraic mirror. Proof of Theorem \ref{onealg}}

In the present subsection we consider that  $a$, $b$, $c$,  $d$ is a 4-reflective analytic planar billiard, and the mirror $a$ is algebraic.
 Without loss of generality we consider that the curves $b$ and $d$ are transcendental: in the opposite case Theorem \ref{onealg} follows 
 immediately from Proposition \ref{twoalg} and \cite[theorem 1.11]{alg}. 
As it is shown below, Theorem \ref{onealg} is implied by the following proposition.

\begin{proposition}  In the above conditions the mirror $c$ is also algebraic. 
\end{proposition}

\begin{proof} The projection $\nu_b:U\to\hb$ of the 4-reflective set $U$ is proper and epimorphic, 
by Corollary \ref{cepi} and since $a$ is algebraic. This implies that  for 
an open and dense set of points $B\in \hb$ the preimage $\nu_b^{-1}(B)\subset U$ is a compact analytic curve with non-constant holomorphic 
projection to $\hc$. Hence, $\hc$ is compact and $c$ is algebraic.  The proposition is proved.
 \end{proof}
 
 Now let us prove Theorem \ref{onealg}.  Fix a non-marked point 
 $B\in\hb$  as in the above proof and a one-dimensional irreducible component 
 $\Gamma_B$ of the compact analytic curve $\nu_b^{-1}(B)\subset U$. Its image $\nu_D(\Gamma_B)\subset d\subset\cp^2$ 
is either an algebraic curve, or a single point (Proper Mapping and Chow Theorems). 
The former case is impossible, since $d$ is non-algebraic. 
Hence, for an open and dense set of points $B\in\hb$ the projection of the curve $\Gamma_B$ to the 
 position of the vertex $D$ is constant and is determined by $B$. Thus, there exists a mapping $\hb\to\hd$, $B\mapsto D_B$, 
 defined on an open set $V\subset \hb$ such that for every fixed $B\in V$ and variable $A\in\ha$ the lines $AB$ and $AD_B$ are 
 symmetric with respect to the tangent line $T_Aa$. This implies that either $a$ is a line and $B$, $D_B$ are symmetric with 
 respect to $a$ for every $B\in\hb$, or $a$ is a conic with one-dimensional family of foci pairs $(B,D_B)$, see \cite[proposition 2.32]{alg}. 
 The latter case being  obviously 
 impossible,  the curves $b$, $d$ are symmetric with respect to the line $a$. Applying the above 
 argument to the algebraic mirror $c$ instead of $a$, we get that $B$ and $D_B$ are symmetric with respect to the line $c$. Thus, the above pairs $(B,D_B)$ are symmetric with respect to both lines $a$ and $c$, by construction. 
 Therefore, $a=c\neq b,d$, and the billiard is of type 1) from 
 Theorem \ref{an-class}. Theorem \ref{onealg} is proved. 
 
 \subsection{Intersected neighbor mirrors. Proof of Corollary \ref{inters}}
 
 In the conditions of Corollary \ref{inters} no mirror is a line, by Theorem \ref{onealg} and since $a$, $b$ are not lines.  
 Without loss of generality we consider that each mirror is transcendental, since otherwise,  $a=c$, by Theorem \ref{onealg}. 
Let $U$ be the 4-reflective set. Its projection to $\ha\times \hb$ is proper and epimorphic, by Corollary 
\ref{cepi}. 

\medskip

{\bf Claim 1.} {\it $a\neq b$.}

\begin{proof} Suppose the contrary: $a=b$. Then $U$ contains a one-parametric analytic family $\mct$ of quadrilaterals $AACD$ with variable 
$A$, $C$, $D$, by the above epimorphicity statement. A generic quadrilateral in $\mct$ is forbidden by Proposition \ref{neighbmir}. 
The contradiction thus obtained proves the claim. 
\end{proof} 

The projection preimage  in $U$ of the pair $(A,A)\in\ha\times\hb$ is a non-empty compact 
analytic subset $\Gamma\subset U$ of dimension at most one. 

Case 1): $dim \Gamma=1$. Then at least one of the curves $\hc$, $\hd$, say $\hc$ 
is  a compact Riemann surface, -- a contradiction to our non-algebraicity assumption. Thus, this case is impossible. 

Case 2): $dim \Gamma=0$: $\Gamma$ is a finite set. 

\medskip
{\bf Claim 2.} {\it Every quadrilateral $AACD\in\Gamma$ is a single-point quadrilateral: the mirrors $c$ and $d$ pass through the same point 
$A$; $\pi_c(C)=\pi_d(D)=\pi_a(A)$.}

\begin{proof} Suppose the contrary: say, $\pi_c(C)\neq \pi_a(A)$. The projection $U\to\ha\times\hb$ is open on a neighborhood of the point 
$AACD$, since it contracts no curve to $(A,A)$ by assumption. Therefore, each converging 
sequence $(A^k,B^k)\to(A,A)$ lifts to a converging sequence $A^kB^kC^kD^k\to AACD$ in $U$. Let us take two  sequences 
$(A^k_j,B^k_j)\to(A,A)$, 
$j=1,2$, with lines  $A^k_jB^k_j$ converging to different limits for $j=1,2$; this is possible, since $a\neq b$. 
We get two sequences of quadrilaterals 
$A^k_jB^k_jC^k_jD^k_j$ converging to the same quadrilateral $AACD$. On the other hand, the lines $B^k_jC^k_j$ 
symmetric to $A_j^kB_j^k$ with respect to 
the tangent lines $T_{B_j^k}b$ converge to two distinct limits $H_j$, $j=1,2$, by assumption and since $A$ is not 
a marked point of the curve $b$. The   lines $H_1\neq H_2$ pass through the same two  points $\pi_a(A)\neq \pi_c(C)$, by 
construction.  The contradiction thus obtained proves the claim. 
\end{proof} 

Thus, $\Gamma$ is a finite set of points corresponding to the single-point quadrilateral $AAAA$. Fix one of them and denote it 
$AAAA$: the corresponding vertices $C\in\pi_c^{-1}(A)$ and $D\in\pi_d^{-1}(A)$ will be denoted by $A$. One has $c\neq d$, as in Claim 1. 
 The projection $U\to\hc\times\hd$
is open on a neighborhood of the point $AAAA$, as in the above discussion. 
Let $\gamma\subset \hc\times\hd$ be an irreducible germ at $(A,A)$ of analytic curve consisting of 
pairs $(C',D')$ with variable $C'$ and 
$D'$ for which $C'\in T_{D'}d$. The germ $\gamma$ lifts to an irreducible germ $\wt\gamma$ of analytic 
curve through $AAAA$ in $U$. The curve  $\wt\gamma$ consists of quadrilaterals $A'B'C'D'\in U$ for which $B',D'\not\equiv C'$ 
(Proposition \ref{neighbmir}). Therefore,  $A'\equiv C'$, by 
 Corollary \ref{connect} and since $A$ is not a marked point of the mirror $b$. Hence, $a=c$. Corollary \ref{inters} is proved.

\section{Algebraicity: proof of Theorem \ref{tallalg}}

Theorem \ref{an-class}, and thus, Theorem \ref{tallalg} are already proved in the case, when at least one mirror is algebraic (Theorem  
\ref{onealg}). Here we prove  Theorem \ref{tallalg} in the general case by contradiction. 
Suppose the contrary: there exists a 4-reflective billiard $a$, $b$, $c$, $d$ with no algebraic mirrors. 
We  study Birkhoff distribution $\mcd_a$ on the space $M_a$ and consider its non-trivial integral surface $S$ formed by a connected 
open set of quadrilateral orbits of the billiard. Recall that  $M_a\subset\mcp_a$ is a six-dimensional analytic subset, and $\mcd_a$ is a 
singular three-dimensional distribution on $M_a$, see Subsection 2.7.  Set 
$$M=\text{ the minimal analytic subset in } M_a \text{ containing } S.$$
This is an irreducible analytic subset in $\mcp_a$, by definition, see Convention \ref{convan}. The intersections 
$$\mcd_M(x)=\mcd_a(x)\cap T_xM, \ x \text{ being a smooth point of the variety } M,$$
induce a singular analytic distribution $\mcd_M$ on $M$, for which $S$ is an integral surface. 
This is either two- or three-dimensional distribution, since $dim S=2$ and $dim\mcd_a=3$. 
The cases, when $dim \mcd_M=2, 3$, will be treated separately in Subsections 4.1 and 4.3 respectively. 

% by using Cartan--Kuranishi--Rashevsky involutivity theory of Pfaffian systems; the corresponding 
%background material will be recalled in Subsection 4.2. 
%
%
%We study the corresponding Pfaffian system: the problem  to find two-dimensional integral surfaces of the distribution $\mcd_M$. 
%The cases, when either $dim\mcd_M=2$, or $dim\mcd_M=3$ and the Pfaffian system is non-involutive in the 
% sense of Cartan--Kuranishi--Rashevsky theory are treated in Subsection 4.2.   The case, when $dim\mcd_M=3$ and the Pfaffian system is involutive is treated in Subsection 4.3. The methods of proof of Theorem \ref{tallalg} in both cases are similar. 
 
 The methods of proof in both cases are similar. We show that 
  an open set of  points $x\in M$ lie in integral surfaces corresponding to 4-reflective 
 billiards $a$, $b(x)$, $c(x)$, $d(x)$ with regularly intersected mirrors 
 $a$ and $b(x)$. Then  either $b(x)$ is a line, or $c(x)=a$, by  Corollary \ref{inters}. We then show that either 
 $\nu_C(M)\subset a$, or the mirror $b$ of the initial billiard is a line. We get a contradiction in both subcases.  In the case, when 
 $dim\mcd_M=3$, the proof uses Cartan--Kuranishi--Rashevskii involutivity theory of Pfaffian systems. 
 The corresponding background material will be recalled in Subsection 4.2. 

 The proof of the existence of the above integral surfaces is based on the following key proposition and corollary. 
 They deal with the natural projection 
 $\nu_a:\mcp_a\to\hat a$ and its restriction to $M$, which are proper and epimorphic. For every $A\in\ha$ the projection preimage 
 $$W_A=\nu_a^{-1}(A)\cap M$$
  is a projective algebraic set.

\begin{proposition} \label{projb} There exists a complement $\ha_0\subset\ha$ to a discrete subset in $\ha$ such that for every $A\in\ha_0$ 
the projection $\nu_B:W_A\to\cp^2$ is epimorphic. 
% and  has rank two on a non-empty Zariski open and dense subset in $W_A$. 
\end{proposition}

\begin{proof} For every $A\in\ha$ the image of the projection $\nu_B:W_A\to\cp^2$ is either the whole projective plane, or an 
algebraic subset of dimension at most one (Remmert's Proper Mapping and Chow's Theorems). Either $\nu_B(W_A)=\cp^2$ 
 for all but a discrete set of points $A$ (and then 
the statement of the proposition obviously holds), or it is at most one-dimensional algebraic set for an open and dense set $Q$ of 
points  $A\in a$, by analyticity. 
The latter case cannot happen, since otherwise for every non-marked $A\in\nu_a(S)\cap Q$  the 
set  $\nu_B(W_A\cap S)$ would be an open subset in the non-algebraic curve $b$ that simultaneously 
lies in at most one-dimensional algebraic set $\nu_B(W_A)$, -- a contradiction. This proves  the proposition. 
\end{proof}

\begin{corollary} \label{corx} For every open dense subset $N\subset M$ whose complement $M\setminus N\subset M$ is an analytic subset 
there exists an open  dense subset $\hat a_N\subset\hat a$ such that for every $A\in\ha_N$ the intersection $\nu_B(W_A\cap N)\cap a$ 
contains  the $\pi_a$- image of an open dense subset in $\ha$: a complement to a discrete subset. 
\end{corollary}
 
 \begin{proof} Let $\ha_0$ be the same, as in Proposition \ref{projb}. There exists an open dense subset $\ha_N\subset\ha_0$ such that for every $A\in\ha_N$ the subset  $W_A^N=W_A\cap N\subset W_A$ is open and dense, 
 since the complement $M\setminus N$ is an analytic subset and all the $W_A$ are 
 hypersurfaces. Then for the same $A$ 
 the subset $W_A\setminus W_A^N\subset W_A$ 
 is  algebraic, since it is analytic (as is $M\setminus N$)  and by Chow's Theorem. Therefore, the complement 
 $\cp^2\setminus\nu_B(W_A^N)$ is contained in an algebraic subset in $\cp^2$ of positive codimension (Proposition \ref{projb} and 
 Chevalley--Remmert and Chow's Theorems). Its $\pi_a$-preimage is at most discrete, since $a$ is non-algebraic. This implies the statement 
 of the corollary.
% Therefore, the subset 
% $\ha\setminus\pi_a^{-1}(\nu_B(W_A^N))\subset\ha$ is at most discrete, since $a$ is non-algebraic. This proves the corollary.
 \end{proof}
\def\mcw{\mathcal W}

\subsection{Case of two-dimensional distribution}
In the present section we consider that $dim\mcd_M=2$. Then the distribution $\mcd_M$ is integrable, by Proposition \ref{anint}. 

We keep the previous notations $\nu_a$, $\nu_G$, $G=B,C,D$ for the projections to the positions of vertices. Let 
$M^0_{reg}\subset M\cap M^0_a$ denote the subset of points regular for both $M$ and $\mcd_M$, cf. Definition \ref{moreg}. 
Set
\begin{equation} M'=\{ x\in M^0_{reg} \ | \ 
d\nu_a(x),d\nu_G(x)\not\equiv0 \text{ on } \mcd_M(x) \text{ for } G=B,C,D\}.
\label{defmo}\end{equation} 
%A(x)B(x)C(x)D(x) \text{ is non-degenerate and }$$
%denote the subset of those points $x$ for which the corresponding quadrilaterals $ABCD$ 
%are non-degenerate, and the restrictions to $\mcd_M(x)$ 
%of the differentials of the latter projections are non-zero. 
The set $M'$ contains  the integral surface $S$, since $S$ is non-trivial. 
It is  an open and dense subset  in $M$, and its complement 
$\Sigma=M\setminus M'$ is analytic. 

\begin{remark} \label{rems}  The integral surface of the distribution $\mcd_M$ through each point $x\in M'$ is non-trivial, 
by (\ref{defmo}). The germ of its image 
under each one of the projections $\nu_G$, $G=B,C,D$, is a germ of analytic curve at its non-marked point. The regularity of germ follows by 
definition from the inequalities in (\ref{defmo}). The non-isotropicity of tangent line to the germ follows from the 
definition of the distribution $\mcd_M$ and non-degeneracy of the quadrilateral corresponding to $x$. Thus, the above integral surface corresponds to 
an open set of quadrilateral orbits of a 4-reflective billiard $a$, $b(x)$, $c(x)$, $d(x)$, and the above germs are  germs of mirrors at 
non-marked points. 
\end{remark}

For every $A\in\ha$ set $W_A^0=W_A\cap M'$. 

\begin{proposition} \label{prop-inters} There exists an open subset $V\subset M'$ of those $x$ for which the mirrors $a$ and $b(x)$ intersect at some  point non-marked for both their local branches.
\end{proposition}
\begin{proof} 
There exists an open dense subset $\ha'\subset\ha$ such that for every $A\in\ha'$ the intersection $\nu_B(W_A^0)\cap a$ contains a regular  
disk $\alpha\subset a$ without marked points for $\alpha$ (Corollary \ref{corx}). Fix $\alpha$ and an $x_0\in M'\cap\nu_B^{-1}(\alpha)$. The point $\nu_B(x_0)\in 
b(x_0)\cap\alpha$ is a non-marked 
point of the corresponding local branches of the curves $b(x_0)$ and $\alpha$ (Remark \ref{rems}). The latter  branches are distinct, by 
Corollary \ref{inters}. Therefore, there exists a neighborhood $V=V(x_0)\subset M'$ such that for every $x\in V$ a regular branch of the curve $b(x)$ 
intersects $\alpha$ at a non-marked point for both curves (Remark \ref{rems} and analyticity of the foliation by integral surfaces). 
This proves the proposition. 
\end{proof}

In Proposition \ref{prop-inters} for every $x\in V$ either $c(x)=a$, or $b(x)$ is a line, by Corollary  \ref{inters}. Hence some of the 
 latter statements holds for all $x\in M'$, by analyticity. The mirror $b(x)$ cannot be a line for all $x$, since the mirror $b$ of the 
initial billiard is not algebraic by assumption. Hence, $\nu_C(M)\subset a$. Let us show that this is impossible. To do this, we use 
the next proposition.

\begin{proposition} \label{focic} For every analytic billiard $a$, $b$, $c$, $d$ with a non-algebraic mirror $b$ in every one-parametric 
family of quadrilateral orbits $ABCD$ with  fixed non-isotropic vertex $A\neq I_1, I_2$ 
the vertex $C$ is non-constant. 
\end{proposition}

\begin{proof} If $C\equiv const$, then $b$ would be either a line, or a conic, by \cite[proposition 2.32]{alg}, -- a contradiction.
\end{proof}

%Let $N$ be the same, as in Corollary \ref{crojb}. Then for every $x\in N$ either the curve $b(x)$ is a line, or $c(x)=a$, by 
%Corollary \ref{inters}. The subset of those $x\in M^0$ for which $b(x)$ is a line is either all of $M^0$, or a proper analytic subset in 
%$M^0$. 
%The former case is impossible, since then the mirror $b$ of the initial billiard $a$, $b$, $c$, $d$ would be also a line, by the  
%connectivity of the domain $M^0$, -- a contradiction. Hence, $c(x)=a$ for an open and dense subset of points $x\in N$, thus 
% $\nu_C(N)\subset a$. 

Suppose, by contradiction,  that $\nu_C(M)\subset a$. Fix an  $A\in\ha$ such that $\pi_a(A)$ is not an isotropic point at infinity and there exists a 
one-parametric family of quadrilateral orbits $ABCD$ of the initial billiard $a$, $b$, $c$, $d$ with the given vertex $A$. The  
 subset $\nu_C(W_A)\subset\cp^2$ is non-discrete, by Proposition \ref{focic}. On the other hand, it is an algebraic subset in $\cp^2$ (Remmert's Proper Mapping and Chow's Theorems). It lies in a transcendental curve $a$. Hence, it is discrete. The contradiction thus obtained 
proves Theorem \ref{tallalg}.

\subsection{Background material:  Phaffian systems and involutivity}

Everywhere below in the present subsection  whenever the contrary is not specified, $\mcf$ is a $k$-dimensional  analytic 
distribution on an analytic manifold $M$, $\mcf(x)\subset T_xM$ are the corresponding subspaces. 

\begin{definition} \cite[p.290]{rash} Let $k,l\in\mathbb N$, $k\geq l$, and let $\mcf$ be as above. 
A {\it Pfaffian system} $\mcf_{k,l}$ is the problem to find $l$-dimensional 
analytic integral surfaces of the $k$-dimensional distribution $\mcf$. 
%If $\mcf$ is a singular analytic distribution on an 
%analytic variety, we will call it a {\it singular} Pfaffian system. 
\end{definition}

\begin{definition} \label{int-el} \cite[p.298]{rash} An $m$-dimensional  {\it integral element} 
of the distribution $\mcf$ is an $m$-dimensional vector subspace $E_m(x)\subset\mcf(x)$ such that for 
every 1-form $\omega$ on the ambient manifold vanishing on the subspaces of the distribution $\mcf$ 
its differential $d\omega$ vanishes on $E_m(x)$. 
\end{definition}

\begin{definition} \label{definv} \cite[p.300]{rash} A Pfaffian system $\mcf_{k,l}$ is {\it in involution} (or briefly, 
involutive), if for every  $x\in M$, $p<l$  each $p$-dimensional 
integral element in $T_xM$ is contained in some $(p+1)$-dimensional 
integral element.  
\end{definition}

\begin{example} A tangent subspace to an integral surface is an integral subspace. Every Pfaffian system defined by a 
Frobenius integrable distribution is involutive. 
\end{example}

\begin{remark} \label{rpencil} 
For every $p$-dimensional  integral element $E_p(x)$ the set of ambient $(p+1)$-dimensional 
integral elements  $E_{p+1}(x)\supset E_p(x)$ either is empty, or 
consists of a unique integral element, or is a projective space: a pencil of $(p+1)$-dimensional subspaces through  
$E_p(x)$, which saturate a linear subspace containing $E_p(x)$,  see \cite[formula (58.13), p.299]{rash}. 
\end{remark}

\begin{definition} \label{defnonsing} \cite[p.306]{rash} Let $\mcf$ be an analytic distribution on a connected manifold. A 
$p$-dimensional integral element 
$E_p(x)$ is said to be {\it nonsingular}, if the  space of ambient $(p+1)$-dimensional integral elements 
$E_{p+1}(x)\supset E_p(x)$ has  minimal dimension. 
\end{definition}

\def\omcp{\overline{\mathcal I_p}}
\def\omcpp{\overline{\mathcal J_p}}

\begin{theorem} \label{ck} (a version of Cauchy--Kovalevskaya Theorem; implicitly contained in \cite[section 60]{rash}). 
 Let an analytic Pfaffian system $\mcf_{k,l}$ 
on a manifold $M$ be involutive. Let $p\leq l$, $\Gamma\subset M$ be a 
$(p-1)$-dimensional analytic integral surface such that all its tangent spaces be 
nonsingular $(p-1)$-dimensional integral elements. Then for every $x\in\Gamma$ there 
exists a germ of $p$-dimensional integral surface through $x$ that contains the germ 
of $\Gamma$ at $x$. 
\end{theorem}

\begin{definition} \cite[p.188]{hajto}. A subset $N$ of a complex manifold $V$ is called {\it analytically constructible,} if each point of the 
manifold $V$ has a neighborhood $U$ such that the intersection $N\cap U$ is a finite union of subsets defined by  finite systems of 
equations $f_j=0$ and inequalities $g_i\neq0$; $f_j$ and $g_i$ are holomorphic functions on $U$. 
\end{definition}

Recall that for a singular analytic distribution $\mcd_M$ on an irreducible analytic subset $M$ in a complex manifold $V$ 
by $M^0_{reg}\subset M$ we denote the open 
and dense subset of points regular both for $M$ and $\mcd_M$; the complement $M\setminus M^0_{reg}\subset V$ is an analytic subset. 

\begin{proposition} \label{propan2} Let $\mcd_M$ be a singular analytic distribution on an irreducible analytic subset $M$ in a complex 
manifold $V$. Its nonsingular one-dimensional integral elements form an  open dense subset in the projectivized bundle 
$\mathbb P(TM_{reg})$ that is an analytically constructible subset in  $\mathbb P(TV)$. 
\end{proposition}

\begin{proof} Fix a point $x_0\in M$. Let $f_1,\dots,f_r$ be holomorphic functions on a neighborhood $U=U(x_0)\subset V$ that define $M$: 
$M\cap U=\{ f_1=\dots=f_r=0\}$, $T_xM=\operatorname{Ker}(df_1(x),\dots,df_r(x))$ for every $x\in M_{reg}\cap U$. 
Let $\omega_1,\dots,\omega_s$ be holomorphic 1-forms on $U$ that 
define a singular analytic distribution  whose restriction to $M\cap U$ coincides with $\mcd_M$. A non-zero vector $v\in TV$ 
generates a non-singular integral element, if and only if its base point $x=x(v)$ lies in $M^0_{reg}$, $(df_i)(v)=0$ for $i=1,\dots,r$, 
$\omega_j(v)=0$ for  $j=1,\dots,s$   
and the system of $(r+s)$ 1-forms $df_i(x)$, $i_v(d\omega_j)(x)$, $i=1,\dots,r$, $j=1,\dots,s$ has maximal rank. This implies the 
statements of   the proposition.
\end{proof}

\begin{remark} \label{int-const}  For every singular analytic distribution $\mcf$ on an analytic set $M$ in a complex manifold $V$ 
the set of its two-dimensional integral elements (integral planes) is an analytically constructible subset in $Gr_2(TV)|_M$. 
The proof of this statement is analogous to the above proof of Proposition \ref{propan2}. Therefore, its image under the  projection to 
$M$ is also analytically constructible (Chevalley--Remmert Theorem). 
\end{remark}

\begin{proposition} \label{altern} Every three-dimensional singular analytic distribution $\mcf$ 
on an irreducible analytic variety $M$ satisfies one of the two following incompatible statements:

1) either the corresponding Pfaffian system $\mcf_{3,2}$ on $M^0_{reg}$ is involutive;

2) or there exists a proper analytic subset $\Sigma\subset M$ such that 

a) either for every $x\in M\setminus\Sigma$ the space $\mcf(x)$ contains 
no integral plane;

b) or for each $x\in M\setminus\Sigma$  the space $\mcf(x)$ contains a unique integral plane. 
 \end{proposition}
 
 \begin{proof} If $\mcf(x)$ contains some two distinct integral planes $P_1$, $P_2$, then it contains 
 a pencil of integral planes through the line $L=P_1\cap P_2$,   which saturate the whole 
 three-dimensional space $\mcf(x)$, by Remark \ref{rpencil}. Thus, each $\mcf(x)$ is either a union of integral planes, 
 or contains a unique integral plane, or contains no integral planes. 
% Recall that the space of integral planes 
% of the distribution $\mcf$ and its image $I\subset M$ under the projection to $M$ are analytically constructible. 
 This together with Remark \ref{int-const} and Chevalley--Remmert Theorem easily implies that 
 $M_{reg}^0$ is a  disjoint union of  three analytically constructible subsets of points $x\in M$ where
 % consisting of those $x\in M_{reg}^0$ for 
 $\mcf(x)$ satisfy respectively 
  one of the three latter statements. One of them contains a complement to a proper analytic subset 
  $\Sigma\subset M$. Therefore, either $\mcf(x)$ is a union of integral planes for all $x\in M\setminus\Sigma$ 
  (and hence, for all $x\in M_{reg}^0$ by analyticity) 
  % (and then, the same is true for all  
  and the Pfaffian system $\mcf_{3,2}$ is involutive, or some of statements a) or 
  b) of Proposition \ref{altern} holds. This proves Proposition \ref{altern}. 
  \end{proof} 

\def\pp{\mathbb P(\mcd)}

\subsection{Case of three-dimensional distribution}

Here we consider that $dim \mcd_M=3$. The subcase, when the Pfaffian system $(\mcd_M)_{2,3}$ 
is non-involutive, is reduced to the two-dimensional case by Proposition \ref{altern}. Indeed, the image 
of the projection to $M$ of the set of  integral planes is analytically constructible, by Remark \ref{int-const}, and  contains the integral surface $S$. Therefore, it contains an open dense  complement to an analytic subset in $M$, since $M$ is the minimal 
analytic set containing $S$. Hence, if the  system is not involutive, then there exists a proper analytic subset $\Sigma\subset M$ 
such that for every $x\in M\setminus\Sigma$ the space $\mcf(x)$ contains a unique integral plane (Proposition \ref{altern}). 
The latter integral planes form a two-dimensional intrinsic singular analytic distribution on $M$,  by Definition \ref{defint} and Remark  
 \ref{int-const}. Then we apply the arguments from Subsection 4.1 to the latter  two-dimensional distribution instead of $\mcd_M$. 
Thus, without loss of generality  we consider that {\it the Pfaffian system $(\mcd_{M})_{2,3}$ is involutive.} 
%since it contains the integral surface $S$ and $M$ is 
%by Corollary \ref{2dim}. In this subcase each its integral surface lifts to an integral surface of the 
%plane field $\wt\mcd_M$ on the (at most double) cover $\wt M$ from the corollary. We 

Corollary \ref{corx} easily implies the next proposition and corollary, which state that there 
 exist a connected open subset $V\subset M^0_{reg}$,  a regular analytic hypersurface 
$V_a\subset V$, $\nu_B(V_a)\subset a$, and an analytic field $\mcl$ of  one-dimensional nonsingular integral elements in $\mcd_M$ on $V$ 
whose all complex orbits intersect $V_a$ transversely (plus a mild genericity condition (\ref{cond})). The germ through every $x\in V$ 
of complex orbit is included into an integral 
surface of the distribution $\mcd_M$, by Theorem \ref{ck}. Condition (\ref{cond}) implies that the integral surface is non-trivial, and hence, 
corresponds to an  open set of  quadrilateral orbits of a 4-reflective billiard $a$, $b(x)$, $c(x)$, $d(x)$.  If $x\in V_a$, then the 
mirrors $a$ and $b(x)$ intersect at $\nu_B(x)$, and we deduce that either $c(x)=a$, or $b(x)$ is a line (Corollary \ref{inters}). 
This easily implies that either $\nu_C(M)\subset a$, or the image under the projection $\nu_B$ of every analytic curve 
tangent to $\mcd_M$ is a line.  We show that none of the latter cases is possible. The contradiction thus obtained will prove 
Theorem \ref{tallalg}.

Let $M'\subset M^0_{reg}$ be the subset from (\ref{defmo}) defined for our three-dimensional distribution $\mcd_M$. 
It is open, dense and the complement $M\setminus M'$ is analytic, as at the same place. By definition, $d\nu_a,d\nu_G
\not\equiv 0$ on $\mcd_M(x)$ for every $x\in M'$.

\begin{proposition} There exist an $x\in M'$ and a one-dimensional nonsingular integral element $\mcl_x\subset\mcd_M(x)$ 
such that $\nu_B(x)\in a$ and 
\begin{equation}(d\nu_a(x))(\mcl_x)\neq0, \ (d\nu_G(x))(\mcl_x)\neq0 \text{ for every } G=B,C,D;\label{cond}\end{equation}
\begin{equation}\text{ the line } (d\nu_B(x))(\mcl_x) 
\text{ is transverse to } T_{\nu_B(x)}a.\label{cond1}\end{equation}
\end{proposition}

\begin{proof} 
%for which the differentials $d\nu_a(x)$, $d\nu_G(x)$, $G=B,C,D$, 
%do not vanish identically on $\mcd_M(x)$. The subset $N\subset M$ is open and dense, since 
%$N$ contains an open and dense subset of the 
%non-trivial integral surface $S$ (corresponding to the initial billiard $a$, $b$, $c$, $d$), 
%and $M$ is the minimal analytic set containing $S$.  
Let $\wt Q$ denote the set of one-dimensional 
nonsingular integral elements $\mcl_x\subset\mcd_M(x)$, $x\in M'$, satisfying (\ref{cond}). 
Let $Q\subset M'$ denote its projection to $M$. 
%Recall that $\wt Q\neq\emptyset$, since $\wt Q$ contains a non-trivial integral 
%surface $S$. 
The sets $\wt Q$ and $Q$ are  analytically constructible subsets in $\mathbb P(T\mcp_a)$ and $M$.
 They are open and dense subsets in $\mathbb P(\mcd_M)$ and $M$ respectively. The two latter statements follow from definition, 
 Proposition \ref{propan2} and Chevalley--Remmert Theorem. The intersection $\nu_B(Q)\cap a$ contains 
a regularly embedded disk $\alpha\subset a$ without isotropic tangent lines, and $V_a=\nu_B^{-1}(\alpha)\cap Q$ is 
a hypersurface (Corollary \ref{corx}). 

\medskip

{\bf Claim.} {\it $\mcd_M(x)\not\subset T_xV_a$ for an open dense set of points $x\in V_a$.} 

\medskip

\begin{proof} Suppose the contrary: each germ of  integral curve (and hence, surface) of the distribution $\mcd_M$ through  each 
point in $V_a$ lies in $V_a$. Fix an $x\in V_a$, a nonsingular integral element $\mcl_x\subset \mcd_M(x)$ satisfying (\ref{cond}), a germ of 
integral curve $\Gamma$ tangent to $\mcl_x$ and a germ of integral surface $\hat S$ containing $\Gamma$ (given by Theorem \ref{ck}). 
The surface $\hat S$  lies in $V_a$, is non-trivial by (\ref{cond}), and hence, represents an open set of quadrilateral orbits of a 4-reflective billiard with two coinciding non-linear 
mirrors $a=b(x)$.  But the latter billiard cannot exist by Corollary \ref{inters}. This proves the claim.
\end{proof}

For every $x\in V_a$ such that $\mcd_M(x)\not\subset T_xV_a$ a generic one-dimensional integral element $\mcl_x\subset\mcd_M(x)$ 
is nonsingular and satisfies conditions (\ref{cond}), (\ref{cond1}). This proves the proposition.
\end{proof}

%Fix an $x\in Q$ with 
%$\nu_B(x)\in \alpha$. By definition, there exists a one-dimensional non-singular integral element 
%$\mcl_x\subset \mcd_M(x)$ satisfying (\ref{cond}). Let us show that 
%one can achieve transversality condition (\ref{cond1}) as well. Suppose the contrary. Then  
%for every $x\in Q\cap\nu_B^{-1}(a)$ each one-dimensional nonsingular integral element 
%in $\mcd_M(x)$ is  tangent to the analytic hypersurface $\sigma(a)=\nu_B^{-1}(a)$. This implies that the distribution $\mcd_M$ is tangent to 
%$\sigma(a)$: the tangent spaces to $\sigma(a)$ contain the distribution subspaces. Therefore, every 
%irreducible germ $\gamma$ of integral curve  of the distribution $\mcd_M$ at a point in $\sigma(a)$ is contained in $\sigma(a)$. 
%Let us choose the latter germ $\gamma$ to be tangent to a nonsingular one-dimensional integral element satisfying inequalities 
%(\ref{cond}). Then  $\gamma$ is contained in a germ of integral surface (Theorem \ref{ck}), and the latter is non-trivial by the same inequalities. Note that $\nu_B(\gamma)\subset a$. 
%Therefore, the latter surface is formed by an open set of quadrilateral orbits of a 4-reflective billiard $a$, $a$, $c'$, $d'$ with two coinciding 
%nonlinear neighbor 
%mirrors, -- a contradiction to Corollary \ref{inters}. This proves the proposition.
%\end{proof}

\begin{corollary} There exist an open subset $V\subset M'$, a regularly embedded disk $\alpha\subset a$ without isotropic tangent lines,  
an analytic hypersurface $V_a\subset V$ with 
$\nu_B(V_a)\subset \alpha$ and an analytic line field $\mcl$ on $V$ transverse to $V_a$ such that the lines  of the field
$\mcl$ are nonsingular integral elements satisfying  inequalities  (\ref{cond}) 
and each its complex orbit  in $V$ intersects $V_a$. 
\end{corollary}

The corollary follows immediately from the proposition and openness of the set of nonsingular integral elements satisfying (\ref{cond}). 

\begin{proposition} Let $V$, $V_a$ and $\mcl$ be as in the above corollary. Then there are two possible cases:

Case 1): $\nu_C(V)\subset a$;

Case 2):  the projection $\nu_B$ sends complex orbits of the field $\mcl$ to lines. 
\end{proposition}

\begin{proof} For every $x\in V$ the germ of the orbit of the field $\mcl$ through $x$ lies in a germ of integral surface of the 
distribution $\mcd_M$ (Theorem \ref{ck}). 
The latter surface is non-trivial by the inequalities from (\ref{cond}), and hence, is given by an open set 
of quadrilateral orbits of a 4-reflective billiard $a$, $b(x)$, $c(x)$, $d(x)$. If $x\in V_a$, then the mirrors $a$ and $b(x)$ intersect at 
the point $B(x)=\nu_B(x)$, and the latter is not marked for their corresponding 
local branches. This follows from  construction and the inequalities from (\ref{cond}). Hence, for every $x\in V_a$ 
either $c(x)=a$, or $b(x)$ is a line (Corollary \ref{inters}). Each orbit of the field $\mcl$ intersects $V_a$, by 
assumption. Hence, either the projection $\nu_C$ sends it to $a$, or the projection $\nu_B$ sends it to a line. One of the two latter 
statements holds for all the orbits, by analyticity. This proves the proposition.
\end{proof}

Now for the proof of Theorem \ref{tallalg} it suffices to show that none of the cases from the above proposition is possible.

Case 1): $\nu_C(V)\subset a$. Then $\nu_C(M)\subset a$, and we get a contradiction, as at the end of Subsection 4.1. 
Hence, Case 1) is impossible. 
%On the other hand, for every $A\in\ha$ the set $W_A$ is algebraic 
%and its image $\nu_C(W_A)\subset\cp^2$ is an algebraic subset by Remmert's Proper Mapping and Chow's Theorems. The latter image 
%being contained in a transcendental curve $a$, it is a finite union of points. In particular, in the initial billiard $a$, $b$, $c$, $d$ 
%defining the set $M$ in each one-parametric family of its quadrilateral orbits with constant vertex $A$ the vertex $C=C(A)$ is also constant. 
%Therefore, the curves $b$ and $d$ are conics (with variable foci $A$ and $C(A)$), by \cite[proposition 2.32]{alg}, -- a contradiction. 

Case 2): $\nu_C(M)\not\subset a$. Let  $V$ and $\mcl$ be the same, as in the above corollary. Let us  deform $\mcl$. 
The set of  line fields $\mcl$  satisfying the conditions of the corollary is open in the space of line fields contained in the distribution 
$\mcd_M$. This together with the corollary implies that for every line field $\mcl$ contained in $\mcd_M$ the projection $\nu_B$ sends 
each its complex orbit to a line. This is equivalent to say that each analytic curve in $M^0_{reg}$  tangent to $\mcd_M$ is sent to a line by $\nu_B$. 
In particular, this holds for 
every one-parametric family of quadrilateral orbits lifted to $M$ of the initial billiard $a$, $b$, $c$, $d$ with  variable $B\in b$. 
This implies that the curve $b$ is a line. The contradiction thus obtained proves Theorem \ref{tallalg}. The proof of Theorem \ref{an-class} 
is complete.

\section{Applications to real pseudo-billiards}

In Subsection 5.1 we introduce and classify the germs of 4-reflective $C^4$-smooth real planar pseudo-billiards. 
The proof of the classification Theorem \ref{class} is presented in the same subsection (analytic case) and in Subsection 5.2 
(smooth case). In the same Subsection 5.2 we prove Theorem \ref{measure} showing that there are no $C^4$-smooth pseudo-billiards with 
only two skew  reflection laws 
at some neighbor mirrors and a positive measure set of 4-periodic orbits.  In Subsections 5.3, 5.4 we 
present applications of Theorems \ref{class} and \ref{measure} respectively to Tabachnikov's Commuting Billiard Conjecture and 
4-reflective Plakhov's Invisibility Conjecture.

\subsection{Classification of real planar 4-reflective pseudo-billiards}
\def\rpd{\mathbb{RP}^2}

 Here by {\it real smooth (analytic) curve}  in  $\rr^2$ or $\rpd$  we mean the image of either $\mathbb R$, or $S^1$ under a locally 
non-constant smooth (analytic) mapping to $\rr^2$ (respectively, $\rpd$). A smooth (analytic) {\it germ}  of curve is given by a 
smooth (analytic) germ of immersion $(\rr,0)\to\rr^2$ ($(\rr,0)\to\rpd$). 

% and distinct from the 
% infinity line.  For each real analytic curve $a\subset\rp^2$ (which may  have singularities: cusps or self-intersections) 
% we consider its maximal real analytic extension $\pi_a:\ha\to a$, where $\ha$ is either $\mathbb R$, or $S^1$, see 
% \cite[lemma 37, p.302]{gk2}. The parametrizing curve $\ha$ will be called here the {\it real normalization.} 
  
 \begin{definition} \cite[remark 1.6]{alg} 
 Let a line $L\subset\rr^2$ and a triple of points  $A,B,C\in\rr^2$ be  such that $B\in L$, $A,C\notin L$ and 
 the lines $AB$, $BC$ are symmetric with respect to the line $L$. We say that the triple 
 $A$, $B$, $C$ and the line $L$ satisfy the {\it usual reflection law}, if the points $A$ and $C$ lie on the same side from the line $L$. 
 Otherwise, if they are on different sides from the line $L$, we say that the {\it skew reflection law} is satisfied.
 \end{definition}
 
 \begin{example} In every planar billiard orbit each triple of consequent vertices satisfies the usual reflection law with respect to the 
 tangent line to the boundary of the billiard at the middle vertex. 
 \end{example}
 
 \begin{definition} (cf. \cite[definition 6.1]{alg}) A {\it real planar  pseudo-billiard} is  a collection of $k$ curves $a_1,\dots,a_k\subset\rr^2$ 
 called {\it mirrors} with  
a prescribed reflection law on each curve $a_j$: either usual, or skew. Its {\it $k$-periodic   orbit}  
  is a $k$-gon $A_1\dots A_k$, $A_j\in a_j$, such that for every $j=1,\dots,k$ one has $A_j\neq A_{j\pm1}$,
 $A_jA_{j\pm1}\neq T_{A_j}a_j$ and the lines $A_jA_{j-1}$, $A_jA_{j+1}$ are symmetric with respect to the tangent line $T_{A_j}a_j$ so that 
 the triple $A_{j-1}$, $A_j$, $A_{j+1}$ and the line $T_{A_j}a_j$ satisfy the reflection law corresponding to $a_j$.  
 Here we set $a_{k+1}=a_1$, $A_{k+1}=A_1$, $a_0=a_k$, $A_0=A_k$. A real pseudo-billiard is called {\it (piecewise) analytic/smooth}, if so are its curves.  A {\it germ} of real pseudo-billiard is a collection of $k$ {\it germs} of  curves $(a_j,A_j)$ with prescribed reflection laws 
 for which the marked $k$-gon $A_1\dots A_k$ is a $k$-periodic orbit. A germ of pseudo-billiard is called 
 {\it $k$-reflective,} if the set of its $k$-periodic orbits has non-empty interior: contains a two-parameter family. 
 A (germ of) pseudo-billiard is called {\it measure $k$-reflective,} if the set of its $k$-periodic orbits has positive two-dimensional Lebesgue 
 measure.  This means that in the set of $k$-orbits $A_1\dots A_k$ that satisfy the corresponding reflection laws at $A_j$ for all $j\neq1,k$ 
 the set of $k$-periodic ones (i.e., those satisfying the reflection laws at $A_1$, $A_k$) has positive two-dimensional Lebesgue measure.
 \end{definition}
 
  \begin{remark} \label{rpseudo} A  $k$-reflective pseudo-billiard is automatically measure $k$-reflective. 
   In the analytic case measure $k$-reflectivity is 
 equivalent to $k$-reflectivity.   The  interior  points  of the set of $k$-periodic orbits will be called {\it $k$-reflective orbits}, cf. definition 6.1 in 
 loc. cit. 
  The complexification of each $k$-reflective planar analytic pseudo-billiard is a $k$-reflective complex billiard. 
 \end{remark}
 
 \begin{convention} Given two germs of smooth curves $(a,A)$, $(c,C)$ in  $\rr^2\subset\rpd$, 
 we say that $a=c$, if they lie in the same {\bf analytic} curve in $\rpd$. 
 \end{convention}
 
  \begin{theorem} \label{class} A   germ of $C^4$-smooth real planar pseudo-billiard   $(a,A)$, $(b,B)$, $(c,C)$, $(d,D)$  
  is  4-reflective, if and only if it has one of the following types:
 
 1) $a=c$ is a line,  the curves $b,d\neq a$ are symmetric with respect to it;
 
 2) $a$, $b$, $c$, $d$ are distinct lines  through the same point $O\in\mathbb{RP}^2$, 
the line pairs $(a,b)$, $(d,c)$ are transformed one 
 into the other by  rotation around $O$ (translation, if $O$ is an infinite point), see Fig.\ref{fig-lines}; 
 
 3) $a=c$, $b=d$, and they are distinct confocal conics: either ellipses, or hyperbolas, or ellipse and hyperbola, or parabolas. 
 
 In every 4-reflective orbit the  reflection law at each pair of opposite vertices is the same; it is skew for at least one opposite vertex pair. 
  \end{theorem}

 {\bf Addendum 1.} {\it In every pseudo-billiard of type 1) from Theorem \ref{class}
  each quadrilateral orbit  $ABCD$ has the same type, as at Fig.\ref{fig-sym}. It 
 is symmetric with respect to the line $a$,  and the reflection law at $A$, $C$ is skew. The reflection law at $B$, $D$ is either usual at both, 
 or skew at both.}
 
{\bf Addendum 2} \cite[addendums 2, 3 to theorem 6.3]{alg}. {\it In  pseudo-billiards of types 2), 3) the 4-reflective orbits have the same types,  
as at Fig.\ref{fig-lines}--\ref{fig-par}.}
%The reflection laws at each pair opposite vertices coincide: either usual at both, or skew at both.}

\begin{remark} The main result of paper \cite{gk2} (theorem 2) concerns  usual real planar billiards with piecewise $C^4$ boundary; 
the reflection law is usual. It implies that the  quadrilateral orbit set has empty interior. 
This statement also follows from the last statement of Theorem \ref{class}. 
\end{remark}

 \begin{figure}[ht]
  \begin{center}
 % \vspace{-0.3cm}
   \epsfig{file=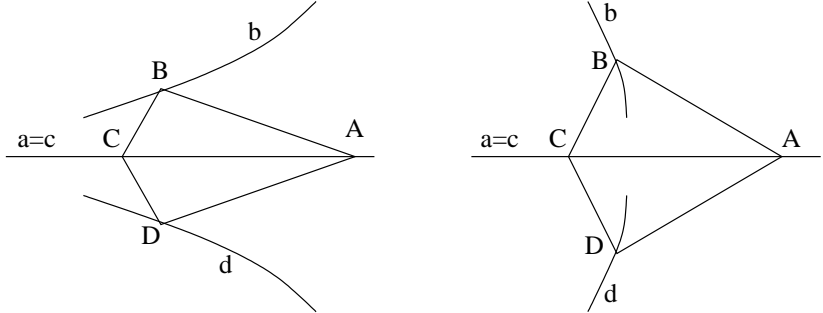}
   \vspace{-0.3cm}
    \caption{4-reflective pseudo-billiards symmetric with respect to a line mirror}
  \label{fig-sym}
  \end{center}
\end{figure}

\begin{figure}[ht]
  \begin{center}
  %\vspace{-0.5cm}
   \epsfig{file=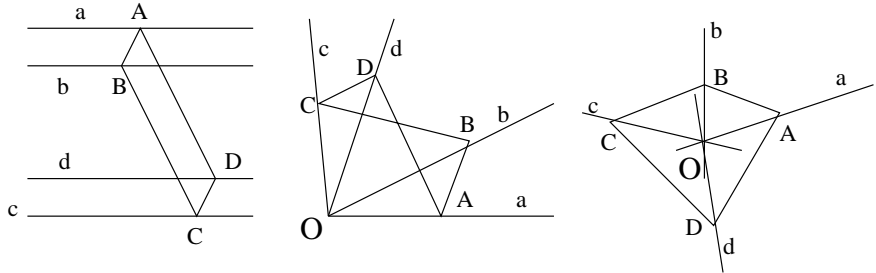}
   \vspace{-0.5cm}
    \caption{4-reflective pseudo-billiards on two positively-isometric line pairs}
\label{fig-lines}
  \end{center}
\end{figure}
\begin{figure}[ht]
  \begin{center}
%  \vspace{-4cm}
   \epsfig{file=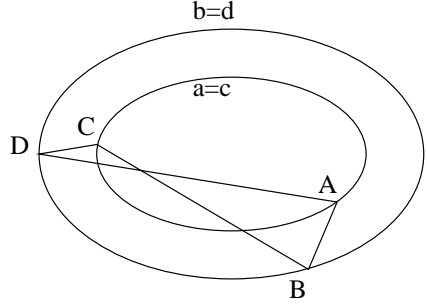}
    \caption{A 4-reflective pseudo-billiard on confocal ellipses}
    \label{fig-ellipses}
  \end{center}
\end{figure}

 \begin{figure}[ht]
  \begin{center}
   \epsfig{file=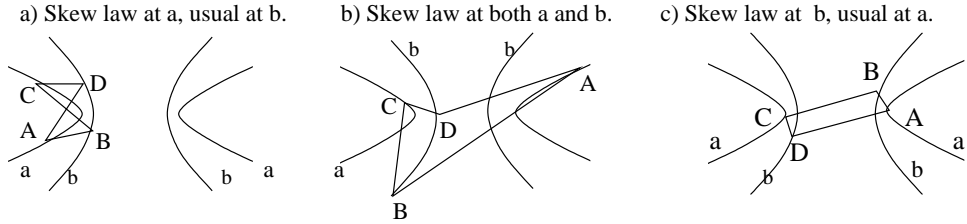}
   \vspace{-0.8cm}
    \caption{Open sets of  orbits on confocal hyperbolas: three types}
    \label{fig-hyp}
  \end{center}
\end{figure}

\begin{figure}[ht]
  \begin{center}
  \vspace{-0.3cm}
   \epsfig{file=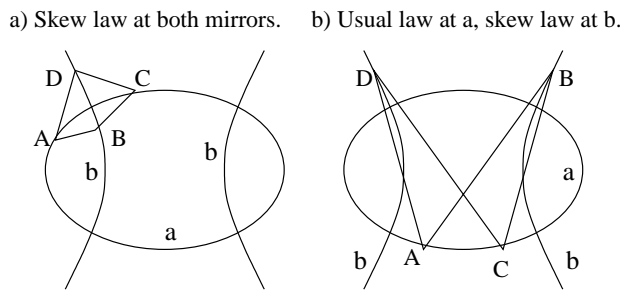}
   \vspace{-0.3cm}
    \caption{Open sets of  orbits on confocal ellipse and hyperbola: two types}
    \label{fig-elhyp}
  \end{center}
\end{figure}

\begin{figure}[ht] 
 \vspace{-0cm}
  \begin{center}
\epsfig{file=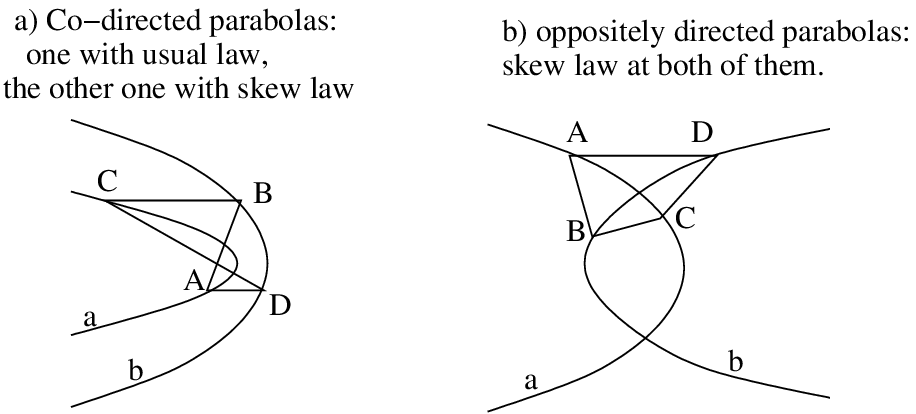}
   \vspace{-0.4cm}
    \caption{Open set of  orbits on confocal parabolas: one type}
    \label{fig-par}
  \end{center}
\end{figure}

Here we prove Theorem \ref{class} for analytic pseudo-billiards. The general smooth case will be treated in the next subsection. We also 
prove the following theorem there, which will be applied to Plakhov's Invisibility Conjecture.

\begin{theorem} \label{measure} There exist no measure 4-reflective $C^4$-smooth planar pseudo-billiard with   
exactly two skew reflection laws at a pair of neighbor mirrors.
\end{theorem}

\begin{proof} {\bf of  Theorem \ref{class} (analytic case) and Addendum 1 (smooth case).} 
The analytic pseudo-billiard  under question being 4-reflective, its complexification is obviously 
4-reflective, by Remark \ref{rpseudo}. This together with Theorem \ref{an-class} 
implies that it has one of the above types 1)--3) (up to  cyclic renaming of the mirrors). The 4-reflectivity of  pseudo-billiards of types   
2), 3) and the classification of open sets of their quadrilateral orbits and reflection law configurations 
was proved in \cite[section 6]{alg}. The 4-reflectivity of pseudo-billiards of type 1) is obvious. 
Addendum 1 (symmetry of the vertices $B$ and $D$) follows from the fact that they are intersection points of symmetric pairs of lines, 
by definition.   
\end{proof}

\subsection{Analytic versus smooth: proofs of Theorems \ref{class} and \ref{measure}}

%New Correction start
The proofs of Theorems \ref{class} and \ref{measure} in the smooth case 
%New Correction end
given here are analogous to the arguments from \cite[section 2]{gk2} due to 
Yu.G.Kudryashov. We start them with the following simple fact.

\begin{proposition} \label{orient} For every $k\in\mathbb N$ there exist no measure $k$-reflective  
real $C^3$-germ of pseudo-billiard with odd number of skew reflection laws.
\end{proposition}

\begin{proof} Reflections from curves act on the space of oriented lines in $\rr^2$, which is a two-dimensional oriented manifold.  
Skew reflections change the orientation, and the usual ones don't. Therefore, a composition of odd number of skew reflections and a number of 
usual ones (in any order) is a local diffeomorphism changing the orientation. Hence, it cannot be equal to the identity on a set of positive measure.
This proves the proposition.
\end{proof}

The proposition implies that 
the only possible reflection law configurations for a potential measure 4-reflective $C^4$ pseudo-billiard are the following:

1) all the reflection laws are usual;

2) the reflection laws are skew only at some pair of neighbor mirrors;

3) all the reflection laws are skew;

4) the reflection laws are skew only at some pair of opposite mirrors.

\begin{remark} Configuration 1) is already forbidden: the proof of its impossibility is implicitly contained  in \cite{gk2}. The proof of 
Theorem \ref{class} given below does not use this result. 
 Configuration 2) is already forbidden in the analytic case, by the last statement of Theorem \ref{class}  proved in this case. 
 \end{remark}

\begin{definition} (\cite[definition 14]{gk2}, Yu.G.Kudryashov). Let $k\geq3$. A $k$-gon $A_1\dots A_k$ in $\rr^2$ is {\it non-degenerate,} if 
for every $j=1,\dots,k$ (we set $A_{k+1}=A_1$, $A_0=A_k$) one has $A_{j+1}\neq A_j$ and 
$A_{j-1}A_j\neq A_jA_{j+1}$.  The subset of degenerate $k$-gons in $\rr^{2k}=(\rr^2)^k$ will be denoted by $\Sigma=\Sigma_k$. 
\end{definition}

\begin{remark} A real non-degenerate $k$-gon is a complex non-degenerate $k$-gon in the sense of Definition \ref{defnondeg}. 
\end{remark}

To each configuration 1)--4) we put into correspondence a distribution $\mcd^4_{\alpha}$ on $\rr^8\setminus\Sigma$ constructed below: 
the Birkhoff distribution corresponding to the chosen reflection laws. To define it, set 
$$\Psi_k=\{\pm1\}^k, \ \Psi^{\pm}_k=\{\alpha=(\alpha_1,\dots,\alpha_k)\in\Psi_k\ | \ \prod_j\alpha_j=\pm1\}.$$
For every $k\in\mathbb N$, $\alpha\in\Psi_k$ 
and every non-degenerate $k$-gon $A_1\dots A_k$ in $\rr^2$ we consider the following collection of lines 
$L_{A_j}=L_{A_j}(\alpha)$ through $A_j$:

- the line $L_{A_j}$ is the exterior bisector of the angle $\angle A_{j-1}A_jA_{j+1}$ if $\alpha_j=1$;

- the line $L_{A_j}$ is its interior bisector  if $\alpha_j=-1$.

We will briefly call the configuration $((A_1,L_{A_1}),\dots,(A_k,L_{A_k}))$ {\it $\alpha$-framed.} We identify each line $L_{A_j}$ with the 
corresponding one-dimensional subspace in $T_{A_j}\rr^2$. For every $x=A_1\dots A_k\in\rr^{2k}\setminus\Sigma$ set
$$\mcd^k_{\alpha}(x) = \oplus_{j=1}^kL_{A_j}(\alpha)\subset T_{x}\rr^{2k}.$$
The latter subspaces define an analytic distribution $\mcd^k_{\alpha}$ on $\rr^{2k}\setminus\Sigma$. The above reflection law configurations 
1)--4) correspond to  $\mcd^4_{\alpha}$ with $\alpha\in\Psi^+_4$. 

We show that  all $\mcd^k_{\alpha}$ with $\alpha\in\Psi_4^+$ naturally embed to the complex 
Birkhoff distribution $\mcd^k$ on one and the same irreducible component $\mcr^+_{0,k}$ of the variety 
$\mcr_{0,k}\subset\mcp^k=(\mathbb P(T\cp^2))^k$ (the next proposition). 
We study the two-dimensional Pfaffian problem for the complex distribution $\mcd^4$ on $\mcr^+_{0,4}$: 
to find its two-dimensional integral surfaces. We study its first and second Cartan 
jet prolongations. (Basic theory  of Cartan prolongations may be found in \cite{cart, kuran, rash}; the background material 
used in the proofs will be recalled below.) 
We show  (the next two lemmas) that each second 
prolongation is a two-dimensional distribution outside a union of two subvarieties $\Lambda_0$, $\Lambda_1$. 
The latter union consists of the quadrilaterals 
symmetric with respect to some of the diagonals $A_1A_3$ or $A_2A_4$ with $L_{A_1}=L_{A_3}=A_1A_3$ (respectively, 
$L_{A_2}=L_{A_4}=A_2A_4$). We then deduce  
that each $C^4$-smooth non-trivial integral surface of the distribution $\mcd_{\alpha}^4$ should be either analytic or 
contained in the above 
set of symmetric quadrilaterals. Therefore, the pseudo-billiard under question should have one of the types 1)--3), by  
the analytic case of Theorem \ref{class} proved in the previous subsection. This will prove Theorem 
\ref{class} in the $C^4$ case. For the proof of Theorem \ref{measure} we show that the existence of a measure 4-reflective pseudo-billiard 
with reflection law configuration 2) implies the existence of a 4-reflective analytic one. This will be deduced from the above-mentioned lemmas and Kudryashov's 
results from \cite{gk2}. But 4-reflective analytic pseudo-billiards with reflection law configuration 2) are forbidden by the last statement of 
Theorem \ref{class}, analytic case. The contradiction thus obtained will prove Theorem \ref{measure}. 
 
 Now let us pass to the proofs. 
 
 \def\ja{j_{\alpha}}
 
Consider the $k$-dimensional complex Birkhoff distribution $\mcd^k$ on the $2k$-dimensional smooth quasiprojective variety 
$\mcr_{0,k}\subset\mcp^k=(\mathbb P(T\cp^2))^k$ introduced at the beginning of Subsection 2.7. For every $\alpha\in\Psi_k$ there is a natural 
embedding
$$j_{\alpha}:\rr^{2k}\setminus\Sigma\to\mcr_{0,k}: \ \ja(A_1\dots A_k)=((A_1,L_{A_1}(\alpha)),\dots,(A_k,L_{A_k}(\alpha))).$$
One has $(\ja)_*\mcd^k_{\alpha}\subset\mcd^k$, and the latter image is a field of totally real subspaces in the corresponding spaces 
of the complex distribution $\mcd^k$. 

\begin{proposition} \label{comp} For every $k\geq3$ the variety $\mcr_{0,k}$ consists of two irreducible components $\mcr^{\pm}_{0,k}$. Each 
$\mcr_{0,k}^{\pm}$ contains the union of all the images $\ja(\rr^{2k}\setminus\Sigma)$ with $\alpha\in\Psi^{\pm}_k$. The component 
$\mcr_{0,k}^+$ consists exactly of those configurations $((A_1,L_{A_1}),\dots,(A_k,L_{A_k}))$ for which the complex lengths $|A_jA_{j+1}|$, 
$j=1,\dots,k$ can be simultaneously normalized so that for every $j$ the lengths $|A_{j-1}A_j|$, $|A_jA_{j+1}|$ are $L_{A_j}$-concordant, 
see Definition \ref{concord}. 
\end{proposition}

\begin{remark} The above collection of lengths $|A_1A_2|,\dots,|A_kA_1|$ (if exists)  will be called {\it concordant.} 
It  is unique up to simultaneous change of sign.
\end{remark}

\begin{proof} The projection of the variety $\mcr_{0,k}$ to the space of $k$-gons has $2^k$ preimages: at each vertex $A_j$ we can choose a symmetry line $L_{A_j}$ between the lines $A_jA_{j\pm1}$ in two ways; the two symmetry lines are orthogonal.

\medskip

{\bf Claim.} {\it For every $j=1,\dots,k$ every  configuration $((A_1,L_{A_1}),\dots,(A_k,L_{A_k}))\in\mcr_{0,k}$ 
is connected by path in $\mcr_{0,k}$ to the same configuration with  simultaneously changed lines $L_{A_j}$, $L_{A_{j+1}}$.} 
%to their orthogonal lines.} 

\begin{proof} Let us identify the complex infinity line $\oc_{\infty}$ with the space of complex  lines through the point $A_j$. Its complement 
$\cc^*=\oc_{\infty}\setminus\{ I_1,I_2\}$ to the isotropic points has fundamental group $\zz$. Consider a closed path 
$\phi:[0,1]\to\cc^*$ starting at the line $A_jA_{j+1}$ and disjoint from the line $A_jA_{j-1}$ that represents its generator. This yields a closed path 
$\psi(t)=A_1\dots A_jA_{j+1}^tA_{j+2}\dots A_k$ in the space of non-degenerate complex $k$-gons 
(in the sense of Definition \ref{defnondeg}): $A_{j+1}^t=\phi(t)\cap A_{j+1}A_{j+2}$, 
$A_{j+1}^0=A_{j+1}$. The line $L_{A_j}=L_{A_j}(0)$ deforms into a family $L_{A_j}(t)$ of symmetry lines between the lines 
$A_jA_{j-1}$ and $\phi(t)$ that is analytic 
along the path $\psi$. One has $L_{A_j}(1)\neq L_{A_j}(0)$. Indeed, let $z_{j-1}$, $z_j(t)$, $z_{j+1}(t)$ denote  the points of intersection of the 
infinity line with the lines $A_jA_{j-1}$, $L_{A_j}(t)$, $A_jA^t_{j+1}=\phi(t)$ written in the standard affine coordinate $z$ on $\oc_{\infty}$: $z(I_1)=0$, 
$z(I_2)=\infty$. One has $z_{j+1}(t)=\frac{z_j^2(t)}{z_{j-1}}$, by symmetry and \cite[proposition 2.4, p.249]{alg}. The point $z_j(t)$ makes 
half-turn around zero, by the latter formula and since $z_{j+1}(t)$ makes one turn by assumption. Hence, $L_{A_j}(1)\neq L_{A_j}(0)$. 
We similarly consider the deformation $L_{A_{j+1}}(t)=L_{A_{j+1}^t}$ of the line $L_{A_{j+1}}$ and get  
$L_{A_{j+1}}(1)\neq L_{A_{j+1}}$. The claim is proved.
\end{proof}

The claim implies that any two configurations $((A_1,L_{A_1}),\dots,(A_k,L_{A_k}))$ and $((A_1,L_{A_1}'),\dots,(A_k,L_{A_k}'))$ 
in $\mcr_{0,k}$ with even number of pairs of different lines $L_{A_j}\neq L_{A_j}'$ are connected by path in $\mcr_{0,k}$. 
Therefore, the variety $\mcr_{0,k}$ consists of at most two irreducible components $\mcr_{0,k}^{\pm}$, each of them contains the 
images $\ja(\rr^{2k}\setminus\Sigma)$ for all $\alpha\in\Psi_k^{\pm}$.  Each real $(1,\dots,1)$-framed configuration 
has a concordant collection of real length (Example \ref{escort}). The existence of concordant complex length collection is invariant 
under continuous deformations in $\mcr_{0,k}$, which follows from definition. Therefore, each configuration from the component 
$\mcr^+_{0,k}$ has a concordant length collection. On the other hand, each real $(1,\dots,1,-1)$-framed configuration has no 
concordant length collection, which follows from Example \ref{escort}. Hence, this is true for 
every configuration from the component $\mcr^-_{0,k}$, and $\mcr^+_{0,k}\neq\mcr^-_{0,k}$. Proposition  \ref{comp} is proved.
\end{proof}

The arguments  below are analogous to Yu.G.Kudryashov's arguments from \cite[section 2]{gk2} on the Cartan prolongations of the Birkhoff distributions. 

%In what follows we use the notion of $r$-th (jet) Cartan  prolongation of analytic Pfaffian system \cite[pp. 293--294]{gk2} (see also ) and results from \cite[section 2]{gk2} on the existence of integral surfaces in terms of properties of prolongations.
In what follows the restriction of the distribution $\mcd^4$ to the component $\mcr^+_{0,4}$ will be denoted $\mcd^4_+$. 
Consider the complex Pfaffian system $\mcd^{4,2}_+$: the problem to find two-dimensional integral surfaces of the 
distribution $\mcd^4_+$. For every $x\in\mcr^+_{0,4}$ let $l_1(x),\dots,l_4(x)$ denote the corresponding 
concordant collection of lengths $l_j=|A_jA_{j+1}|$. The lengths being defined up to simultaneous change of sign, 
their ratios are single-valued holomorphic functions on $\mcr^+_{0,4}$. Set
$$\Lambda=\{ l_1l_3=l_2l_4\}\subset \mcr^+_{0,4},$$
$$\Lambda_0=\{ x\in\mcr^+_{0,4} \ | \ A_1(x)A_2(x)A_3(x)A_4(x) \text{ is symmetric with respect}$$
$$\text{to the line } A_1(x)A_3(x)=L_{A_1}(x)=L_{A_3}(x)\},$$
$$\Lambda_1=\{ x\in\mcr^+_{0,4} \ | \ A_1(x)A_2(x)A_3(x)A_4(x) \text{ is symmetric with respect}$$
$$\text{to the line } A_2(x)A_4(x)=L_{A_2}(x)=L_{A_4}(x)\}.$$

\begin{remark} \label{remcor} One has  $\Lambda_0,\Lambda_1\subset\Lambda$. Indeed, if $x\in\Lambda_0$, then $l_4(x)=-l_1(x)$, $l_2(x)=-l_3(x)$, by symmetry and length concordance (Definition \ref{concord}). Hence, $l_1(x)l_3(x)=l_2(x)l_4(x)$ and $x\in\Lambda$. For every  
$x\in\Lambda_0$ the lines $L_{A_2}(x)$, $L_{A_4}(x)$ are symmetric with respect to the line $L_{A_1}(x)$, since the symmetry 
respects concordance of lengths and transforms the $L_{A_2}(x)$-concordant lengths $l_1(x)$, $l_2(x)$ to the $L_{A_4}(x)$-concordant 
lengths $-l_4(x)$, $-l_3(x)$. Similar statement holds for $x\in\Lambda_1$ and $L_{A_1}(x)$, $L_{A_3}(x)$. 
\end{remark}

Recall that an integral plane $E$ (see Definition \ref{int-el}) 
of the distribution $\mcd^4_+$ is  {\it non-trivial}, if for every $j$ the restriction $dA_j|_{E}$ is not identically zero.

Let  $\mathcal I_2=\mathcal I_2(\mcd^4_+)\subset Gr_2(\mcd^4_+)$ denote the subset of integral planes; 
it is algebraic, since so is $\mcd^4_+$. 
By $\mathcal I_2^0\subset\mathcal I_2$ we denote the Zariski open  subset of the non-trivial integral planes. A point of the space 
$\mathcal I_2^0$ is a pair $(x,E)$, where $x\in\mcr^+_{0,4}$, $E\subset\mcd^4_+(x)$ is a non-trivial integral plane. 
Let $\pi_{gr}:\mathcal I_2^0\to\mcr^+_{0,4}$ denote the standard projection. The subspaces 
$$\mcf^3(x,E)=(d\pi_{gr})^{-1}(E)\subset T_{(x,E)}\mathcal I_2^0$$
form a (singular) analytic distribution $\mcf^3$ on $\mathcal I_2^0$. Set 
$$\mathcal J_2=\mathcal I_2^0\setminus\pi_{gr}^{-1}(\Lambda)\subset\mathcal I_2^0.$$ 

\begin{lemma} \label{second} The projection $\pi_{gr}:\mathcal J_2\to\mcr^+_{0,4}\setminus\Lambda$ is a regular fibration by 
holomorphic curves.  The restriction to $\mathcal J_2$ 
of the distribution $\mcf^3$  is regular and three-dimensional. For every $y\in\mathcal J_2$ the corresponding subspace $\mcf^3(y)$ 
contains a unique integral 2-plane $\wt E(y)$ of the distribution $\mcf^3$. The planes $\wt E(y)$ form a two-dimensional analytic 
distribution on $\mathcal J_2$. 
%Every second prolongation of the Pfaffian system $\mcd^{4,2}$ on $\mcr^+_{0,4}\setminus\Lambda$ is a two-dimensional distribution. 
\end{lemma}

\begin{lemma} \label{first} The restriction to $\Lambda\setminus(\Lambda_0\cap\Lambda_1)$ of the 
distribution $\mcd^4_+$ has at most one non-trivial integral 2-plane at each point. 
%In other words, the first prolongation of the latter restriction is at most two-dimensional distribution. 
\end{lemma}

The two lemmas are proved below. In their proofs we use the notations introduced in \cite[subsection 2.5.1, p.297]{gk2}\footnote{The edge lengths 
denoted by $L_j$ in \cite{gk2} are denoted here by $l_j$.}. The 
analytic extensions to $\mcr^+_{0,4}$ of the complexified 1-forms $\theta_j$, $\nu_j$ from loc. cit. 
are regular and double-valued: well-defined up to sign. 
The forms $\nu_1,\dots,\nu_4$ yield a coordinate system on the subspaces of the distribution $\mcd^4_+$, as in loc. cit. The tangent functions 
$t_j$ from loc. cit. are holomorphic single-valued on $\mcr^+_{0,4}$. Let us show this in more detail. Recall that for a real convex quadrilateral  
$A_1\dots A_4$ equipped with its exterior bisectors $L_{A_j}$ 
one has $t_j=\tan\angle(L_{A_j}^{\perp},A_jA_{j+1})$. Let $z$ be the above standard affine coordinate on $\oc_{\infty}$: 
$z(I_1)=0$, $z(I_2)=\infty$. Let $z_j$, $w_j$ denote respectively the $z$- coordinates of the intersection points $L_{A_j}^{\perp}\cap\oc_{\infty}$, 
$A_jA_{j+1}\cap\oc_{\infty}$. It follows from definition that 
$$t_j=\frac{i(z_j-w_j)}{z_j+w_j}.$$
The latter formula defines the analytic extension of the functions $t_j$ to $\mcr^+_{0,4}$. One has $z_j\neq \pm w_j$ on $\mcr^+_{0,4}$. Indeed,  
otherwise either $A_jA_{j+1}=L_{A_j}^{\perp}$, or $A_jA_{j+1}=L_{A_j}$; in both cases $A_jA_{j-1}=A_jA_{j+1}$, 
which is impossible by non-degeneracy.  Thus, {\it $t_j$ are non-vanishing holomorphic functions on $\mcr^+_{0,4}$. }

\begin{proof} {\bf of Lemma \ref{second}.} 
The real Pfaffian system $\mcd^{4,2}_{\alpha}$ corresponding to the usual real Birkhoff distribution $\mcd^4_{\alpha}$, $\alpha=(1,\dots,1)$, 
was studied by Yu.G.Kudryashov in \cite[subsection 2.5]{gk2}. The proof of the statement of the lemma for the integral planes of the system $\mcd^{4,2}_{\alpha}$ is 
presented in \cite[subsection 2.5.4, pp. 299-300]{gk2}. The calculations presented there extend analytically to all of 
$\mcr^+_{0,4}\setminus\Lambda$ and apply 
without changes. 
\end{proof}

\begin{proof} {\bf of Lemma \ref{first}.} 
%New correction -start
Kudryashov's calculations
%New correction-end
 in \cite[p.298]{gk2} extended analytically to complex domain show that for every 
$x\in\mcr^+_{0,4}$ each non-trivial integral plane $E_2\subset\mcd^4_+(x)$ 
%(if any) 
has a basis of the following type: 
$$\left(\begin{matrix} & 0 & l_1 & \eta & -l_4 \\
& l_1 & 0 & -l_2 & \eta'\end{matrix}\right); \ \eta\eta'=l_2l_4-l_1l_3;$$
the two raws represent vectors in $\mcd^4_+(x)$ written in the coordinates given by the forms $\nu_j$ from loc. cit. 
Let $x\in\Lambda\setminus(\Lambda_0\cap\Lambda_1)$. Then $\eta\eta'=l_2l_4-l_1l_3=0$. Thus, either $\eta=0$, or $\eta'=0$. 

Case 1): $\eta'=0$. Then the inclusion $E_2\subset T_x\Lambda$ implies that 
\begin{equation} t_3(l_1+l_4)\eta=l_1(t_2-t_4)(l_3+l_4),\label{eta'}\end{equation}
see \cite[p.299, subsection 2.5.3]{gk2}. 
Les us suppose the contrary: there exist at least two different non-trivial 
integral planes in $T_x\Lambda$. Or equivalently, equation (\ref{eta'}) in $\eta$ 
has more than one solution. Recall that $t_j,l_j\neq0$. Therefore, $l_4=-l_1$. Hence, $l_3=-l_2$, since $l_1l_3=l_2l_4$. 
The first equality $l_4=|A_1A_4|=-l_1=-|A_1A_2|$ implies that the points $A_2$, $A_4$ are symmetric with respect to the 
line $L_{A_1}$, by length concordance, see Definition \ref{concord}. 
%Indeed, the latter equality says that 
%the $L_{A_1}$-concordant lengths $|A_1A_4|$ and $|A_1A_2|$ are opposite. The latter is equivalent to the above symmetry, by the definition of 
%concordant lengths. 
Similarly, the points $A_2$, $A_4$ are symmetric with respect to the line $L_{A_3}$. Finally, the quadrilateral 
$A_1A_2A_3A_4$ is symmetric with respect to the line $L_{A_1}=L_{A_3}$ and hence, $x\in\Lambda_0$, -- a contradiction.

Case 2): $\eta=0$. We similarly get that $x\in\Lambda_1$, -- a contradiction. Lemma \ref{first} is proved.
\end{proof}

\begin{remark} \label{subs} As it was shown in loc.cit., an integral  plane $E_2\subset T_x\Lambda$ may 
exist only for $x$ from an algebraic subset in $\Lambda$. For example, in the case, when $\eta'=0$, it exists only if $t_1(x)=t_3(x)$. 
\end{remark}

\begin{remark} \label{type1} The distributions $\mcd^4_+|_{\Lambda_j}$, $j=0,1$, are three-dimensional. One can easily show that the restriction to 
$\Lambda_j$ of the  Pfaffian system $\mcd^{4,2}_+$ is involutive. Each its complex integral surface is an open set of quadrilateral orbits of a 
4-reflective 
complex billiard of type 1): if, say, $j=0$, then $a_1=a_3$ is a line, the curves $a_2$ and $a_4$ are symmetric with respect to it. Similar 
statement holds for smooth integral surfaces. 
\end{remark}

\begin{corollary} \label{an-surf} Let $\alpha\in\Psi^+_4$, $S\subset\rr^8\setminus\Sigma$ be a non-trivial $C^3$-smooth integral surface of the distribution $\mcd^4_{\alpha}$. Let in addition $S\cap\ja^{-1}(\Lambda_0\cap\Lambda_1)=\emptyset$. 
Then the set of analyticity points of the surface $S$ is open and dense in $S$. 
%All the second prolongations of the Pfaffian system $\mcd^{4,2}$ on the complement of the union $\Lambda_0\cup\Lambda_1$ 
%are at most two-dimensional distributions.
\end{corollary}

\begin{proof} The image $\ja(S)\subset\mcr^+_{0,4}$ is a totally real integral surface of the complex distribution $\mcd^4_+$. 
The subset $(S\setminus\ja^{-1}(\Lambda))\cup Int(S\cap\ja^{-1}(\Lambda))\subset S$ is open and dense. Hence, 
it suffices to prove the corollary in each one of the two following separate cases: $\ja(S)\cap\Lambda=\emptyset$; $
\ja(S)\subset\Lambda\setminus(\Lambda_0\cup\Lambda_1)$.

Case 1): $\ja(S)\cap\Lambda=\emptyset$. 
The complex span of each tangent plane to $\ja(S)$ is a complex integral plane of the distribution $\mcd^4_+$. 
Therefore, $\ja(S)$ lifts to a totally real $C^2$-smooth integral surface 
$\wt S\subset\mathcal J_2$ of the distribution  $\wt E$ from Lemma \ref{second}. 
Let $M\subset\mathcal J_2$ denote the minimal complex analytic subset containing $\wt S$. The restriction to $M$ of the distribution $\wt E$ is 
two-dimensional and integrable (cf. Proposition \ref{anint}). Without loss of generality we can and will assume that $\wt S$ is contained in the 
regular part of the analytic set $M$. 
%The set $M$ being complex analytic, it admits a locally finite at most countable 
%stratification so that the restriction of the distribution $\wt E$ to each stratum is an analytic distribution that is either at most one-dimensional, or  
% two-dimensional. The distribution is integrable on each stratum: the Frobenius integrability equation holds on the regular part of $M$ 
% and remains valid on the singular strata (passing to limit). The surface $\wt S$ intersects some strata by its subsets 
% with non-empty interiors, and the union of the latter interiors is dense in $\wt S$. 
%Therefore, without loss of generality it suffices to consider the case, when $\wt S$ is contained in a single 
%stratum, which will be now denoted by $M$. 
Then $\wt S$ is contained in a complex analytic integral surface of the distribution $\wt E$ on $M$ 
and coincides with its intersection with the real part of the variety $\mathcal I_2$.
%: the variety of real integral planes. 
Therefore, $\wt S$ is 
real analytic, and hence, so is $S$. 

Case 2): $\ja(S)\subset\Lambda\setminus(\Lambda_0\cap\Lambda_1)$. Let $M\subset\Lambda\setminus(\Lambda_0\cup\Lambda_1)$ 
denote the minimal complex analytic subset that contains $\ja(S)$. Recall that for every $x\in\Lambda\setminus(\Lambda_0\cup\Lambda_1)$ 
 there exists at most one non-trivial integral plane at $x$ of the 
restriction $\mcd^4_+|_{\Lambda}$ (Lemma \ref{first}). The union of the latter integral planes is an analytic subset in 
$Gr_2(T\mcrr^+_{0,4})|_{\Lambda\setminus(\Lambda_0\cup\Lambda_1)}$ with proper projection to 
$\Lambda\setminus(\Lambda_0\cup\Lambda_1)$. Its image under the latter projection contains $M$, since it contains $\ja(S)$ and 
is analytic (Proper Mapping Theorem). Therefore, for every $y\in M$ there exists a unique integral plane in $\mcd^4_+|_{\Lambda}(y)$. 
The restriction to $M$ of the singular distribution formed by the latter planes is two-dimensional and integrable 
(cf. Proposition \ref{anint}). Afterwards without loss of generality we assume that 
$\ja(S)$ lies in the regular part of the set $M$ and deduce that $S$ is analytic, as in the above case. Corollary \ref{an-surf} is proved.
\end{proof}

\begin{proof} {\bf of Theorem \ref{class}, $C^4$-smooth case.} Fix a germ of 
$C^4$-smooth 4-reflective pseudo-billiard $a_1$, $a_2$, $a_3$, 
$a_4$. It has an open set $S$ of 
quadrilateral orbits, which is a germ of $C^3$-smooth non-trivial integral surface of a real distribution $\mcd^4_{\alpha}$ 
at a point $p\in\rr^8\setminus\Sigma$. One has  $\alpha\in\Psi_4^+$, by Proposition \ref{orient}.

Case 1): the complement $S\setminus \ja^{-1}(\Lambda_0\cup\Lambda_1)$ is non-empty. 
Let us show that $S$ is analytic at $p$. This will imply that the above pseudo-billiard is 
of type either 2), or 3) (Theorem \ref{class}, analytic case). Suppose the contrary. Then the open subset of analyticity points in $S$ 
contains a connected component $W\not\subset\ja^{-1}(\Lambda_0\cup\Lambda_1)$ 
with non-empty boundary, by Corollary \ref{an-surf}. Fix a boundary 
point $x\in\partial W$ and a path $\psi\subset W$ going to $x$ such that for every $j$ one has $A_j(y)\neq A_j(x)$ for $y\in\psi$ 
arbitrarily close to $x$. The surface $S$ is not 
analytic at $x$. The subset $W$ is an open set of quadrilateral orbits of an analytic pseudo-billiard. The latter is not of type 1), since 
$W\not\subset\ja^{-1}(\Lambda_0\cup\Lambda_1)$. 
Hence, it is of type either 2) or 3): its mirrors are either all lines, or all confocal conics. Thus, the mirrors $a_j$ are analytic at $A_j(y)$, $y\in\psi$ and 
some of them, say $a_2$ is not analytic at $A_2(x)=\lim_{y\to x; y\in\psi}A_2(y)$. Let us show that this is impossible. To do this, let us fix a 
$y\in\psi$ close to $x$ with $A_1(y)\neq A_1(x)$. Consider 
the smooth deformation of a quadrilateral orbit $A_1(y)A_2(y)A_3(y)A_4(y)$ with fixed $A_1$. That is, the family of quadrilateral orbits 
$Q(s)=A_1^sA_2^sA_3^sA_4^s$ depending on the angle parameter $s=\angle A_2^sA_1^sA_4^s$,  
$$A_1^s=A_1(y), \ s_0=\angle A_2(y)A_1(y)A_4(y), \ A_j^{s_0}=A_j(y).$$
Let $(s_-,s_+)$ be the maximal interval of analyticity of the family of quadrilaterals $Q(s)$  as a function of $s$. For some of $s_\pm$, say $s_+$, 
some vertices $A_j^{s_+}$ should be singular points of the corresponding mirrors 
and  coincide with $A_j(x)$, by definition. At least two points $A_j^{s_+}$ should be singular, 
see \cite[lemma 41 and its proof, pp. 305--306]{gk2}. If two neighbor vertices are singular, e.g., $A_2^{s_+}=A_2(x)$, $A_3^{s_+}=A_3(x)$, 
then $A_1(y)=A_1(x)$: $A_1(y)$ is the point of intersection close to $A_1(x)$ of the curve $a_1$ with the line $A_1(x)A_2(x)$ 
symmetric to $A_2(x)A_3(x)$ with respect to the line $T_{A_2(x)}a_2$. This contradicts the 
assumption $A_1(y)\neq A_1(x)$. The contradiction thus obtained shows that $A_2^{s_+}=A_2(x)$ and $A_4^{s_+}=A_4(x)$ are singular. 

Let $a_j^0\subset a_j$ denote the arcs saturated by the vertices $A_j(y)$, $y\in\psi$. Recall that the arcs $a_j^0$ are analytic, and 
they are either all lines, or all confocal conics, see the above discussion.  
%they should form an analytic pseudo-billiard of type 2) or 3), for which $W$ is an open set of quadrilateral orbits. 
The above statement implies that 
for $A=A(y)\in a_1^0$ the lines $AA_2(x)$, $AA_4(x)$ are symmetric with respect to the line $T_Aa_1^0$. Hence, this is true for every 
$A\in a_1^0$, by analyticity. Therefore, 
$a_1^0$ is either the symmetry line between the points $A_2(x)$, $A_4(x)$, or a conic with 
foci at them \cite[proposition 2.32]{alg}. 
In the case, when $a_j^0$ are confocal conics, the conic $a_2^0$ would contain its own focus $A_2(x)$, -- a contradiction. In the case, when 
they are lines, they should intersect at one point (the pseudo-billiard is of type 2)).  This implies that  the lines $a_2^0$, $a_4^0$ are 
%%zdes
%The points $A_2(x)$, 
%$A_4(x)$ are boundary points of the arcs $a_2^0$ and $a_4^0$ respectively. 
%Considering the above deformation and argument for 
%all $y\in\psi$ close to $x$ yields that for every $A\in a_1^0$ the lines $AA_2(x)$, $AA_4(x)$ are symmetric with respect to the line $T_Aa_1^0$. 
%Hence,  
symmetric with respect to the line $a_1^0$, as are $A_2(x)$, $A_4(x)$, and hence, 
the pseudo-billiard is of type 1), -- a contradiction. Hence, $S$ is analytic. 
%In the latter case the curves $a_2^0$, $a_4^0$ are contained in one and the same conic 
%confocal to $a_1^0$ that contains its own foci $A_2(x)$, $A_4(x)$. This is obviously impossible. 

Case 2): $\ja(S)\subset\Lambda_j$ for some $j=0,1$. 
Then we obviously get a pseudo-billiard of type 1) (see Remark \ref{type1}).

Case 3): $\ja(S)\subset\Lambda_0\cup\Lambda_1$, $\ja(p)\in\Lambda_0\cap\Lambda_1$ and 
both complements $S\setminus\ja^{-1}(\Lambda_j)$, $j=0,1$, are non-empty. Let us show that this case is impossible. 
Suppose the contrary: the latter assumptions hold. Then $p$ corresponds to a rhombus with interior bisectors. Some of germs of mirrors 
$a_j$, $j=2,4$ at the points $A_j(p)$ is not a line, since otherwise, $a_2=a_4$ and we obviously get that 
$\ja(S)\in\Lambda_1$, -- a contradiction. The similar statement holds for the mirrors $a_1$ and $a_3$. Finally, some two germs of neighbor 
mirrors, say $a_1$ and $a_2$ are not lines. The billiard under question being 4-reflective, the surface $S$ contains a point $q$ 
corresponding to a 4-reflective orbit 
$A_1(q)A_2(q)A_3(q)A_4(q)$ with $A_1(q)\in a_1$, $A_2(q)\in a_2$ being points of non-zero curvature with non-orthogonal tangent lines. 
Hence, the latter quadrilateral is not a rhombus and thus,  
$\ja(q)$ is contained   in only one set $\Lambda_j$, $j=0,1$. Hence, the germ at $q$ of the surface $S$ should consist of quadrilateral 
orbits of a pseudo-billiard of type 1) (see Case 2)). Thus, at least one of the neighbor mirror germs $(a_1,A_1(q))$, $(a_2,A_2(q))$ 
should be a line, while they both have non-zero curvature. The contradiction thus obtained proves Theorem \ref{class}.
\end{proof}

\begin{proof} {\bf of Theorem \ref{measure}.} Suppose the contrary: there exists a  $C^4$-smooth pseudo-billiard where 
only some two neighbor mirrors, say $a_1$, $a_2$  
have skew reflection law and the set of 4-periodic orbits has positive Lebesgue measure. In more detail, 
consider the set $S$ of its 3-edge orbits $A_1A_2A_3A_4$: the reflection law is required only at $A_2$, $A_3$ but 
not necessarily at $A_1$, $A_4$. This is a two-dimensional non-trivial $C^3$-smooth surface, which contains a positive measure set of 4-periodic 
orbits. Thus, $S$ is a $C^3$-smooth pseudo-integral surface (in terms of \cite[definition 13, p.291]{gk2})
of the analytic distribution $\mcd^4_{\alpha}$, $\alpha=(-1,-1,1,1)$. Note that $\ja^{-1}(\Lambda)=\emptyset$. Indeed, for every quadrilateral 
$A_1A_2A_3A_4\in\rr^8\setminus\Sigma$ consider the concordant lengths $l_j=|A_jA_{j+1}|$ with respect to the interior bisectors at $A_1$, 
$A_2$ and the exterior bisectors at $A_3$, $A_4$. The lengths $l_j$, $j=2,3,4$ have the same signs, while $l_1$ has opposite sign, 
see Example \ref{escort}.  
Therefore, the equality $l_1l_3=l_2l_4$ defining $\Lambda$ is impossible. Hence, $S\cap\ja^{-1}(\Lambda)=\emptyset$. 
This together with Lemma \ref{second} and  \cite[theorem 28, p.295]{gk2} implies that $\mcd^4_{\alpha}$ has an 
analytic non-trivial integral surface. Therefore, there exists an analytic 4-reflective pseudo-billiard with exactly two skew reflection laws at 
some neighbor mirrors, -- a contradiction to the last statement of Theorem \ref{class}. Theorem \ref{measure} is proved.
\end{proof}

\subsection{Application 1: Tabachnikov's Commuting Billiard Conjecture}

The following theorem solves the piecewise $C^4$-smooth case of S.Tabachnikov's conjecture on commuting convex planar billiards 
\cite[p.58]{tabcom}. 
It deals with two billiards in nested convex compact domains $\Omega_1\Subset\Omega_2\Subset\rr^2$, set $a=\partial\Omega_1$, 
$b=\partial\Omega_2$. We consider that both $a$ and $b$ are piecewise $C^4$-smooth. For every $\Omega_j$ consider the corresponding 
billiard transformation acting  on the space of oriented lines in the plane. It acts as identity on the lines disjoint from $\Omega_j$. 
Each oriented line $l$ intersecting $\Omega_j$  is 
sent to its image under the reflection from the boundary $\partial\Omega_j$ at its 
last intersection point $x$ with $\partial\Omega_j$ in the sense of orientation:  the orienting arrow of the line $l$ at $x$ 
is directed outside $\Omega_j$.  The reflected line is oriented by a tangent  vector at $x$ directed inside $\Omega_j$.  This is a continuous 
dynamical system if the boundary $\partial\Omega_j$ is smooth and piecewise-continuous (measurable) otherwise. 

It is known that confocal elliptic billiards  commute \cite[p.49, corollary 4.6]{tab}. The next theorem shows that the converse is also true.

\begin{theorem} \label{tab} Let two  nested planar convex  piecewise $C^4$-smooth 
Jordan curves be such that the corresponding billiard transformations commute. 
Then they are confocal ellipses.
\end{theorem} 

%\begin{remark} \label{remcom}  The commutativity of 
%\end{remark}

In the proof of Theorem \ref{tab} we use the next commutativity criterion.
%
%
%
%\begin{proof} Let $a=\partial\Omega_1$ and $b=\partial\Omega_2$ be the curves under consideration. 
%The commutation of billiard transformations is equivalent to the 4-reflectivity of the piecewice-analytic pseudo-billiard 
%$a$, $b$, $a$, $b$, see Fig.5. This implies that appropriate collections of analytic pieces of mirrors $a$ and $b$ form a real 
%analytic pseudo-billiard, and thus, belong to the list from Theorem \ref{class}. 
%Moreover, the union of the latter collections cover all the analytic pieces of the curves $a$ and $b$. The precise 
%statement is given by the following proposition.

\begin{proposition} \label{commpseudo} Let $a,b\subset\rr^2$ be  nested convex Jordan curves, as at the  beginning of the subsection. 
The corresponding billiard transformations commute, if and only if each pair $(A,B)\in a\times b$ extends to a quadrilateral orbit $ABCD$ 
of the pseudo-billiard $a$, $b$, $a$, $b$ as at Fig.\ref{fig-ellipses}: the reflection law is usual at $b$ and skew at $a$; only one of the segments 
$AB$, $BC$ intersects the domain bounded by the curve $a$, if both ambient lines intersect it. 
%Take an arbitrary points $A\in a$ and $B,D\in b$ (in smooth pieces of the corresponding curves) 
%such that the lines $AB$ and $AD$ are symmetric with respect to the line $T_Aa$, 
%the segment $AD$ intersects the interior of the domain $\Omega_1$ and 
%the segment $AB$ doesn't. Let $C$ be the point of intersection of the lines symmetric to $AB$ and $AD$ with respect to $T_Bb$ and 
%$T_Dd$ respectively, as at Fig.6. Then the quadrilateral $ABCD$ is an interior point of the set of quadrilateral orbits of the 
%pseudo-billiard $a$, $b$, $a$, $b$ with usual reflection law at  $b$ and skew one at $a$. 
\end{proposition}

The proposition follows from definition. 

\begin{proof} {\bf of Theorem \ref{tab}.} The interior of the set of quadrilateral orbits  from the proposition contains an open and dense 
subset  of those quadrilaterals $ABCD$ for which the germs of curves $a$, $b$ at the vertices are $C^4$-smooth and 
form a 4-reflective pseudo-billiard. The latter has thus some of types 1)--3) from Theorem \ref{class}. 
It cannot be of type 1) with germs $(b,B)$, $(b,D)$ lying on the same line and $(a,A)$, $(a,C)$ being symmetric with respect to it 
(see Addendum 1). This is impossible by convexity and the obvious 
inclusion $ABCD\subset\overline{\Omega_2}$. Therefore, either $(a,A)$, $(a,C)$ lie in the same line and the pseudo-billiard is of type 1), 
or its has type 2) or 3). This implies that the curve $a$ contains an open dense union of analytic pieces, each of them being either 
a line segment, or an arc of  conic. In what follows the reflections from the curves $a$, $b$ acting on the space of oriented lines will be 
denoted by $\sigma_a$ and $\sigma_b$ respectively. 

Case 1): some analytic arc $a'$ of the curve $a$ is a conic. Then  the above germs of pseudo-billiards with $A\in a'$ 
have type 3). Thus, the curve $b$ contains an open dense union of analytic arcs lying in conics confocal to $a'$. 

Subcase 1a): the curve  $a=a'$  is analytic, and hence, is an ellipse. 
Note that the tangent lines to $b$ are disjoint from $a$ and hence, from the focal segment, by convexity. 
This implies that each conical piece of the curve $b$ is not a hyperbola, and thus, is 
an ellipse. Recall that $b$ is a finite union of $C^4$-smooth arcs. Each smooth arc $b'$ is elliptic. Indeed, consider an auxiliary function 
on $\rr^2$: the sum of distances of a variable point to the foci of $a'$. It  is constant on $b'$: it is smooth on $b'$ 
and locally constant on an open dense union of confocal elliptic subarcs; 
hence it has zero derivative on $b'$, and $b'$ is elliptic.  
Thus, $b$ is  a finite chain of adjacent confocal elliptic arcs. Therefore $b$ is an ellipse: the 
above sum of distances is continuous on $b$ and constant on each arc, and hence, on all of $b$. 

Subcase 1b): some boundary point $A$ of the analytic arc $a'$ is singular, the curve $a$ is not analytic at $A$. 
Fix $A$ and a subarc $b'$ of a conical arc of the curve $b$ such that for every $B\in b'$ the interval $(A,B)$ intersects $a$, the line 
$AB$ is distinct from the lines $T_Aa'$, $T_A^{\perp}a'$, and the line $\Lambda=\Lambda(B)$ symmetric to $AB$ with respect to $T_Aa$ 
intersects the  curve $b$  at its  analyticity points. Let us orient the line $L=AB$ from $B$ to $A$. Consider the one-parametric 
analytic family of lines  $L^*=L^*(B)=\sigma_b^{-1}(L)$. 
We orient the line $\Lambda$ so that $\Lambda=\sigma_a(L)$. We claim that all the lines $L^*$ pass through the 
same point of the curve $a$. Indeed,  the mapping $\sigma_a\circ\sigma_b=\sigma_b\circ\sigma_a$ is singular at each $L^*$, since 
$\Lambda=\sigma_a(L)=\sigma_a\circ\sigma_b(L^*)$, $\sigma_b$ is regular (that is, a local analytic diffeomorphism) at $L^*$ and 
$\sigma_a$ is singular at 
$L=\sigma_b(L^*)$, as is $A$. 
 On the other hand, $\sigma_b$ is regular at $\sigma_a(L^*)$, since $\Lambda=\sigma_b\circ\sigma_a(L^*)$ and 
 $\sigma_b^{-1}$ is a local analytic diffeomorphism 
 at $\Lambda$: the points of intersection $\Lambda\cap b$ are regular and transversal (by convexity). 
Therefore, $\sigma_a$ is singular at  $L^*$, as is $\sigma_b\circ\sigma_a$. 
Hence, either all the lines $L^*$ 
pass through the same singular point $A'$ of the curve $a$ (the set of its singular points is totally disconnected), 
or they are tangent to $a$. In the latter case 
one has $\sigma_b\circ\sigma_a(L^*)=\sigma_b(L^*)=L$, while $\sigma_a\circ\sigma_b(L^*)=\sigma_a(L)\neq L$, since $L\neq T_Aa',(T_Aa')^{\perp}$, 
-- a contradiction to commutativity. Hence, the  tangency case is impossible, and the lines $L^*$ pass through the same point 
$A'$. Therefore, $b'$ is a conical arc confocal to $a'$ with foci $A$, $A'$. 
Thus, the conic $a'$ contains its own focus $A$, -- a contradiction. 
%
% is a focus of each analytic arc of the curve $b$. 
%Fix a conical analytic piece $a'\subset a$ and its 
%boundary point  $A$, which is singular. The set of quadrilateral orbits from Proposition \ref{commpseudo} contains a 
%one-parameter family of quadrilaterals  $ABCD$ with fixed $A$ and variable $B$, $D$ lying in small conical arcs of the curve $b$. The 
%vertex $D=D(B)$ is an analytic function in variable $B$. We claim that $C=C(B)\equiv const$. Indeed, suppose the contrary: $C(B)\not\equiv const$. 
%The vertex $C=C(B)$ is a non-constant analytic function
%in  $B$, being the intersection point of two analytic families of lines: the line $BC$ symmetric to $AB$ with respect to the line $T_Bb$; 
%the line $DC$ symmetric to $AD$ with respect to $T_Db$. Hence, there exists a germ of 4-reflective pseudo-billiard 
%$(a,A)$, $(b,B)$, $(a,C)$, $(b,D)$ where all the mirror germs are analytic, except for just one: the germ $(a,A)$.  
%This is impossible, as in  \cite[proof of proposition 2.16]{alg}. Thus, $C(B)\equiv const$, hence $B$ ranges in a conical arc $b'$ of the curve $b$ with foci at $A$, $C$, by \cite[proposition 2.32]{alg}. Finally, the conic containing $a'$ is confocal to $b'$ and passes through its own focus $A$, -- 
%a contradiction. 

Case 2): the curve $a$ is a convex polygon. For every its vertex $A$  each line through $A$ intersecting a $C^4$-smooth arc 
$b'$ of the curve $b$ is reflected from $b'$  
to a line through another vertex  $A'$ of the polygon $a$, as in the above discussion. This implies that each $C^4$-smooth 
arc $b'$ is a conic with foci $A$, $A'$. Finally,  each vertex of the polygon $a$ is a focus of each $C^4$-smooth arc of the curve $b$. 
Therefore, $a$ has at most two vertices and cannot be a convex polygon. The contradiction thus obtained proves Theorem \ref{tab}.
\end{proof}

%For every its edge $E$  the collection of non-linear $C^4$-smooth pieces of the curve $b$ is 
%symmetric with respect to $E$. Indeed, for every $A\in E$ the nonlinear $C^4$-smooth part of the curve 
%$b$ contains an open and dense subset of those points $B$ for which the pair $(E,A)$, $(b,B)$ extends to a germ of $C^4$-smooth 
%4-reflective pseudo-billiard, as at the beginning of the proof of Theorem \ref{tab}. The latter pseudo-billiard should have type 1) with 
%germs $(b,B)$, $(b,D)$ being symmetric with respect to the edge $E$ (Theorem \ref{class}). This together with Addendum 1 
%implies the above symmetry statement.  The complement (if non-empty) of the curve $b$ to the closure of its 
%non-linear $C^4$-smooth part consists of  straight pieces connecting boundary points of the latter closure.
%The latter straight pieces should be obviously symmetric as well. Finally, the whole curve $b$ is symmetric with respect to every edge $E$. Hence, 
%all the edges of the convex polygon $a$  intersect at the same point: the center of mass of the domain bounded by $b$. This is obviously 
%impossible.  Theorem \ref{tab} is proved.

 \subsection{Application 2: planar  Plakhov's Invisibility Conjecture with four reflections} 
 
 This subsection is devoted to Plakhov's Invisibility Conjecture: the analogue of Ivrii's conjecture in the invisibility theory 
\cite[conjecture 8.2]{pl}. We recall it below and show that in planar case it follows from a conjecture saying that no finite collection of germs of 
smooth curves can form a measure $k$-reflective billiard for appropriate ``invisibility'' reflection  law. Both Plakhov's and 
Ivrii's conjectures have the same complexification, see 
 \cite[subsection 5.2, proposition 8]{odd}. We prove the $C^4$-smooth case of  planar Plakhov's Invisibility Conjecture for four reflections as an immediate 
corollary of Theorem \ref{measure}.

\begin{definition} \cite[chapter 8]{pl} 
Consider a perfectly reflecting (may be disconnected) closed bounded body $B$ in a Euclidean space.  For every oriented line 
(ray) $R$ 
take its first intersection point $A_1$ with the boundary $\partial B$ and reflect $R$ from the tangent hyperplane $T_{A_1}\partial B$. The 
reflected ray goes from the point $A_1$ and defines a new oriented line, as in  the previous subsection. 
Then we repeat this procedure. Let us assume that after a finite 
number $k$ of reflections the output oriented line coincides with the input line $R$ and will not hit the body any more. 
Then we say that the body $B$ {\it is invisible for the ray 
$R$,} see Fig.\ref{fig-inv1}. We call $R$ a {\it ray of invisibility} with $k$ reflections.  
%The finite piecewise-linear curve  formed by reflected edges and 
%bounded by the first 
%and last reflection points will be called its  {\it complete trajectory}.  
 \end{definition}

 \begin{figure}[ht]
  \begin{center}
   \epsfig{file=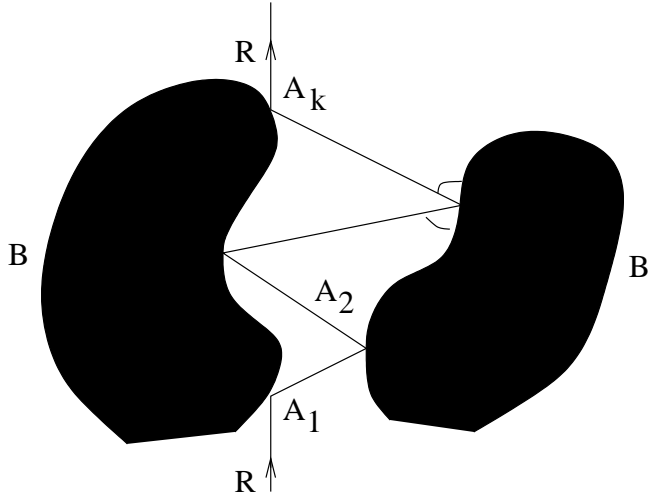}
    \caption{A body invisible for one ray.}
    \label{fig-inv1}
  \end{center}
\end{figure}

{\bf Invisibility Conjecture} (A.Plakhov, \cite[conjecture 8.2, p.274]{pl}.) 
{\it There exist no body with piecewise $C^{\infty}$ boundary for which the set of rays of invisibility  has positive measure.}

\begin{remark} As is shown by A.Plakhov in his book \cite[section 8]{pl}, there exist no body invisible for all rays. The same book 
contains a very nice survey on invisibility, including examples of bodies invisible in a finite number of (one-dimensional families of) 
rays. See also papers \cite{ply, pl2, pl3, pl4} for more results. 
The Invisibility Conjecture is equivalent to the statement saying that  for every $k\in\mathbb N$ there are no measure 
$k$-reflective bodies, see the next definition. It is open even in dimension 2.  
\end{remark}

\begin{definition} (cf. \cite[subsection 5.2, definition 12]{odd})
A body $B$  with piecewise-smooth boundary is called {\it measure $k$-reflective,} if the set of invisibility rays 
with $k$ reflections has positive measure. 
\end{definition}

\begin{definition} (cf. \cite[subsection 5.2, definition 13]{odd}) A (germ of) real planar smooth pseudo-billiard 
 $a_1,\dots,a_k$ is called {\it measure $k$-invisible},  if it is measure $k$-reflective for skew reflection law at $a_1$, $a_2$ and 
 usual  law at the other mirrors $a_j$: the set of its $k$-periodic orbits for the above reflection law (called {\it $k$-invisible} orbits, 
 see Fig.\ref{fig-inv2}) 
 has positive Lebesgue measure. 
\end{definition}

 \begin{figure}[ht]
  \begin{center}
   \epsfig{file=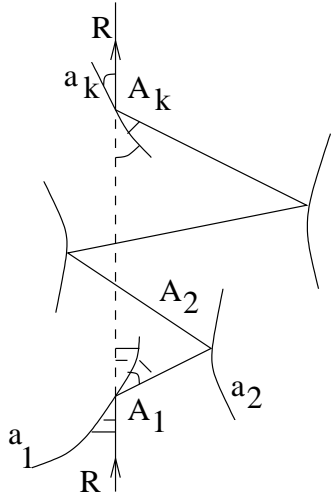}
    \caption{A $k$-invisible $k$-gon: skew reflection law at $A_1$ and $A_k$.}
    \label{fig-inv2}
  \end{center}
\end{figure}

\begin{proposition} \label{inv-ivr} Let $k\in\mathbb N$ and $B\subset\rr^2$ be a body with piecewise-smooth  
boundary,  and  no collection of $k$ germs of its boundary 
 form a measure $k$-invisible smooth pseudo-billiard. Then $B$ is not measure $k$-reflective. 
 \end{proposition}

Proposition \ref{inv-ivr} is implicitly contained in \cite[section 8]{pl}. 

 \begin{theorem} \label{tinvis} There are no measure 4-reflective bodies in $\rr^2$ with piecewise $C^4$-smooth boundary. 
 \end{theorem}
 
 \begin{proof} The existence of a measure 4-reflective body as above implies the existence of a measure 4-invisible planar $C^4$-smooth 
 pseudo-billiard (Proposition \ref{inv-ivr}). This is 
 a measure 4-reflective planar $C^4$-smooth pseudo-billiard with skew reflection law at some pair of neighbor 
 vertices and usual reflection law at the other vertices. This contradicts Theorem \ref{measure}.
 % and proves Theorem \ref{tinvis}.
 \end{proof}

 \section{General case of complex $k$-reflective billiards: state of art}
 
 First let us  recall the next conjecture and partial positive results from \cite{odd}.
 
 {\bf Conjecture A \cite[p.295]{odd}.} {\it There are no $k$-reflective complex analytic (algebraic) planar billiards for odd $k$.}
  
  \begin{theorem} \cite[p.295]{odd}. There are no 3-reflective complex analytic planar billiards.
  \end{theorem}
  
  \begin{theorem}  \cite[p.295]{odd}. For every odd $k$ there are no $k$-reflective complex algebraic planar billiards whose  mirrors 
avoid isotropic points at infinity.
\end{theorem}

{\bf Conjecture B.} {\it For every $k\geq3$ there are no $k$-reflective complex analytic planar billiards $a,\dots,a$ with all the mirrors 
coinciding with the same irreducible analytic curve $a\subset\cp^2$.}

Recently a positive result was proved by the author (paper in preparation) for every irreducible algebraic curve $a$ that either is smooth, or 
 satisfies a mild condition on either singularities, or tangential correspondence. 
 
\begin{definition} \label{defcom}  A {\it combination} of complex analytic  
billiards $\alpha=(a_1,\dots,a_l)$, $\beta=(b_1,\dots,b_m)$ is a billiard $\alpha\circ_s\beta=(a_1,\dots,a_s,b_1,\dots,b_m,a_{s+1}, 
\dots,a_l)$, $s\in\{1,\dots,l\}$. For every collection 
of analytic curves $\delta=(d_1,\dots,d_t)$ in $\cp^2$ distinct from isotropic lines 
the billiard 
$$\alpha\circ_{s,\delta}\beta=(a_1,\dots,a_s,d_1,\dots,d_t,b_1,\dots,b_m,d_t,\dots,d_1,a_{s+1},\dots,a_l)$$
 will be called 
a {\it combination with mirror adding} of the billiards $\alpha$ and $\beta$. (The previous combination corresponds to $\delta=\emptyset$.) 
\end{definition}

\begin{definition} Let $\alpha=(a_1,\dots,a_l)$, $\beta=(b_1,\dots,b_m)$ be complex planar billiards such that for some 
$1\leq s< \min\{ l, m\}$, one has 
$a_j=b_{m-j+1}$ for $j=1,\dots,s$.  
Then the billiard $\alpha\circ_{[1,s]}\beta=(a_{s+1},\dots,a_{l}, b_{1},\dots,b_{m-s})$ is called a {\it combination with 
mirror erasing} of the billiards $\alpha$, $\beta$. 
\end{definition}

\begin{remark} For every $l$- and $m$-reflective  billiards $\alpha$,  $\beta$ and $s$, $t$, $\delta$ as in Definition \ref{defcom} 
 the billiard $\alpha\circ_{s,\delta}\beta$ is  $(l+m+2t)$-reflective, provided that 

{\bf (E)} it has a periodic orbit 
%$(l+m+2t)$-gonal orbit 
$A_1\dots A_s D_1\dots D_tB_1 \dots B_mD_t\dots D_1 A_{s+1}\dots A_{l}$ 
 for which $A_1\dots A_l$ and $B_1\dots B_m$ are respectively $l$- and $m$-reflective orbits of the billiards $\alpha$ and $\beta$: 
 interior points of the set of $l$- ($m$-) periodic orbits. 

%For example, condition (E) holds automatically, whenever $\alpha$ and $\beta$ are algebraic $l$- and $m$-reflective billiards and 
%$a_s\neq b_1$, $b_m\neq a_{s+1}$. 
For every $l$- and $m$-reflective billiards $\alpha=(a_1,\dots,a_l)$, $\beta=(b_1,\dots,b_m)$ and $1\leq s<\min\{ l,m\}$ the billiard $\alpha\circ_{[1,s]}\beta$ is $(l+m-2s)$-reflective, if  

{\bf (E')} there exist $l$- and $m$-reflective orbits $A_1\dots A_l$, $B_1\dots B_m$ of billiards $\alpha$ and $\beta$ respectively such that 
$A_j=B_{m-j+1} \text{ for } 1\leq j\leq s$ and $A_{s+1}\dots A_{l}B_{1}\dots B_{m-s}$ is a periodic orbit of the combination 
$\alpha\circ_{[1,s]}\beta$.
\end{remark}

{\bf Known $k$-reflective complex analytic planar billiards.}

%I. $a$, $b_1,\dots,b_{l-1}$, $a$, $b_{l-1}^*,\dots,b_1^*$, where $a$ is a line, $b_j$, $b_j^*$ are symmetric 
%with respect to it, and there exists at least one $2l$-periodic orbit symmetric with respect to the line $a$: in more detail, it is required that 
%the corresponding branches of the mirrors $b_j$, $b_j^*$ at vertices be symmetric.

I. 4-reflective billiards of types 1)--3) from Theorem \ref{an-class}.

II. Billiards that are obtained from  them by subsequent combinations (with mirror adding or erasing); 
each subsequent combination should satisfy the above condition (E) (respectively,  (E')).

\begin{example} The billiards of type II include the following ones:

- Every billiard $\alpha_l=(a, b_{l-1},\dots,b_1, a, b_1^*,\dots,b_{l-1}^*)$, where $a$ is a line, $b_j$, $b_j^*$ are symmetric 
with respect to the line $a$, that has at least one symmetric $2l$-periodic orbit. 
%in more detail, it is required that 
%the corresponding branches of the mirrors $b_j$, $b_j^*$ at vertices be symmetric. 
Its $2l$-reflectivity is obvious. It is  obtained 
by subsequent combinations with mirror erasing of 4-reflective billiards of type 1): on each step we combine billiards $\alpha_j$ and 
 $\beta_j=(b_j^*, a, b_j, a)$ erasing one mirror a, which is the last mirror $a$ in $\beta_j$ identified 
with the first one in $\alpha_j$. 

- Every billiard formed by an even number of complex confocal conics,  
some of them coincide and each conic is taken even number of times; no two neighbor mirrors coincide. It  
is obtained by subsequent combinations (usual ones and those with mirror erasing) of 4-reflective billiards of type 3). 

- Every billiard formed by an even number of non-isotropic complex 
lines such that the product of the corresponding symmetries is the identity. It is 
obtained by subsequent combinations (usual ones and those with mirror erasing) of 4-reflective billiards of type 2). This easily 
follows from the complexification of \cite[theorem 1.B, p.3]{vi}. 
\end{example}

The next small technical  results, which generalize  Corollary \ref{cepi}, might be useful in studying the general case. They 
 are immediate consequence of the results of Section 3.  To state them, consider a complex 
 $k$-reflective billiard $a_1,\dots,a_k$. For every 
 $j=1,\dots,k$ let us introduce the corresponding space of  $(k-2)$- orbits:  collections 
 $A_1\dots A_{j-1} A_{j+2}\dots A_k\in\prod_{i\neq j, j+1}\ha_i$, $A_i\in\ha_i$, 
 such that for every $i\neq j-1, j, j+1, j+2$ one has $A_i\neq A_{i\pm1}$, 
the lines $A_iA_{i\pm1}$ are symmetric with respect to the line $T_{A_i}a_i$ and the three latter lines are distinct and non-isotropic. 
(In the case, when $j=k$, we replace $j+s$ by $s$, $s=1,2$.) The  
closure of set of the latter $(k-2)$- orbits  is an analytic subset $V_j\subset\prod_{i\neq j,j+1}\ha_i$ 
that has only two-dimensional irreducible components. 
Let $U\subset\ha_1\times\dots\times\ha_k$ denote the $k$-reflective set, see Subsection 2.3. 
For every $j=1,\dots,k$ let $P_j$ denote the product projection  
 $$P_j:\ha_1\times\dots\times\ha_k\to\prod_{r\neq j, j+1}\ha_r, \ U_j=P_j(U)\subset V_j.$$

\begin{theorem} \label{kmer} 
Let $a_1,\dots,a_k$ be a $k$-reflective complex analytic planar billiard. Let $U\subset\ha_1\times\dots\times \ha_k$ be its 
$k$-reflective set.  For every $j=1,\dots,k$ the set $U_j=P_j(U)$ is analytic:  a union of  irreducible components of the set $V_j$.
 The projection $P_j:U\to U_j$ is proper and  bimeromorphic.
 \end{theorem}

\begin{proof} Without loss of generality we prove the theorem for the pair of neighbor indices 1, $k$: $j=k$.  

{\bf 1) Properness and analyticity.} Let us show that the mapping $P_k:U\to V_k$ is proper: then $U_k$ is analytic by 
Proper Mapping Theorem.
%the analyticity of its image $U_k$ 
%then follows by Remmert's 

1a)  Case, when  $a_1$, $a_k$ are algebraic:  the above statements are obvious.

1b) Case, when some of them, say $a_k$ is not algebraic.  The  properness will be deduced from Corollary \ref{cnonint}. To do this, 
we consider the space 
$$\mcp_{1,k}=\ha_2\times\dots\ha_{k-1}\times\mcp^2$$
equipped with the distribution 
$$\mch_{1,k}=T\ha_2\oplus\dots\oplus T\ha_{k-1}\oplus\mch\oplus\mch,$$
where $\mcp=\mathbb P(T\cp^2)$, $\mch$ is the standard contact plane field on $\mcp$, see Subsection 2.7. A point of the space 
$\mcp_{1,k}$ is a triple $A_2\dots A_{k-1}, (A_1,L_{A_1}), (A_k,L_{A_k})$, where $L_{A_j}\subset T_{A_j}\cp^2$ is a one-dimensional 
subspace. Let $\mcr_{1,k}\subset\mcp_{1,k}$ 
denote the analytic variety defined by the conditions that for every $j=1,\dots,k$ one has $A_j\in\cc^2=\cp^2\setminus\oc_{\infty}$, 
$A_j\neq A_{j\pm1}$, the lines $A_jA_{j\pm1}$ are 
symmetric with respect to the line $T_{A_j}a_j$ if $2\leq j\leq k-1$ (the line $L_{A_j}$ if $j\in\{1,k\}$), the three lines $A_jA_{j\pm1}$, 
$T_{A_j}a_j$ (respectively, $L_{A_j}$) are distinct 
and non-isotropic. In addition, it is required that $A_j$ be not cusps of the curves  $a_j$ for $j\neq1,k$. The set $\mcr_{1,k}$ is an analytically constructible smooth variety, and its closure $\overline{\mcr_{1,k}}$ is an analytic set. The distribution $\mch_{1,k}$ 
induces a two-dimensional analytic distribution $\mcd_{1,k}$ on $\mcr_{1,k}$ that extends to a singular analytic distribution 
on $\overline{\mcr_{1,k}}$. An open dense subset $U_1\subset U$  
lifts to the union of non-trivial integral surfaces of the distribution $\mcd_{1,k}$, as in Proposition \ref{birktang}. 
Let $M\subset\mcp_{1,k}$ denote the minimal analytic subset containing all the non-trivial integral surfaces. 

Suppose the contrary: the mapping $P_k:U\to V_k$ is not proper. Then  $dim M\geq3$: if $dim M=2$, then $P_k$ is 
proper, as in the proof of Corollary \ref{cepi}.  In what follows we fix some at least three-dimensional irreducible component 
of the set $M$ and denote  by $M$ the latter component. The restriction $\mcd_M$  
 to $M$ of the distribution $\mcd_{1,k}$ is two-dimensional and integrable, 
by Proposition \ref{anint}. As in Subsection 3.1,  there exists an analytic subset $\Sigma\subset M$ with dense complement 
$M^0=M\setminus\Sigma\subset M$ such that $\mcd_M$ is analytic on $M^0$ and its integral surface through each
$x\in M^0$ represents an open set of $k$-periodic orbits of a $k$-reflective billiard $a_1(x)$, $a_2,\dots,a_{k-1}$, $a_k(x)$. 
Some integral surface, which we will denote by $S$, represents a family of $k$-periodic orbits of the initial billiard. 
The projection $\mu_{1,k}:M\to\ha_2\times\dots\times\ha_{k-1}$ is proper, and the image $N=\mu_{1,k}(M^0)\subset V_k$ is a 
purely two-dimensional analytically constructible subset, by Chevalley--Remmert Theorem.

{\bf Claim 1.} {\it There exists an open subset of points $x\in M^0$ such that the billiard $a_1$, $a_1(x)$, $a_k(x)$, $a_k$ is 4-reflective.}

\begin{proof} Fix  $x_1\in S$ and $x_2\in M^0\setminus\{ x_1\}$ with $p=\mu_{1,k}(x_1)=\mu_{1,k}(x_2)$. 
There exist neighborhoods $Y=Y(p)\subset N$, $X_1=X_1(p)\subset S$, $X_2=X_2(x_2)\subset M^0\setminus X_1$ 
such that $\mu_{1,k}$ projects 
$X_1$ and each integral surface of the distribution $\mcd_M|_{X_2}$ diffeomorphically onto $Y$ (let us fix them).  Fix an arbitrary 
$x\in X_2$,  let $\wt S(x)$ denote the integral surface of the distribution $\mcd_M|_{X_2}$ through $x$. 
Each $y=A_2\dots A_{k-1}\in Y$ lifts to two points in $X_1$ and $\wt S(x)$, 
which correspond to $k$-periodic orbits $A_1\dots A_k$ and $A_1'A_2\dots A_{k-1}A_k'$ of the billiards $a_1,\dots,a_k$ and $a_1(x),a_2,\dots,a_{k-1}, 
 a_k(x)$ respectively. The quadrilaterals $A_1A_1'A_k'A_k$ corresponding to generic $y\in Y$ form a two-parametric 
  family of 4-periodic orbits of the billiard $a_1,a_1(x),a_k(x),a_k$, and the latter is 4-reflective, as in the proof of Proposition \ref{bil4refl}.
\end{proof}

The above claim yields at least one-dimensional family of 4-reflective billiards with two fixed neighbor 
mirrors $a_1$, $a_k$, since $dim M\geq3$. The 
curves $a_1$, $a_k$  are not both algebraic, by assumption. This contradicts Corollary \ref{cnonint} and thus, 
proves properness of the projection $P_k:U\to V_k$.

{\bf 2) Bimeromorphicity.} Suppose the contrary: the proper analytic set projection $P_k:U\to U_k=P_k(U)$ is not bimeromorphic. 
This means that its inverse has at least two distinct holomorphic 
branches on an open subset in the  analytic set $U_k$. In other words, each $A_2\dots A_{k-1}$ from an open subset in $V_k$ 
extends to two distinct $k$-periodic orbits $A_1A_2\dots A_{k-1}A_k$, $A_1'A_2\dots A_{k-1}A_k'$. Then the two-dimensional 
family of quadrilaterals $A_1A_1'A_k'A_k$ are 4-periodic orbits of the billiard $a_1$, $a_1$, $a_k$, $a_k$ with coinciding neighbor mirrors, 
as in \cite[proof of lemma 3.1]{alg}. Hence, the latter billiard is 4-reflective, -- a contradiction to Corollary \ref{inters} (or Theorem \ref{an-class}) 
forbidding 4-reflective billiards with coinciding neighbor mirrors. Theorem \ref{kmer} is proved.
\end{proof}

\begin{corollary} \label{ckhmer} Let $a_1,\dots,a_k$ be a $k$-reflective complex analytic planar billiard. The subsequent 
$(k-2)$- orbit correspondence 
$P_{j+1}\circ P_j^{-1}:U_j\to U_{j+1}$, $A_1\dots A_{j-1}A_{j+2}\dots A_k\mapsto A_1\dots A_jA_{j+3}\dots A_k$ is bimeromorphic. 
Its graph is projected epimorphically onto both $U_j$ and $U_{j+1}$. If it contracts a curve, then the latter is compact and the mirror 
$a_{j+2}$ is algebraic.
\end{corollary}

\section{Acknowledgements} 
I am grateful to Yu.S.Ilyashenko, Yu.G.Kudryashov, A.Yu.Plakhov and S.L.Tabachnikov 
 for attracting my attention to Ivrii's Conjecture, invisibility and Commuting Billiard Conjecture. 
 I wish to thank Yu.G.Kudryashov, to whom this paper owes much: some his arguments from 
 our joint paper \cite{gk2} were used here in a crucial way. I  am grateful to them 
 and to  E.M.Chirka, E.Ghys, A.L.Gorodentsev, Z.Hajto,  B.Sevennec, V.V.Shevchishin, J.-C.Sikorav, 
 M.S.Verbitsky, O.Ya.Viro, V.Zharnitsky  for helpful discussions.

\end{document}